\documentclass[12 pt]{amsart}

\usepackage{fullpage} 
\usepackage{diagbox}

\usepackage{hyperref}
\usepackage{etex}
\usepackage[shortlabels]{enumitem}
\usepackage{amsmath}
\usepackage{amsxtra}
\usepackage{amscd}
\usepackage{amsthm}
\usepackage{amsfonts}
\usepackage{amssymb}
\usepackage{eucal}
\usepackage{cleveref}
\usepackage[all]{xy}
\usepackage{graphicx}
\usepackage{tikz}
\usetikzlibrary{cd}
\usetikzlibrary{fit, patterns}
\usepackage{mathrsfs}
\usepackage{subfiles}
\usepackage{mathpazo}
\usepackage[colorinlistoftodos, textsize=tiny]{todonotes}
\usepackage{morefloats}
\usepackage{pdfpages}
\usepackage{thmtools}
\usepackage{thm-restate}
\usepackage[utf8]{inputenc}
\usepackage[T1]{fontenc}
\usepackage{epigraph}
\usepackage{csquotes}
\usepackage{moreverb}
\usepackage[margin=1in]{geometry}
\usepackage{adjustbox}
\usepackage{accents}
\usepackage{xparse}

\graphicspath{ {images/} }

\RequirePackage{color}
\definecolor{myred}{rgb}{0.75,0,0}
\definecolor{mygreen}{rgb}{0,0.5,0}
\definecolor{myblue}{rgb}{0,0,0.65}

\usepackage{hyperref}
\hypersetup{citecolor=blue}
\usepackage{tikz}
\usetikzlibrary{matrix,arrows,decorations.pathmorphing}

\theoremstyle{plain}
\newtheorem{theorem}[subsubsection]{Theorem}
\newtheorem{construction}[subsubsection]{Construction}
\newtheorem{proposition}[theorem]{Proposition}
\newtheorem{lemma}[theorem]{Lemma}

\theoremstyle{definition}
\newtheorem{definition}[theorem]{Definition}
\newtheorem{remark}[theorem]{Remark}
\newtheorem{example}[theorem]{Example}

\newtheorem{question}[theorem]{Question}
\newtheorem{conjecture}[theorem]{Conjecture} 

\newtheorem{warn}[theorem]{Warning}

\newtheorem{notation}[theorem]{Notation}

\theoremstyle{remark}
\numberwithin{equation}{section}

\newcommand\nc{\newcommand}
\nc\on{\operatorname}
\nc\renc{\renewcommand}

\newcommand*{\shom}{\mathscr{H}\kern -.5pt om}
\newcommand*{\stor}{\mathscr{T}\kern -.5pt or}
\newcommand*{\sext}{\mathscr{E}\kern -.5pt xt}

\makeatletter
\providecommand\@dotsep{5}
\renewcommand{\listoftodos}[1][\@todonotes@todolistname]{%
\@starttoc{tdo}{#1}}
\makeatother

\makeatletter
\def\Ddots{\mathinner{\mkern1mu\raise\p@
		\vbox{\kern7\p@\hbox{.}}\mkern2mu
		\raise4\p@\hbox{.}\mkern2mu\raise7\p@\hbox{.}\mkern1mu}}
\makeatother

\newtoggle{draft}
\togglefalse{draft}

\iftoggle{draft}{
	\usepackage[notcite]{showkeys}

	\newcommand{\NB}[1]{\todo[color=gray!40]{#1}}
	\newcommand{\TODO}[1]{\todo[color=red]{#1}}
}{ 
	\newcommand{\NB}[1]{}
	\newcommand{\TODO}[1]{}
	\renewcommand{\todo}[1]{}
	\renewcommand{\todo}[1]{}
}

\newcommand{\customlabel}[2]{\protected@write \@auxout {}{\string \newlabel {#1}{{#2}{\thepage}{#2}{#1}{}} }\hypertarget{#1}{#2}}

\DeclareMathOperator\id{id}

\DeclareMathOperator\Alg{Alg}

\DeclareMathOperator\unit{\mathbb{1}}

\DeclareMathOperator\pic{Pic}

\DeclareMathOperator\spec{Spec}

\DeclareMathOperator\Mod{Mod}

\DeclareMathOperator\im{im}

\DeclareMathOperator\pgl{PGL}

\DeclareMathOperator\ab{ab}
\DeclareMathOperator\std{std}
\DeclareMathOperator\CHur{CHur}

\DeclareMathOperator\gal{Gal}

\DeclareMathOperator\tr{tr}

\DeclareMathOperator\Hur{Hur}

\DeclareMathOperator\ord{ord}

\DeclareMathOperator\inv{inv}

\DeclareMathOperator\invc{inv}

\DeclareMathOperator\aut{Aut}

\newcommand\ZZ{\mathbb{Z}}
\newcommand\RR{\mathbb{R}}
\newcommand\NN{\mathbb{N}}
\newcommand\FF{\mathbb{F}}
\newcommand\EE{\mathbb{E}}

\DeclareMathOperator\fib{fib}

\DeclareMathOperator\Tot{Tot}

\DeclareMathOperator\conf{Conf}

\DeclareMathOperator\colim{colim}

\DeclareFontFamily{U}{wncy}{}
\DeclareFontShape{U}{wncy}{m}{n}{<->wncyr10}{}
\DeclareSymbolFont{mcy}{U}{wncy}{m}{n}
\DeclareMathSymbol{\Sha}{\mathord}{mcy}{"58}

\newcommand{\phur}[4]{\operatorname{Hur}^{#1,#3}_{#2,#4}}
\newcommand{\cphur}[4]{\operatorname{CHur}^{#1,#3}_{#2, #4}}

\newcommand{\quohur}[5]{[\operatorname{Hur}^{#1,#4}_{#3,#5}/#2]}

\newcommand{\cgquohur}[5]{[\operatorname{C^{#2}Hur}^{#1,#4}_{#3,#5}/#2]}
\newcommand{\cquohur}[5]{[\operatorname{CHur}^{#1,#4}_{#3,#5}/#2]}

\newcommand{\phurc}[2]{\operatorname{Hur}^{#2}_{#1}}
\newcommand{\cphurc}[2]{\operatorname{CHur}^{#2}_{#1}}

\DeclareMathOperator\surj{Surj}

\newcommand{\hura}[4]{[\operatorname{Hur}^{#1, #3, \partial \in #4}_{#2}/#1]}
\newcommand{\hurp}[3]{[\operatorname{Hur}^{#1, #3}_{\mathbb P^1, #2}/#1]}
\newcommand{\hurz}[3]{[[\operatorname{Hur}^{#1, #3}_{\mathbb P^1, #2}/#1]/\pgl_2]}
\newcommand{\chura}[4]{[\operatorname{CHur}^{#1, #3, \partial \in #4}_{#2}/#1]}
\newcommand{\churp}[3]{[\operatorname{CHur}^{#1, #3}_{\mathbb P^1, #2}/#1]}
\newcommand{\churz}[3]{[[\operatorname{CHur}^{#1, #3}_{\mathbb P^1, #2}/#1]/\pgl_2]}

\newcommand{\churzb}[3]{\overline{\mathscr H}^{#1,#3}_{\mathbb P^1,
#2}}

\def\listtodoname{List of Todos}
\def\listoftodos{\@starttoc{tdo}\listtodoname}
\title{Homological stability for Hurwitz spaces and applications}

\author{Aaron Landesman}
\author{Ishan Levy}

\usepackage{microtype}
\begin{document}

\begin{abstract}
We show the homology of the Hurwitz space associated to an arbitrary finite rack stabilizes integrally in a suitable sense. We also compute the dominant part of its stable homology after inverting finitely many primes. This proves a conjecture of Ellenberg--Venkatesh--Westerland and improves upon our previous results for non-splitting racks. We obtain applications to Malle's conjecture, the Picard rank conjecture, and the Cohen--Lenstra--Martinet heuristics.
\end{abstract}

\maketitle
\tableofcontents

\section{Introduction}

In this paper, we show Hurwitz spaces parameterizing connected $G$-covers of $\mathbb A^1$ with specified monodromy
have homology which stabilizes.
After inverting finitely many primes, we also compute the
dominant part of the stable homology, i.e., the stable value of the homology
of covers where every conjugacy class appears as inertia sufficiently many times.
This improves upon our previous work in
\cite{landesmanL:the-stable-homology-of-non-splitting}
by removing the very stringent non-splitting assumption that appeared there.
As a consequence, 
we prove versions of Malle's conjecture,
prove an asymptotic version of the Picard rank conjecture,
and compute the 
moments predicted by Cohen-Lenstra-Martinet heuristics over $\mathbb F_q(t)$, for
$q$ suitably large depending on the moment.
We start by explaining our main results toward these applications, and then
proceed to explain our main results toward homological stability.

These three applications are merely meant to be a sampling of some of the
conjectures that our main homological stability results imply.
Just as Bhargava's thesis opened the gate to make significant progress
in arithmetic statistics problems over $\mathbb Q$, we hope that the
homological stability results we begin to develop here will give arithmetic statisticians the necessary
tools to
explore arithmetic statistics problems over function
fields. 

\subsection{Application 1: Malle's conjecture}

The well-known inverse Galois problem predicts that for any finite group $G$,
there is an extension of $\mathbb Q$ with Galois group $G$.
Malle's conjecture is a refinement of the inverse Galois problem, which not only
predicts that such a $G$ extension exists, but moreover predicts the asymptotic number of
such $G$ extensions.
One application of our homological stability results is that we can compute this
number over the function field $\FF_q(t)$ when $q$ is sufficiently large,
depending on $G$, and of suitable characteristic.

An interesting aspect of Malle's conjecture is that it is incorrect
in general. Kl\"uners gave
the first counterexample in
\cite{kluners:a-counterexample-to-malles-conjecture},
and since then many similar counterexamples have been constructed.
From the function field perspective, Malle's original conjecture seems to be
correct when one only counts $G$ covers which are geometrically connected, as
was observed in \cite{turkelli:connected-components-of-hurwitz-schemes}.
Moreover, a fix to Malle's conjecture was proposed
in \cite{turkelli:connected-components-of-hurwitz-schemes}. 
In \cite[Theorem 1.3]{wang:counterexamples-for-turkelli},
Wang shows that there are still counterexamples to T\"urkelli's modification over number fields and 
Wang proposes a new version of Malle's Conjecture on number fields in 
\cite[Conjecture 7]{wang:counterexamples-for-turkelli}.
See also
\cite[Conjecture 6]{wang:counterexamples-for-turkelli}.
Nevertheless, in this paper, we resolve this question over function fields
$\mathbb F_q(t)$, at least when $q$ is
sufficiently large depending on $G$ and relatively prime to $|G|$.
We show that, in this setting, T\"urkelli's prediction is actually correct.

We next state a special case of one of our results toward Malle's conjecture.
Let $\Delta(\mathbb F_q(t), G-\id, X)$ denote the number of $G$ field extensions $K$ over $\mathbb
F_q(t)$ such that 
the discriminant of $\mathscr O_K$ over $\mathbb F_q[t]$ is at most $X$;
here $\mathscr O_K$ is defined to be the normalization of
$\mathbb F_q[t]$ in $K$.

\begin{theorem}
	\label{theorem:malle-intro}
	Fix a finite permutation group $G \subset S_d$. There are constants
	$a(G-\id,\Delta), b_T(\mathbb F_q(t),
(G - \id)_\Delta)$, defined later in \autoref{notation:malle}, and a constant
$C$,
depending on $G$, with the following property: 
For $X$ sufficiently large, and $q > C$ a prime power with $\gcd(q,|G|) =1$, 
there are constants $C_-^{q,G}$ and $C_+^{q,G}$ so that
\begin{align*}
	C_-^{q,G}X^{\frac{1}{a(G-\id,\Delta)}} (\log
	X)^{b_T(\mathbb F_q(t), (G - \id)_{\Delta})-1} &\leq \Delta(\mathbb
	F_q(t),G-\id, X) \\
	&\leq
C_+^{q,G}X^{\frac{1}{a(G-\id,\Delta)}} (\log
X)^{b_T(\mathbb F_q(t), (G - \id)_{\Delta})-1}.
\end{align*}
\end{theorem}

We prove \autoref{theorem:malle-intro} as a special case of
\autoref{theorem:turkelli}, where we moreover verify asymptotics where we
restrict the ramification types of these $G$ extensions to lie in certain unions
of conjugacy classes, and also count by general invariants instead of only the
discriminant.
\begin{remark}
	\label{remark:}
	The constant $b_T(\mathbb F_q(t),
(G - \id)_\Delta)$ is the constant predicted by T\"urkelli
\cite[Conjecture 6.7]{turkelli:connected-components-of-hurwitz-schemes}
and so this confirms T\"urkelli's modified version of Malle's conjecture over
suitable function fields.
\end{remark}

\begin{remark}
	\label{remark:malle-periodic}
	Readers familiar with Malle's conjecture in the number field setting may
	be used to a formulation of Malle's conjecture stating that the
	number of $G$ Galois extensions of discriminant at most $X$ is of the form
$C^{q,G}X^{\frac{1}{a(G-\id,\Delta)}} (\log
X)^{b_T(\mathbb F_q(t), (G - \id)_{\Delta})-1}$. 
However,
in \autoref{theorem:malle-intro}, we only bound the number of extensions
above and below by different constants times
$X^{\frac{1}{a(G-\id, \Delta)}} (\log X)^{b_T(\mathbb F_q(t), (G - \id)_{\Delta})-1}$.
The reason that we only have upper and lower bounds with different constants in
\autoref{theorem:malle-intro}
is that, unlike in the number field case, it is simply false that the
limit 
\begin{align}
	\label{equation:malle-limit}
	\lim_{X \to \infty} \frac{\Delta(\mathbb
	F_q(t),G-\id, X)}{X^{\frac{1}{a(G-\id, \Delta)}}(\log X)^{b_T(\mathbb F_q(t), (G -
\id)_{\Delta})-1}}
\end{align}
exists in the function field setting.
See \autoref{example:malle-periodicity} for an explicit example of this.
However, we prove this limit does exist if one only counts geometrically connected
extensions by reduced discriminant, and takes the number of such instances to have fixed
residues modulo $|G|^2$.
We prove this variant in \autoref{theorem:malle-g-connected}.
See \autoref{question:extend-constant} for possible extensions of this.
\end{remark}

\begin{example}
	\label{example:malle-periodicity}
	In general, the discriminant of any extension of $\mathbb F_q(t)$ has a
	power of $q$, so \eqref{equation:malle-limit} cannot exist unless one
	takes the discriminant $X$ to range over powers of $q$. However, even restricting $X$ to
	powers of $q$, the limit rarely will exist.
	Consider the case $G = \mathbb Z/3 \mathbb Z$ and $q
	\equiv 2 \bmod 3$. 
	In this case, the discriminant is always a square, which again implies
	that if we counted by discriminant, and took a limit over all powers of
	$q$, the limit would not exist. Instead, we count by the reduced
	discriminant, see \autoref{example:reduced-discriminant}, which is
	equivalent to counting by the square root of the discriminant in this
	case.
	Then, one can show there are no $G$ extensions
	of $\mathbb F_q(t)$ of reduced discriminant $q^{n}$ for $n$ odd,
since there must always be the same number of geometric points with inertia 1 as
with inertia $2 \in \mathbb Z/3 \mathbb Z$. However, there are many extensions
of reduced discriminant $q^n$ for $n$ even. 
Let
$\on{rDisc}(\mathbb F_q(t),G-\id,  X)$ denote the number of $G$ Galois extensions of
$\mathbb F_q(t)$ with reduced discriminant at most $X$.
For $q \equiv 2 \bmod 3$, there is a constant $C$ so that the growth of
$\on{rDisc}(\mathbb F_q(t),G-\id,  X)$ 
is asymptotic to $C X$ when $X$ ranges over integers of the form $q^{2n}$,
for $n$ an integer,
and is asymptotic to $\frac{C}{q} X$ when $X$ ranges over integers of the
form $q^{2n+1}$ because $\on{rDisc}(\mathbb F_q(t),G-\id,  q^{2n+1}) =
\on{rDisc}(\mathbb F_q(t),G-\id,  q^{2n})$.
Although the situation is $2$-periodic here, and one of these cases has no
extensions, we believe in general that this
periodicity can become arbitrarily complicated when counting by discriminant.
For general groups, if one counts by reduced discriminant, and only counts geometrically connected
extensions, a ``periodic'' version of the limit from
\eqref{equation:malle-limit} does exist, as shown in
\autoref{theorem:malle-g-connected}.
\end{example}

\begin{remark}[Comparison with Ellenberg--Tran--Westerland]
	\label{remark:etw}
	Prior to this paper,
there has been substantial
progress toward proving Malle's conjecture over function fields. Namely,
\cite{ellenbergTW:fox-neuwirth-fuks-cells}
prove a weak upper bound for the number of $G$ extensions with the correct power
of $X$, but with a power of $\log X$ that is not correct in general (with the
	same restriction on $q$ that $q > C$ and $\gcd(q,|G|) = 1$ as in
\autoref{theorem:malle-intro}).
Additionally, they do not obtain a lower bound for the number of $G$ extensions.
In contrast, our \autoref{theorem:malle-intro} obtains both upper and lower
bounds, as well as the correct power of $\log X$.
\end{remark}

\begin{remark}
	\label{remark:}
	For past work on counting components of Hurwitz spaces, we refer the reader to
\cite{ellenbergV:counting-extensions} which dealt with counting components
parameterizing geometrically connected covers of $\mathbb A^1_{\mathbb F_q}$,
\cite{turkelli:connected-components-of-hurwitz-schemes}, which dealt with
counting components parameterizing connected covers of $\mathbb A^1_{\mathbb F_q}$, and
\cite{seguin:counting-components-of-hurwitz-spaces}, which dealt with counting
components parameterizing connected covers of 
$\mathbb A^1_{\mathbb C}$.
\end{remark}

\subsection{Application 2: The asymptotic Picard rank conjecture}

Let $G$ be a finite group and $c \subset G-\id$ be a conjugacy class generating
$G$.
We use $\churz {G} n c$ to denote the Hurwitz stack over $\mathbb C$
parameterizing geometrically connected $G$ covers of genus $0$ curves with $n$ branch points, each of
which have inertia in $c$.
See \autoref{notation:hurwitz-picard} for a precise definition.
An important special case is where $G = S_d$ and $c$ is the conjugacy class of
transpositions;
this Hurwitz space then also parameterizes simply branched covers of genus $0$ curves of degree $d$, which have
genus $g$, where $n = 2g - 2 + 2d$.
This is the setting of the original Picard rank conjecture, which predicts
$\pic(\churz {S_d} n c) \otimes \mathbb Q \simeq 0$
whenever $g \geq 0$ is an
integer, so $n \geq 2d-2$ is even
\cite[Conjecture 2.49(1)]{harris2006moduli}.
(See also the closely related
\cite[Conjecture 3]{diazE:towards-the-homology-of-hurwitz-spaces},
although there they work with Hurwitz spaces where they do not quotient by the
$\pgl_2$ action.)
We note that the roots of this conjecture extend further back, and a version of
it appears in work of Ciliberto from 1986
\cite[Conjecture 3.2]{ciliberto:rationally-determined-line-bundles} and even in
work of Enriques from 1919
\cite[p. 371]{enriques:questioni};
see the discussion surrounding 
\cite[Conjecture 3.2]{ciliberto:rationally-determined-line-bundles}
for more explanation on the relation to Enriques' work.

\begin{conjecture}[Picard rank conjecture]
	\label{conjecture:picard-rank}
	For all $d > 0$, even $n \geq 2d-2$, and $c \subset S_d$ the conjugacy
	class of transpositions, we have $\pic(\churz {S_d} n c) \otimes \mathbb Q \simeq 0.$
\end{conjecture}
\begin{remark}
	\label{remark:known-picard-rank}
	\autoref{conjecture:picard-rank} has been proven for $d \leq 5$ by Deopurkar-Patel
\cite{deopurkar2015picard}, but this proof relies on explicit parameterizations
of low degree Hurwitz spaces quite similar to those Bhargava used to count
number fields of degree at most $5$.
We also note that Mullane has proven \autoref{conjecture:picard-rank} whenever
$n \leq 4d - 2$ \cite{mullane:the-hurwitz-space-picard-rank-conjecture}. 
(He proves the result when the genus $g$ of the covering curve satisfies $g \leq
d$, and $n = 2g-2+2d$ is the number of branch points of a degree $d$ simply
branched cover.)
However, the case that $d \geq 6$ and $g > d$ remains open.
\end{remark}

\begin{figure}
	\centering
	\includegraphics[scale=.5]{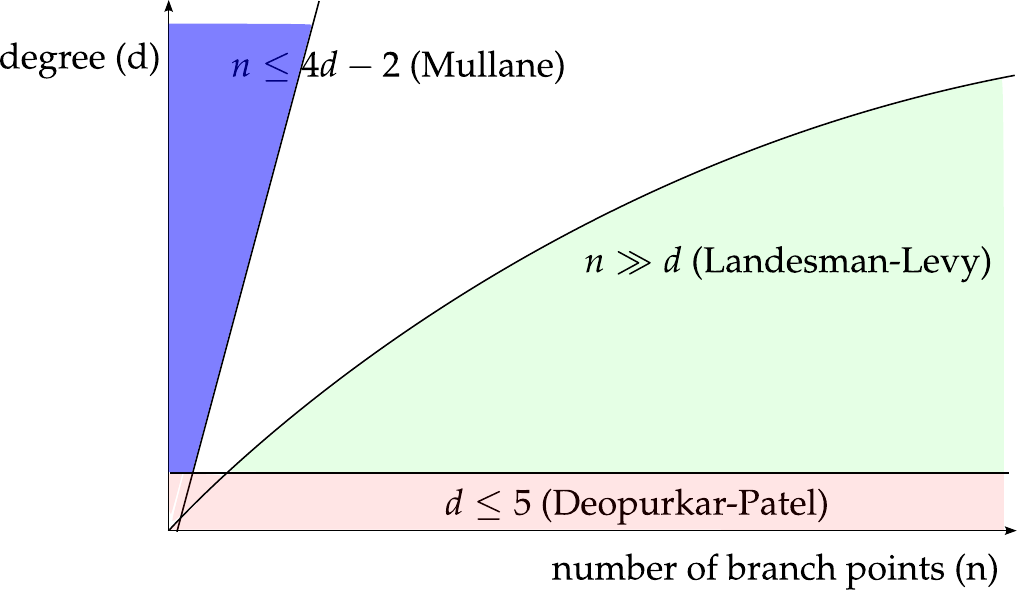}
	\caption{This figure depicts known cases of the Picard rank conjecture,
		see \autoref{remark:known-picard-rank} and
		\autoref{theorem:stable-picard-intro}.
}\label{figure:picard-rank-chart}
\end{figure}

We prove an asymptotic version of the Picard rank conjecture. That is, we prove
it when $n$, the number of branch points, or equivalently the genus of the
cover, is sufficiently large. 
\begin{theorem}
       \label{theorem:stable-picard-intro}
               Let $c \subset S_d$ denote the conjugacy class of
	       transpositions in the symmetric group acting on $d$ elements.
	For $n$ even and large enough depending on $d$, 
	$\pic(\churz {S_d} n c) \otimes \mathbb Q= 0$.
\end{theorem}
\autoref{theorem:stable-picard-intro} follows from
a stronger integral version stated in 
\autoref{theorem:stable-picard} below.
We also prove a version which not only applies to the symmetric group with
transpositions, but to arbitrary groups with a specified conjugacy class that
generates the group.

\begin{remark}
	\label{remark:}
	Although we do not explicitly work out the bound on $n$ as a function of
	$d$ appearing in \autoref{theorem:stable-picard-intro}, it is possible to
	make our proof effective by tracing through the proof and computing a bound.
	Without trying to optimize things, we found that the cohomology should
	be bounded for $n$ which grows just a little faster than $2^{\binom{d}{2}}$.
	Specifically, we found it is enough to take $n > 2 \cdot 2^{\binom{d}{2}} (2+\binom{d}{2})$.
\end{remark}

\subsection{Application 3: Cohen-Lenstra-Martinet heuristics}

In addition to Malle's conjecture on counting the number of $G$ extensions,
another fundamental suite of conjectures in arithmetic statistics are the
Cohen-Lenstra-Martinet heuristics. These conjectures are about the distribution
of class groups of $\Gamma$-extensions of $\mathbb Q$, for $\Gamma$ a fixed
finite group. The case $\Gamma = \mathbb Z/2 \mathbb Z$ was originally developed
by Cohen-Lenstra, and this case is known as the Cohen-Lenstra heuristics.
In \cite[Theorem 1.2.1]{landesmanL:the-stable-homology-of-non-splitting}, we were able to
compute the moments predicted by the Cohen-Lenstra heuristics over suitable
function fields.
Due to the limitation of the homological stability results there,
which only
applied to certain ``non-splitting'' Hurwitz spaces, we were only able to prove
the case of the Cohen-Lenstra-Martinet heuristics when $\Gamma = \mathbb Z/2
\mathbb Z$,
see \cite[Remark 5.5.2]{landesmanL:the-stable-homology-of-non-splitting}.
In this paper, we generalize our results toward homological
stability to apply to Hurwitz spaces associated to arbitrary racks.
As a consequence, we are able to compute many moments predicted by the more general Cohen-Lenstra-Martinet heuristics.

As those working in the area know, there are quite a number of variants of the
Cohen-Lenstra-Martinet heuristics. There are variants where one allows arbitrary
roots of unity in the base field. There are also variants where one considers
the distribution of the maximal unramified extension, instead of just the
distribution of the class group (which is the maximal {\em abelian} unramified
extension).
To illustrate the power of our methods, we state and prove a fairly general
version of the moments predicted by the Cohen-Lenstra-Martinet heuristics.
Namely, we prove a non-abelian version with roots of unity. However, we restrict
to the case that the extensions are split completely at infinity.

The conjecture for the non-abelian Cohen-Lenstra-Martinet  heuristics is laid out in 
\cite[Conjecture 1.3]{liuWZB:a-predicted-distribution},
and the generalization to the case with roots of unity in the base field,
was described in
\cite[Conjecture 1.2]{liu:non-abelian-cohen-lenstra-in-the-presence-of-roots},
where the author also keeps track of the additional data of a lifting invariant.
We now introduce some notation to explain the cases of this conjecture we can
prove.

\begin{notation}
	\label{notation:non-abelian}
	We say a finite extension $L/\mathbb F_q(t)$ is {\em split completely
	over $\infty$} if 
the place $\infty$ of $\mathbb
F_q(t)$, corresponding to the point $\infty := \mathbb P^1_{\mathbb F_q} -
\mathbb A^1_{\mathbb F_q} \in \mathbb P^1$ has $\deg(L/ \mathbb F_q(t))$
places over it in $L$.
We say a profinite extension $L/\mathbb F_q(t)$ is 
{\em split completely over $\infty$} if each intermediate finite extension of
$\mathbb F_q(t)$ is split completely over $\infty$.

	Fix a finite group $\Gamma$, let $K$ be a $\Gamma$
	extension of $\mathbb F_q(t)$, and let
$K^\sharp_\emptyset/K$ denote the maximal unramified extension of $K$ that 
has degree prime to $|\Gamma| q$ and
is
split completely over $\infty$.
Define
$G^\sharp_\emptyset(K) := \gal(K^\sharp_\emptyset/K)$.
Then, $G^\sharp_\emptyset(K)$ has an action of $\Gamma$ coming
from the $\Gamma = \gal(K/\mathbb F_q(t))$ action on $K$.
This action is well-defined using the Schur-Zassenhaus lemma; for further
explanation, see the paragraph prior to \cite[Definition
2.1]{liuWZB:a-predicted-distribution}.
For $H$ a group with a $\Gamma$ action,
we use $\surj_{\Gamma}(G^\sharp_\emptyset(K), H)$ to denote the
set of $\Gamma$ equivariant surjections $G^\sharp_\emptyset(K)
\to H$.
We use
$\hat{\mathbb Z}(1)_{(|\Gamma|q)'}$ to denote the pro-prime to $|\Gamma| q$ completion of
$\widehat{\mathbb Z}(1)_{\overline{\mathbb F}_q} := \lim_n \mu_n(\overline{\mathbb F}_q)$.
Here, for $\beta \mid \alpha$ the maps $\mu_\alpha(\overline{\mathbb F}_q) \to \mu_\beta(\overline{\mathbb F}_q)$
are given by $x \mapsto x^{\alpha/\beta}$.
We also use $(|\Gamma| q)'$ subscript on an abelian group to denote the prime to
$|\Gamma|q$ quotient of that group.
For $L/K$ an extension of fields, there is a certain map
$\omega_{L/K} : \hat{\mathbb Z}(1)_{(|\Gamma|q)'} \to H_2(\gal(L/\mathbb F_q(t)),
\mathbb Z)_{(|\Gamma|q)'}$ defined in \cite[Definition
2.13]{liu:non-abelian-cohen-lenstra-in-the-presence-of-roots}.
For $\pi \in \surj_{\Gamma}(G^\sharp_\emptyset(K), H)$, we use
$\pi_* : H_2( \gal(K^\sharp_\emptyset/\mathbb F_q(t)), \mathbb Z)_{(|\Gamma| q)'} \to
H_2(H \rtimes \Gamma, \mathbb Z)_{(|\Gamma| q)'}$ to denote the corresponding map
induced by $\pi$.
We say that $H$ is an {\em admissible $\Gamma$ group} 
if $\gcd(|H|, |\Gamma|) = 1$
and $H$ is generated 
by elements of the form $h^{-1} \cdot \gamma(h)$ for $h \in H$ and
$\gamma \in \Gamma$.
Also, let $E_\Gamma(D, \mathbb F_q(t))$ denote the set of pairs $(K, \iota)$
where $K$ is an extension of $\mathbb F_q(t)$ split completely at $\infty$
with reduced discriminant (meaning the radical of the ideal of the discriminant)
equal to $D$
and $\iota$ an isomorphism $\gal(K/\mathbb F_q(t)) \xrightarrow{\iota} \Gamma$.
\end{notation}

\begin{restatable}{theorem}{clm}
	\label{theorem:clm}
	With notation as in \autoref{notation:non-abelian}, suppose $H$ is an
	admissible $\Gamma$ group.
	Fix a prime power $q$ with $\gcd(q, |\Gamma||H|) = 1$.
	Let $\delta: \hat{\mathbb Z}(1)_{(|\Gamma| q)'} \to H_2(H \rtimes
	\Gamma, \mathbb Z)_{(|\Gamma| q)'}$
	be a group homomorphism with $\on{ord}(\im \delta) \mid q - 1$.
	Then, there is some constant $C$, depending on $H$ and $\Gamma$, so that if $q > C$,
\begin{equation}
	\label{equation:clm-ratio}
	\lim_{N \to \infty} \frac{\sum_{n \leq N} \sum_{K \in E_\Gamma (q^n,
		\mathbb F_q(t))} \left| \left\{ \pi \in \surj_\Gamma(
	G^\sharp_\emptyset(K), H) : \pi_* \circ \omega_{K^\sharp/K} = \delta
	\right\} \right| }{\sum_{n \leq N} |E_\Gamma(q^n, \mathbb F_q(t))| } =
	\frac{1}{[H : H^\Gamma]}.
\end{equation}
\end{restatable}

We prove this in \autoref{subsubsection:proof-clm}.

\begin{remark}
	\label{remark:}
	We note that our results will not prove the Cohen-Lenstra-Martinet heuristics in
full because if one fixes $q$, the Cohen-Lenstra-Martinet predict the $H$-moment
of the class group of $\Gamma$ extensions for arbitrary finite abelian groups $H$, and we can only
compute this for finitely many $H$. However, if one fixes a finite group $H$, we can compute
these moments for all $q$ which are relatively prime to $|H| |\Gamma|$ and which
are sufficiently large, depending on $H$.
\end{remark}

\begin{remark}
	\label{remark:}
	An additional generalization one may desire would be not to require that the
	$\Gamma$
	extension $K$ over $\mathbb F_q(t)$ is split completely over $\infty$,
	but instead allow
	different types of ramification over $\infty$.
It seems likely that one
could make a suitable conjecture 
and then prove it by combining the results of this
paper and with a generalization of the results of
\cite{liu:non-abelian-cohen-lenstra-in-the-presence-of-roots} to this setting. 
We believe it would be interesting to carry this out.
\end{remark}

\subsection{Homological stability results}

We next focus on explaining our advances toward understanding the stable homology
of Hurwitz spaces, which led to the above described applications.

Although Hurwitz spaces are usually constructed in the setting where one has a
group $G$ and a union of conjugacy classes $c$ generating that group, we will work
in the more general setting of a rack $c$ with orbits $c_1, \ldots, c_\upsilon$ and a reduced structure group $G^0_c$.
See \autoref{subsection:rack} for background on racks.
We suggest the reader unfamiliar with racks focus on the case that the rack
$c$ is a union of $\upsilon$ conjugacy classes $c = c_1 \cup \ldots \cup c_\upsilon$ which
generate a group $G$; in this case, the Hurwitz space
$\cphurc {n_1, \ldots, n_\upsilon} c$ parameterizes 
connected $G$ covers
of $\mathbb A^1_{\mathbb C}$ whose
branch divisor has $n_i$ points with inertia in $c_i$.
See \autoref{definition:rack-pointed-hurc} for a definition of Hurwitz spaces
over $\mathbb C$ in
the more general context of racks.
Our first main result shows that the integral homology of Hurwitz spaces
stabilizes in a
suitable sense.
We use $[g]$ to denote the map $\cphurc n c \to \cphurc {n+1} c$, viewed as a
map of homotopy quotients $c^n/B_n \to c^{n+1}/B_{n+1}$ induced by
$(g_1, \ldots, g_n) \mapsto (g_1, \ldots, g_n, g)$.

\begin{theorem}
	\label{theorem:some-large-homology-stabilizes}
	Let $c$ be a finite rack whose connected components are $c_1, \ldots, c_\upsilon$.
	for any
	$i$ and $\lambda$, with $i \geq 0$, $1 \leq \lambda \leq \upsilon$, 
	there are constants $I$ and $J$, depending only on $|c_\lambda|$ and the
	maximum order of any element of $c_\lambda$, so that 
	for
	$n_\lambda > Ii + J$,
	the maps $[g]$ for $g \in c_\lambda$
	all induce isomorphisms
	$H_i(\cphurc {n_1, \ldots, n_\upsilon} c, \mathbb Z) \to H_i(\cphurc
	{n_1,\ldots, n_{\lambda-1}, n_\lambda + 1, n_{\lambda+1}, \ldots, n_\upsilon} c, \mathbb Z)$.
\end{theorem}

We deduce \autoref{theorem:some-large-homology-stabilizes} from 
\autoref{theorem:spectra-homology-stabilization} in
\autoref{subsubsection:some-large-homology-stabilizes-proof}.

We next wish to compute the stable value of the homology, when all the $n_1,
\ldots, n_\upsilon$ are sufficiently large. We refer to this as the {\em
dominant} stable homology.
We use $\conf_{n_1, \ldots, n_\upsilon}$ to denote the configuration space of points in
$\mathbb C$ with $n_i$ points of color $i$, as defined in
\autoref{definition:conf}.
Recall we use $G^0_c$ to denote the reduced structure group of the rack $c$.

\begin{theorem}
	\label{theorem:all-large-stable-homology}
	Let $c$ be a finite rack whose connected components are $c_1, \ldots, c_\upsilon$.
	Then there are constants $I$ and $J$, depending on $c$, so that for any
	$i \geq 0$ and $n_1, \ldots, n_\upsilon > Ii + J$, 
	and any component $Z \subset \cphurc {n_1, \ldots, n_\upsilon} c$,
	the map
	$H_i(Z, \mathbb Z[\frac 1 {|G^0_c|}]) \to H_i(\conf_{n_1, \ldots, n_\upsilon}, \mathbb
	Z[\frac 1 {|G^0_c|}])$
	is an isomorphism.
\end{theorem}

When working in the case that $c \subset G$ is a union of conjugacy classes
generating a finite group $G$, we have $G^0_c = G/Z(G)$, for $Z(G)$ the center
of $G$, and so the above
computes the stable homology after inverting $|G/Z(G)|$.

\autoref{theorem:all-large-stable-homology} follows from
\autoref{theorem:dominantstablehomology}, computing the homology with all
elements of $c$ inverted,
and \autoref{theorem:some-large-homology-stabilizes},
showing that the $i$th homology of Hurwitz space does stabilize when $n_1, \ldots,
n_\upsilon > Ii +J$.
We omit further details, except to say that
this deduction is analogous to the deduction of
\cite[Theorem 1.4.7]{landesmanL:the-stable-homology-of-non-splitting},
carried out in \cite[\S4.2.3]{landesmanL:the-stable-homology-of-non-splitting},
from \cite[Theorem
4.2.2]{landesmanL:the-stable-homology-of-non-splitting},
computing the relevant homology with elements of $c$ inverted,
and
\cite[Theorem
6.1]{EllenbergVW:cohenLenstra},
showing the relevant homology stabilizes.

\begin{remark}
	\label{remark:}
	If one takes $c$ to be the set of transpositions in the symmetric group
	$S_d$, \autoref{theorem:all-large-stable-homology} implies
	\cite[Conjecture 1.5]{EllenbergVW:cohenLenstra}.
\autoref{theorem:all-large-stable-homology} also implies a recent conjecture of
the authors stated in
	\cite[Conjecture
	1.6.4]{landesmanL:the-stable-homology-of-non-splitting}.
\end{remark}

\subsection{Proof ideas}

\subsubsection{Proof of the main homological stability results}
We now explain the main ideas in the proof of our homological stability
results.
We focus on the case that $c \subset G$ is a single conjugacy class
generating $G$, but a similar explanation applies to arbitrary finite racks. 
Our proof builds crucially both on
Ellenberg-Venkatesh-Westerland's Annals paper
showing that non-splitting Hurwitz spaces have homology which stabilizes
\cite{EllenbergVW:cohenLenstra}
(and the alternate proof outline in Oscar Randal-Williams Bourbaki article
\cite{randal-williams:homology-of-hurwitz-spaces})
and our prior work computing the stable value of this homology in the
non-splitting case
\cite{landesmanL:the-stable-homology-of-non-splitting}.

A key point in the work of 
\cite{EllenbergVW:cohenLenstra} proving homological stability for certain
Hurwitz spaces is that one can often prove that $H_i(\Hur_n^c)$ stabilizes as $n
\to \infty$ by induction on $i$, once one knows that $H_0(\Hur_n^c)$ stabilizes. In the case that $c$ is a single conjugacy class, this only happens when $c$ satisfies the \textit{non-splitting condition} of 
\cite{EllenbergVW:cohenLenstra}, i.e., the intersection of $c$ with any proper
subgroup $G' \subset G$ must either be empty or remain a single conjugacy class. We note that this condition is quite restrictive, and fails already when $c$ is the conjugacy class of transpositions in $S_d$ for $d \geq4$. 

One can rephrase the condition that the operator $x: H_*(\Hur_n^c) \to
H_*(\Hur_{n+1}^c)$ induces homological stability (meaning that $x$ induces an
isomorphism $H_i(\Hur_n^c) \to
H_i(\Hur_{n+1}^c)$ for $i$ sufficiently large, depending on $n$)
as follows: Consider $C_*(\Hur^c)$, the graded
differential graded algebra of singular chains of $\Hur^c =
\coprod_{n=0}^{\infty}\Hur^c_n$. Then proving homological stability for
$C_*(\Hur^c)$ is the same as proving that $C_*(\Hur^c)/x$ is \textit{bounded in
a linear range}, i.e., it has homology groups vanishing in a range of degrees of
the form $*>An+B$ with $A>0$. Here $C_*(\Hur^c)/x$ refers to the cofiber (or mapping cone) of multiplication by $x$ on $C_*(\Hur^c)$, whose homology groups encode the kernel and cokernel of multiplication by $x$ on the homology of $\Hur^c$. 

The first ingredient in our proof is to show that 
 for a sequence of central operators $x_1,\dots,x_m$, if the Koszul complex $C_*(\Hur^c)/(x_1,\dots,x_m)$ has homology bounded in degree $0$, then its homology is bounded in a linear range. The proof is a natural generalization of what the
techniques of \cite{EllenbergVW:cohenLenstra} (specifically as presented in
\cite{randal-williams:homology-of-hurwitz-spaces}) allow one to prove with
regards to homological stability. However, directly proving boundedness of $C_*(\Hur^c)/(x_1,\dots,x_m)$ in a linear range allows us to significantly simplify the previous techniques for proving homological stability, even in the non-splitting case. We take the $x_1,\dots,x_m$ to be powers of each of the elements of $c$, so that a simple pigeonhole principle argument shows that $C_*(\Hur^c)/(x_1,\dots,x_m)$ is bounded in degree $0$.

%

So far, this proves $C_*(\Hur^c)/(x_1,\dots,x_m)$ is bounded in a linear range,
which is not homological stability, but rather a polynomial stability result,
roughly saying that $H_i$ asymptotically grows like a polynomial of degree at
most $m-1$.\footnote{See \cite[Theorem A]{bianchiM:polynomial-stability} for a statement of similar flavor.}

Because $\Hur^c$ is a disjoint union of connected Hurwitz spaces for conjugacy classes of various subgroups of $G$, a filtration argument along with induction on the size of $c$ implies that $C_*(\CHur^c)/(x_1,\dots,x_m)$ is also bounded in a linear range.

The next step is to decrease the number of cofibers taken by descending induction, until we have proven homological stability. More precisely, we show by descending induction on $1\leq k \leq m$ that $C_*(\CHur^c)/(x_1,\dots,x_k)$ is bounded in a linear range. The key ingredient in the inductive step is proving that $C_*(\CHur^c)/(x_1,\dots,x_k)[x_{k+1}^{-1}] = 0$. Because we are inverting an operator, this is a stable homology question, so we prove this using the techniques in \cite{landesmanL:the-stable-homology-of-non-splitting} for understanding stable homology, as well as another filtration argument to go between $\Hur^c$ and $\CHur^c$.

In the end, we learn that $C_*(\CHur^c)/x_1$ is bounded in a linear range, which
is equivalent to the claimed homological stability result. Computing the
dominant part of the stable homology then proceeds very similarly to the proof
of \cite[Proposition 4.5.1]{landesmanL:the-stable-homology-of-non-splitting}.

\subsubsection{Proof idea for the Picard rank conjecture}
\label{subsubsection:picard-rank-idea}
The rough idea for proving the Picard rank conjecture is that the tangent space
to the Picard group of a variety should be controlled by its first
cohomology and the component group should be controlled by its second
cohomology.
At least this is true for smooth projective varieties.
We use our computation of the stable values of these two cohomology groups
to compute the Picard group of Hurwitz space via applying the above ideas to 
a suitable compactification of Hurwitz space.

\subsubsection{Proof idea for the Cohen-Lenstra-Martinet moments}

The Cohen-Lenstra-Martinet conjectures can be rephrased as showing that the
number of $\mathbb F_q$ points on certain Hurwitz spaces for the group $H \rtimes \Gamma$ agree with the number
of $\mathbb F_q$ points on certain Hurwitz spaces for $\Gamma$.
We reduce the question of counting their $\mathbb F_q$ points to showing their
cohomologies over $\overline{\mathbb F}_q$ agree and the Frobenius actions also
agree.
The fact that their cohomologies agree boils down to
our computation of the stability homology of Hurwitz spaces from \autoref{theorem:all-large-stable-homology}.

\subsubsection{Proof idea for Malle's conjecture}

In order to count $G$ extensions of $\mathbb F_q(t)$, we can equivalently count
$\mathbb F_q$ points of Hurwitz spaces for $G$.
As usual, we can reduce this to understanding the trace of Frobenius on the
cohomology of these Hurwitz spaces over $\overline{\mathbb F}_q$.
Using our main homological stability results, along with equivariance of the
Frobenius actions along the stability maps,
which we prove in \autoref{section:frobenius-equivariance}, we can essentially reduce the
problem to counting the components of Hurwitz spaces.
Many of the basic ideas for how one could count these were outlined by
T\"urkelli \cite{turkelli:connected-components-of-hurwitz-schemes}.
We give our own rendition of these ideas, utilizing the machinery of the lifting invariant
of Ellenberg-Venkatesh-Westerland in a paper of Wood
\cite{wood:an-algebraic-lifting-invariant}
to count these components.

A crucial ingredient beyond homological stability is the verification that the topological stabilizations
maps are suitably equivariant for the action of Frobenius on cohomology.
To verify this, we use methods from logarithmic geometry, similar to
\cite[Appendix
A]{ellenbergL:homological-stability-for-generalized-hurwitz-spaces}. 
There are a number of additional subtleties in our present setting,
especially relating to the fact that the boundary monodromy may be nontrivial.

\subsection{Outline}

This paper is organized as follows.
The first part of the paper concerns our topological results.
In \autoref{section:background}, we review background and collect various
notation we use throughout the paper for Hurwitz spaces.
In \autoref{section:stabquot} we prove a weak form of homological stability for
Hurwitz spaces, showing that a suitable quotient (in the sense of higher algebra)
of Hurwitz spaces stabilizes.
In \autoref{section:algebra-preliminaries}
we prove a key algebraic tool \autoref{prop:rigiditymodules},
which gives a criterion for the base change along a map of ring spectra to not change a module.
We then use this proposition in
\autoref{section:homological-stability}
to show that the homology of Hurwitz spaces stabilizes, proving
\autoref{theorem:some-large-homology-stabilizes}.
We conclude our topological part of the paper in
\autoref{section:dominant-stable-homology} by computing the dominant stable
homology
Hurwitz spaces, thereby proving \autoref{theorem:all-large-stable-homology}.
We then press on to the second part of the paper, which contains our three main
applications.
In \autoref{section:picard}, we prove an asymptotic version of the Picard rank
conjecture.
In \autoref{section:frobenius-equivariance}, we use tools from log geometry to 
show the stabilization map on cohomology of Hurwitz spaces is suitably
equivariant for the action of Frobenius.
We then apply this in \autoref{section:clm} to compute the moments predicted by
Cohen-Lenstra-Martinet
and in \autoref{section:malle} to prove versions of Malle's conjecture.
We conclude with further questions in \autoref{section:further-questions}.

\subsection{Acknowledgements}
We would like to give an especially gracious acknowledgement to Andrea Bianchi,
who explained to us a way to remove an extraneous hypothesis from a technical result of our previous paper
\cite{landesmanL:the-stable-homology-of-non-splitting}
so that it was suitably general to apply in the setting of the present paper.
We thank Jiuya Wang for extensive discussions regarding Malle's conjecture, and
many detailed comments on a draft.
We thank Joe Harris and Anand Patel for helpful discussions regarding the Picard rank
conjecture.
We thank Dori Bejleri for helping us prove a technical result regarding normal
crossings compactifications of Hurwitz spaces in \cite[Appendix
B]{ellenbergL:homological-stability-for-generalized-hurwitz-spaces}, which was
needed for this paper.
We thank Sun Woo Park for extensive comments on a draft.
We also thank 
Brandon Alberts,
Dan Abramovich,
Anand Deopurkar,
Andrea Di Lorenzo,
Jordan Ellenberg,
Pavel Etingof,
Nir Gadish,
Jack Hall,
Peter Koymans,
Andrew Kresch,
Eric Larson,
Daniel Litt,
Yuan Liu,
Curt McMullen,
Martin Olsson,
Dan Petersen,
Sasha Petrov,
Tim Santens,
Will Sawin,
Phil Tosteson,
Leandro Vendramin,
Akshay Venkatesh,
Nathalie Wahl,
Craig Westerland,
Melanie Wood,
Zhiwei Yun,
and
Bogdan Zavyalov
for helpful conversations.
Landesman 
was supported by the National Science
Foundation 
under Award No.
DMS 2102955.
Levy was supported by the Clay Research
Fellowship.

\section{Background and notation for Hurwitz spaces}
\label{section:background}

In this section, we develop the theory of Hurwitz spaces associated to a general
rack. Hurwitz spaces parameterizing $G$-covers, for $G$ a group, are quite well
established, but it seems to us that the natural setting to define Hurwitz
spaces is really in the more general setting of a rack. It will be important to
our proofs that we consider Hurwitz spaces associated with quotient racks, which
may not come from unions of conjugacy classes in a group.

To this end, we review background on racks in \autoref{subsection:rack}.
We introduce Hurwitz spaces over the complex numbers associated to racks in
\autoref{subsection:complex-hurwitz}.
We then discuss Hurwitz spaces associated to unions of conjugacy classes
in groups over more general bases in \autoref{subsection:general-hurwitz-notation}.
The above Hurwitz spaces occur over $\mathbb A^1$, which are more natural from
the perspective of topology. In \autoref{subsection:hurwitz-p1}, we introduce
notation for
Hurwitz spaces over $\mathbb P^1$, which are more natural from the perspective
of algebraic geometry.

\subsection{Background on racks}
\label{subsection:rack}

We review some basic definitions associated to racks.

\begin{definition}
	\label{definition:rack}
	A {\em rack} is a set $c$ with an action map $\triangleright: c \times c
	\to c, (a,b) \mapsto a \triangleright b$ such that for all $n \geq 1$
	and all $1 \leq i \leq n-1$, the operation
	\begin{align*}
		\sigma_i : c^n & \rightarrow c^n \\
		(x_1, \ldots, x_{i-1}, x_i, x_{i+1}, x_{i+2}, \ldots, x_n) & \mapsto
		(x_1, \ldots, x_{i-1}, x_{i+1}, x_{i+1} \triangleright x_i, x_{i+2}, \ldots,
		x_n)
	\end{align*}
	defines an action of the braid group $B_n$, generated by $\sigma_1,
	\ldots, \sigma_{n-1}$, on $c^n$.
\end{definition}
\begin{remark}
	\label{remark:}
	Often, racks are defined as sets with a binary operation $\triangleright$
	such that $x\triangleright(-):c \to c$ is a bijection for each $x \in c$ and and $x \triangleright (y \triangleright z) = (x
	\triangleright y) \triangleright (x \triangleright z)$. It is
	straightforward to see the above definition is equivalent to this one
	using the defining relations of the braid group.
\end{remark}
\begin{example}
	\label{example:}
	Suppose $G$ is a group and $c \subset G$ is a conjugacy invariant
	subset, in the sense that for any $g,h \in c$ we also have $g^{-1}hg \in
	c$.
	Then the operation $g \triangleright h := g^{-1} hg$ endows $c$ with the structure
	of a rack.
\end{example}
\begin{definition}
	\label{definition:}
	The {\em reduced structure group} $G^0_c$ of a rack $c$ is the subgroup of
	$\on{Aut}(c)$, the automorphism group of the rack $c$, generated by the
	automorphisms $y \mapsto x \triangleright y$ for all $x \in c$. It
	follows from the definition of a rack that these are rack automorphisms.
\end{definition}

\begin{definition}\label{definition:abelianizationrack}
	Given a rack $c$, we let $c/c$ denote the rack with underlying set the orbits of $c$ under the action of the reduced structure group. The rack structure is commutative, i.e. $x \triangleright y = y$ for all $x,y \in c/c$. There is a natural map of racks $c \to c/c$. We refer to $c/c$ as the set of \textit{components} of the rack $c$, and refer to each fiber of the projection of the map $c \to c/c$ as a component of $c$.
\end{definition}

\subsection{Hurwitz spaces over the complex numbers}
\label{subsection:complex-hurwitz}

In this subsection, we define Hurwitz spaces over the complex
numbers associated to an arbitrary rack.

\begin{definition}
	\label{definition:conf}
	Given a scheme $B$, there is an open subscheme $U \subset \mathbb A^n_B$
	parameterizing the locus where all coordinates are distinct. There is an
	action of the symmetric group $S_n$ on $U$ by permuting the coordinates and we define $\conf_{n,B} :=
	U/S_n$ to be the {\em configuration space} of $n$ points in $\mathbb A^1$ over
	$B$.
	More generally, let $\conf_{n_1, \ldots, n_\upsilon, B} := U/S_{n_1} \times \cdots
	\times S_{n_\upsilon}$ where $S_{n_i}$ acts on the $n$ consecutive
	coordinates in the range $[n_1 + \cdots +
	n_{i-1} + 1, n_1 + \cdots + n_i]$.
	When $B = \spec R$, for $R$ a ring, we often write this as $\conf_{n_1,
	\ldots, n_\upsilon,R}$
	and when $R = \mathbb C$, we abbreviate this to $\conf_{n_1, \ldots,
	n_\upsilon}$.
\end{definition}

\begin{definition}
	\label{definition:rack-pointed-hurc}
	Let $c$ be a rack. 
	Upon identifying $\pi_1(\conf_{n}) \simeq B_n$, we can
	identify finite \'etale covers of $\conf_{n}$ with maps from
	$B_n$ to finite sets.
	Define the {\em pointed Hurwitz scheme over $\mathbb C$} to be the
	finite \'etale cover $\phurc n c \to \conf_{n}$
	corresponding to the
	map $B_n \to \aut(c^n)$ associated to the rack from its definition in
	\autoref{definition:rack}.
	If $c$ is a rack with orbits $c_1, \ldots, c_\upsilon$ and $n_1+ \cdots+ n_\upsilon =
	n$, let $c^{n_1, \ldots, n_\upsilon} \subset c^n$ denote the subset such that
	there are $n_i$ elements in orbit $c_i$. Then we define 
	$\phurc {n_1, \ldots, n_\upsilon} c \to \conf_{n_1, \ldots, n_\upsilon}$ to be the finite \'etale cover
	corresponding to the map $B_n \to \aut(c^{n_1, \ldots, n_\upsilon})$.
	This is a union of components of $\phurc {n_1, \ldots, n_\upsilon} c$.
There is a subset $(c^n)^\circ \subset c^n$ parameterizing $n$ tuples of
elements in $c$ whose actions together generate $G^0_c$.
We let $\cphurc n c$ denote the finite \'etale cover of $\conf_n$ corresponding
to the map $B_n \to \aut( (c^n)^\circ)$. Finally, we let $\cphurc {n_1, \ldots,
n_\upsilon} c$ denote the finite \'etale cover of $\conf_{n_1, \ldots, n_\upsilon}$
corresponding to the map $B_n \to \aut( (c^n)^\circ \cap c^{n_1, \ldots, n_\upsilon})$.
\end{definition}

\begin{example}
	\label{example:}
	Suppose $c \subset G$ is a union of $\upsilon$ conjugacy classes $c_1
	\cup \cdots \cup c_\upsilon$ in $G$.
	Then the complex points of 
	$\phurc n c$ parameterize $G$ covers of $\mathbb{C}$ with $n$ branch
	points whose inertia lies in $c \subset G$ with
	trivialization of the $G$ cover in some subset $[k,\infty]$ for some $k \in \RR$ (where the choice of $k$ is not part of the data of a point in the space).
	Similarly, $\cphurc n c$ parameterizes covers as above whose source is connected. 
	The complex points of
	$\phurc {n_1, \ldots, n_\upsilon} c$ parameterize 
	covers with $n_i$ branch points whose inertia lies in $c_i$,
	and $\cphurc {n_1, \ldots, n_\upsilon} c$ parameterizes such covers whose source is connected. 
\end{example}

\subsection{Notation for Hurwitz spaces over general bases}
\label{subsection:general-hurwitz-notation}

We now introduce notation for Hurwitz schemes over bases other than the complex
numbers. We start by recalling the definition of the pointed Hurwitz scheme
as defined in \cite{landesmanL:the-stable-homology-of-non-splitting}.
\begin{definition}
	\label{definition:hurwitz-over-b}
	Fix a finite group $G$
	and a subset $c \subset G$ of the form
	$c = c_1 \cup \cdots \cup c_\upsilon$ for $c_1, \ldots, c_\upsilon$
	pairwise distinct
	conjugacy classes in the subgroup of $G$ generated by $c$,
	a scheme $B$ on which $|G|$ is invertible, and positive integers $n$ and
	$v$.
	Assume either $c$ is closed under invertible powering (meaning $g \in c
	\implies g^t \in c$ for any $t$ relatively prime to $|G|$) or
	$B$ is Henselian with residue field $\spec \mathbb F_q$ and $c$ is
	closed under $q$th powering ($g \in c \implies g^q \in c$).
	We use the notation $\phur G {n} {c} B$
	to denote the {\em pointed Hurwitz scheme} as defined in
	\cite[Definition
	2.1.3]{landesmanL:the-stable-homology-of-non-splitting}.
	The scheme
	$\phur G {n} {c} B$
	was shown to exist in \cite[Remark 2.1.2 and
	2.1.4]{landesmanL:the-stable-homology-of-non-splitting}.
	The $T$ points of $\phur G {n} {c} B$ correspond to tuples 
	$(D, h' : X \to \mathscr{P}^w_T, t: T \to X \times_{h', \mathscr{P}^w_T,
	\widetilde{\infty}_T} T, i: D \to \mathbb P^1_T, X, h: X \to
	\mathbb P^1_T)$
	such that $D$ is a finite \'etale of degree $n$ over
	$T$, $i: D \subset \mathbb A^1_T$ is a closed immersion, $X$ is a smooth
	proper relative curve over $T$, $h: X \to \mathbb P^1_T$ is a finite
	locally free Galois $G$ cover \'etale away from $\infty_T \cup i(D),$
	(for $\infty_T$ the section of $\mathbb P^1_T \to T$ at $\infty$,)
	the inertia of $X \to \mathbb P^1_T$ over any geometric point of $i(D)$
	lies in $c$, $\mathcal P^w_T$ is a root stack of $\mathbb P^1$ over
	$\infty$ of order equal to the order of inertia of $h$ over $\infty$,
	$\widetilde{\infty}_T$ is the base change of the natural section
	$\widetilde{\infty} : B \to \mathscr P^w$, $h'$ is a finite locally free
	cover \'etale over $\widetilde{\infty}_T$ so that the composition with
	the coarse space map $\mathscr{P}^w_T \to \mathbb P^1_T$ is $h$, and $t$
	is a section of $h'$ over $\widetilde{\infty}_T$. 

	We let $\cphur {G}{n} {c} B \subset \phur G {n} {c} B$
	denote the union of connected components parameterizing geometrically
	connected covers $X$. 
\end{definition}

\begin{warn}
	\label{warning:}
	We have defined 
	$\phurc {n} {c}$
	over the complex numbers for an arbitrary rack $c$, but
	we have only defined 
$\phur {G} {n} {c} B$
	over a general base $B$ when $c$ is a rack which is a subset of a group $G$ whose
	order is invertible on $B$.
	The reason for this is that, when $c \subset G$, there is a stack
	parameterizing $G$ covers defined over $B$. For a general rack, we are
	not sure how to define such a stack over $B$.
\end{warn}

\begin{definition}
	\label{definition:hurwitz-stack-over-b}
	Let $G'$ be a group and $c \subset G'$ be a subset which is closed under
	conjugation.
	Suppose moreover that $c$ generates a normal subgroup $G \subset G'$.
	Fix a base scheme $B$ and
	suppose $c$ satisfies the same hypotheses as in
	\autoref{definition:hurwitz-over-b}.
	Then, the pointed Hurwitz scheme $\phur {G'} n c B$ parameterizes
	$G'$ covers of $\mathbb A^1$ together with a marked section $t$, as
	described in \autoref{definition:hurwitz-over-b}.
Note here that
since $c$ is contained in $G$, $c$ does not generate $G'$ unless $G = G'$.
The geometric points of 
$\phur {G'} n c B$
correspond to covers
consisting of disjoint unions of
$|G'|/|G|$ many $G$ covers.
The fiber over $t$ as above still has an action of $G'$.
This moreover endows both
$\phur {G'} {n} {c} B$ and $\cphur {G'}{n} {c} B$
with actions of $G'$, the latter only being nonempty if $G = G'$. For any subgroup $K \subset G'$, we define
$\quohur {G}{K}{n} {c} B$ and $\cquohur {G}{K} {n} {c} B$
to be the corresponding quotients.
We note that the definitions of these stacks implicitly involves choosing
an inclusion $K \to G'$ and an inclusion $G \subset G'$ as
a normal subgroup.
\end{definition}

The next lemma will be frequently used when comparing Hurwitz schemes over
$\mathbb C$ and $\overline{\mathbb F}_q$.
\begin{lemma}
	\label{lemma:irreducible-generization-criterion}
	Suppose $B$ is a local Henselian scheme with closed point $s$ and generic point
	$\eta$. Let $X$ be a Deligne-Mumford stack and $\pi:X \to B$ a smooth proper morphism.
	Suppose $X^\circ \subset X$ is an open immersion which is dense in
	every irreducible component over every fiber of $B$.
	If $Z \subset (X^\circ)_s \to s$ is an irreducible component of the special fiber of
	$X^\circ \to B$,
	there is an irreducible component $Z_B \subset X^\circ$ so that $Z_B
	\times_B s =Z$ and $Z_B \times_B \eta$ have the same number of geometric
	components that $Z$ has.
\end{lemma}
\begin{proof}
	Since $X^\circ \subset X$ is dense in every irreducible component over
	every fiber of $B$, it suffices to prove the statement when $X^\circ =
	X$.
	Consider the Stein factorization $X \to T \to B$, so that by definition
	$T$ is the Deligne-Mumford stack which is the relative spectrum $\underline{\spec}_{\mathscr O_B} (\mathscr
	O_X)$.
	Note that $T\to B$ is smooth and proper because $\pi: X\to B$ is smooth
	and proper and $X \to T$ is surjective.
	Since $T \to B$ is also quasi-finite, $T \to B$ is \'etale.
	Since $X \to T$ has geometrically connected fibers, it suffices to prove
	the result for $T$ in place of $X$, and so we may reduce to the
	case that $X \to B$ is proper and \'etale.
	Now, if $S$ denotes the coarse moduli space of $X$, we have that $X \to
	S$ induces a bijection on irreducible components over each fiber. Hence,
	it suffices to prove the result for $S$ in place of $X$, and
	hence we may assume that $X$ is a scheme which is finite and \'etale over
	$B$.
	The statement then follows because finite \'etale covers of $B$ are in
	bijection with finite \'etale covers of $s$, using that $B$ is
	Henselian.
\end{proof}

One easy consequence of \autoref{lemma:irreducible-generization-criterion}
is that we can identify components of Hurwitz schemes over $\overline{\mathbb
F}_q$ and $\mathbb C$, as we next record.
We will use this bijection implicitly throughout the paper.
One may also deduce the following result from \cite[Lemma
10.3]{liuWZB:a-predicted-distribution}.
\begin{lemma}
	\label{lemma:component-bijection}
	The specialization map induces a bijection between the irreducible
	components of $\cquohur G K {n} c{\mathbb{C}}$
	and the irreducible components of $\cquohur G K {n} c {\overline{\mathbb
	F}_q}$.
\end{lemma}
\begin{proof}
	Let $B$ be a local henselian scheme with residue field
	$\overline{\mathbb F}_q$ and geometric generic point $\spec \kappa$ of characteristic $0$.
By \cite[Corollary
B.1.4]{ellenbergL:homological-stability-for-generalized-hurwitz-spaces},
there is a normal crossing compactification of
$\cquohur G K {n} c {B}$ which is smooth over $B$.
Hence, using \autoref{lemma:irreducible-generization-criterion}, the
specialization map induces a
bijection between the components of 
$\cquohur G K {n} c {\overline{\mathbb F}_q}$ and components of
$\cquohur G K {n} c {\kappa}$.
The latter is identified with the irreducible components of
$\cquohur G K {n} c {\mathbb C}$ via base change to a common algebraically
closed field containing both $\mathbb C$ and $\kappa$.
\end{proof}

We next define generalize 
$\phurc {n_1, \ldots,n_\upsilon} c$ to other base fields, when $c$ is a union of
conjugacy classes in a group.
\begin{definition}
	\label{notation:multi-hurwitz-pre-quotient}
	Continuing to use notation as in \autoref{definition:hurwitz-over-b},
	let $S$ be the spectrum of an algebraically closed field.
	Using \autoref{lemma:component-bijection} we define
	$\phur {G'} {n_1, \ldots, n_\upsilon} c S$ informally as the union of
	components of 
$\phur {G'} {n} c S$
parameterizing covers with $n_i$ branch points with inertia in $c_i$ and $n =
n_1 + \cdots + n_\upsilon.$
More formally, when $S = \spec \mathbb C$ we define it to be $\phurc {n_1, \ldots,
n_\upsilon} c$.
This is base changed from $\spec \overline{\mathbb Q}$ and we define 
$\phur {G'} {n_1, \ldots, n_\upsilon} c {\spec \overline{\mathbb Q}}$ to denote those
components of 
$\phur {G'} {n} c S$
whose base change to $\mathbb C$ is
$\phurc {n_1, \ldots,
n_\upsilon} c$.
When $S=\overline{\mathbb F}_p$,
it is the union of components of 
$\phur {G'} {n} c S$ 
corresponding to 
$\phurc {n_1, \ldots, n_\upsilon} c$ under 
\autoref{definition:hurwitz-over-b}.
Since base change between algebraically closed fields induces a bijection on
components, for general $S$, this enables us to define
$\phur {G'} {n_1, \ldots, n_\upsilon} c S$
as the set of components of
$\phur {G'} {n} c S$
obtained via base change from either $\overline{\mathbb Q}$ or
$\overline{\mathbb F}_p$.
\end{definition}

Using the above, we can define quotients of Hurwitz spaces as well.
\begin{notation}
	\label{notation:multi-hurwitz}
	We use notation from \autoref{notation:multi-hurwitz-pre-quotient}.
	Let
	$\quohur G K {n_1,\ldots, n_\upsilon} c S$
	denote those components of
$\quohur G K {n} c S$ 
in the $K$ 
orbit of components of the form
$\phur {G'} {n_1, \ldots, n_\upsilon} c S$ 
in 
the quotient of
$[\phur {G'} {n} c S/K]$.
We similarly use
$\cquohur G K {n} c S$ 
to denote the $K$ orbit of components of
$\cphur {G'} {n_1, \ldots, n_\upsilon} c S$ 
in
$[\cphur {G'}{n} c S/K]$.
\end{notation}

The next lemma gives a combinatorial description of the components of the above
Hurwitz stacks.
For the next lemma, we use the notion of descent along the Galois extension
$\spec \overline{\mathbb F_q} \to \spec \mathbb F_q$.
For a general reference on descent, we recommend
\cite[\S6.1]{BoschLR:Neron}, see also 
\cite[\S6.2 Example B]{BoschLR:Neron} for the case of finite Galois descent.

\begin{lemma}
	\label{lemma:g-irred-components}
Using notation from \autoref{definition:hurwitz-over-b},
and assume $\gcd(q, G) = 1$,
the following hold true.
	\begin{enumerate}
\item 	The components of
	$\cquohur G K {n} c {\overline{\mathbb F}_q}$
	can be described as tuples $(g_1, \ldots, g_n)$ modulo the action of
	$B_n$, up to the simultaneous $K$ conjugation action.
\item 	Such a component is the base change of a component of 
	$\cquohur G K {n} c {\mathbb F_q}$
	if and only if there is descent data for $\overline{\mathbb F}_q$ over $\mathbb
	F_q$ sending that component to
	a component of the form
	$(h^{-1} g_1 h, \ldots, h^{-1} g_n h)$,
	which is defined up to the action of $B_n$,
	for some $h \in K$.
\item 	In particular, if there are $n_i$ elements among $g_1, \ldots, g_n$ which
	lie in the conjugacy class $c_i \subset G$, 
	in order for $(g_1, \ldots, g_n)$ to correspond to a geometrically
	irreducible component of 
	$\cquohur G K {n} c {\mathbb F_q}$,
	it is necessary that there exists $h \in K$ with 
	$\sum_i n_i c_i = \sum_i n_i h c_i^{q^{-1}} h^{-1}$.
	\end{enumerate}
\end{lemma}
\begin{proof}
To prove $(1)$, the description in the case $K =\id$ is verified in
\cite[Theorem 12.4]{liuWZB:a-predicted-distribution} and
\autoref{lemma:component-bijection}
to pass between $\mathbb C$ and $\overline{\mathbb F}_q$.
For the case of general $K$, we simply note that the $K$ action corresponds to
changing the choice of marked point over $\infty$, and hence acts by conjugation
on the choice of generators of the fundamental group of a punctured $\mathbb
A^1$, so components of the quotient by $K$ correspond to $K$ orbits of
components before the quotient by $K$.

Then, $(2)$ follows from $(1)$, because the condition for a component 
to be the base change of a component over $\mathbb F_q$ is precisely the
existence of descent data which fixes that component.

Finally, $(3)$ is a consequence of $(2)$ because if 
$(g_1^{q^{-1}}, \ldots, g_n^{q^{-1}})$
	is equivalent under the $B_n$ action to
	$(h^{-1} g_1 h, \ldots, h^{-1} g_n h)$
	for some particular $h \in K$,
then we also have that $(g_1, \ldots, g_n)$
is equivalent under the $B_n$ action to 
$(h g_1^{q^{-1}}h^{-1}, \ldots, h g_n^{q^{-1}}h^{-1})$.
Because the $B_n$ action preserves the number of elements in each conjugacy
class, if there are $n_i$ elements among $(g_1, \ldots, g_n)$ in conjugacy class
$c_i$, then there will also be 
$n_i$ elements among 
$(h g_1^{q^{-1}}h^{-1}, \ldots, h g_n^{q^{-1}}h^{-1})$
in conjugacy class $c_i$.
\end{proof}

\begin{definition}
	\label{definition:powering}
	For 
	$(g_1, \ldots, g_n) \in c^n$
	as in \autoref{lemma:g-irred-components}(2),
	there is a corresponding irreducible component of 
	$W \subset \cquohur G K {n} c {\overline{\mathbb F}_q}$
	and we say $W$ is {\em indexed by} 
	$[g_1] \cdots [g_n]$.
	We note that $W$ is also indexed by
	$[\kappa^{-1} g_1 \kappa] \cdots [\kappa^{-1} g_n \kappa]$
	for any $\kappa \in K$.
	In the case $W \subset \cquohur G K {n} c {{\mathbb F}_q}$,
	we say $W$ is indexed by 
	$[g_1] \cdots [g_n]$ if any of the geometrically irreducible components
	of $\cquohur G K {n} c {\overline{\mathbb F}_q}$ mapping to $W$ are
	indexed by $[g_1] \cdots [g_n]$.

	We say such a component has {\em boundary monodromy in the $K$ orbit of
	$h \in G$} if
	$g_1 \cdots g_n = h$.
	We assume $c$ is closed under the $q$th powering, which is the operation
	sending $x \in c$ to $x^q$.
	We also assume that $q$th powering is a bijection on $c$, and we let
	$q^{-1}$ powering denote the inverse map to $q$th powering.
	We define the $q^{-1}$ {\em powering action} on components by
	$[g_1] \cdots [g_n] \mapsto [g_1^{q^{-1}}] \cdots [g_n^{q^{-1}}]$.
\end{definition}

\begin{notation}
	\label{notation:component-maps}
	For $c$ a rack and $g_1, \ldots, g_j \in c$, we use $[g_1] \cdots [g_j]
	: \phurc n c \to \phurc {n+j} c$
	for the map on hurwitz space induced by the maps $c^n/ B_n \to
	c^{n+j}/B_{n+j}$ given by sending $(h_1, \ldots, h_n) \mapsto (h_1, \ldots,
	h_n,g_1, \ldots, g_j)$.
	In particular, if $j = 1$ this just sends 
	$(h_1, \ldots, h_n) \mapsto (h_1, \ldots, h_n,g_1)$.
	By abuse of notation, for $\ell$ an auxiliary prime, we also use
	$[g_1] \cdots [g_j] : H^i(\phurc {n+j} c, \mathbb Q_\ell) \to H^i(\phurc
	n c, \mathbb Q_\ell)$ to denote the map induced on cohomology by $[g_1]
	\cdots [g_j]$.

	If $c \subset G$ is a union of conjugacy classes in a finite group,
	$G \subset G'$ is a normal subgroup of a finite group $G'$ and $K
	\subset G$, then we use
	$\sum_{\kappa \in K} [\kappa^{-1} g_1 \kappa] \cdots [\kappa^{-1} g_j \kappa]$
	to denote the map 	$\sum_{\kappa \in K} [\kappa^{-1} g_1 \kappa] \cdots [\kappa^{-1} g_j
	\kappa]:
	H^i(\quohur {G}{K} {n+j} c {\mathbb C}, \mathbb Q_\ell) \to H^i(\quohur
	{G}{K} {n} c {\mathbb C}, \mathbb Q_\ell)$
on cohomology induced via transfer along the quotient
by $K$.
\end{notation}

\subsection{Hurwitz spaces over $\mathbb P^1$}
\label{subsection:hurwitz-p1}
We mostly will work with Hurwitz spaces
over $\mathbb A^1$, but the Picard rank conjecture concerns Hurwitz spaces over
$\mathbb P^1$. We now introduce some notation, almost exclusively used in
\autoref{section:picard}, to
describe different Hurwitz spaces we will work with.

\begin{notation}
	\label{notation:hurwitz-picard}
	We now fix a finite group $G$ and a union of conjugacy classes $c \subset G$.
	Let $\beta \subset G$ be a subset closed under $G$ conjugation.
	In the case that the base scheme $B = \spec \mathbb C$ we will use 
	$\hura G n c {\beta}$ to denote the union of components of $\quohur {G}{G}n
	c{B}$ whose boundary monodromy lies in the $G$ conjugation orbits $\beta \subset G$. (See
	\autoref{definition:hurwitz-stack-over-b} and \autoref{definition:powering}.)
	Recall this parameterizes $G$-covers of $\mathbb P^1$ whose branch locus
	over $\mathbb A^1 \subset \mathbb P^1$ has degree $n$ and inertia in
	$c$, and the inertia over $\infty \in \mathbb P^1$ lies in $\beta$.
We let $\hurp G n c$ denote the Hurwitz stack of pointed $G$ covers of $\mathbb
P^1$,
branched over a degree $n$ divisor over $\mathbb P^1$, 
with all inertia above this divisor lying in $c$.
There is further an action of $\pgl_2$ on $\hurp G n c$, acting via
automorphisms of $\mathbb P^1$, and we let $\hurz G n c$ denote the quotient of
$\hurp G n c$ by this action.
We let
$\chura G n c {\beta}, \churp G n c, \churz G n c$ denote the connected components of
$\hura G n c {\beta}, \hurp G n c, \hurz G n c$ parameterizing such covers which are
geometrically connected.
Finally, we let $\churzb G n c$ denote the Abramovich-Corti-Vistoli compactification of
$\churz G n c$. 
This is defined in \cite[\S2.2]{abramovichCV:twisted-bundles}, where it is
denoted $\mathcal K^{\on{bal}}_{g,n}(\mathscr BG,0),$
where $g$ is the genus of the covering curve.
\end{notation}
\begin{remark}
	\label{remark:compactification}
	For $n \geq 3$, the Deligne-Mumford stack 
	$\churzb G n c$
	is smooth and proper by \cite[Theorem
3.0.2 and Corollary 3.0.5]{abramovichCV:twisted-bundles}.
\end{remark}
%

\section{Stability of a quotient}\label{section:stabquot}

The goal of this section is to prove a weak form of homological stability for
Hurwitz spaces. Namely, for $c$ a finite rack, we show that the homology of
$\CHur^c$ exhibits homological stability with respect to stabilization by an
element $x$ in a component of $c$, after 
$\CHur^c$ has been quotiented by all other elements of $c$ in the component of
$x$.

\begin{remark}[Comparison to \cite{randal-williams:homology-of-hurwitz-spaces}]
	\label{remark:}
	Our argument is inspired from the stability arguments in \cite{randal-williams:homology-of-hurwitz-spaces}, though we depart from it in several places. 
One key difference is that the heart of the argument in \cite{randal-williams:homology-of-hurwitz-spaces}
is based on a regularity theorem proven in \cite[Theorem
7.1]{randal-williams:homology-of-hurwitz-spaces}, which already was a
	substantial simplification of the proof given in \cite[Theorem
4.2]{EllenbergVW:cohenLenstra}.
The regularity theorem is a statement about modules over $\pi_0R$ for an
$\EE_1$-ring $R$, and a key step in the proof presented in
\cite{randal-williams:homology-of-hurwitz-spaces} is to switch between
considering $\pi_iR$ as a module over $R$ and $\pi_0R$. 
Randal-Williams then uses the regularity theorem to obtain connectivity estimates for $\pi_iR\otimes_{\pi_0R}k$. 
In contrast, our proof avoids considering $\pi_iR\otimes_{\pi_0R}k$ altogether, and makes no use of a regularity theorem. 
We directly prove connectivity estimates for an
$R$-module $M$ which is annihilated by a sufficiently large power of the
augmentation ideal of $\pi_0R$, instead of proving a statement about $R$ itself. Because of this, to see that $\pi_iM$ is bounded in grading, it is enough to see that $\pi_iM$ is generated as a $\pi_0R$-module in a bounded range of degrees.
\end{remark}

In \autoref{subsec:stabe1alg}, we prove a general stability result about $\EE_1$-algebras, which we intend to apply to Hurwitz spaces.
Then, in \autoref{subsec:hurwitzstabquotient}, we apply these results to Hurwitz spaces to prove a 
weak form of homological stability.

\subsection{Homological stability for $\EE_1$-algebras}\label{subsec:stabe1alg}

The goal of this subsection is to prove a general homological stability result
for $\EE_1$-algebras in \autoref{theorem:e1stability}. We work over a base commutative ring $k$, and use $\Mod(k)$ to refer to the symmetric monoidal $\infty$-category of $k$-modules in spectra.\footnote{The arguments of this section are quite general, and can be made to work in a presentably monoidal stable $\infty$-category with compatible $t$-structure instead of in $\Mod(k)$. However, we do not present it in this generality as it is not needed.} This is equivalent to the derived category of $k$-modules. 
If $A$ is an $\mathbb E_1$ algebra in $\Mod(k)$,
we say $M$ is {\em an $A$ module} if $M$ is an object in the $\infty$-category of left
$k$-modules in $\Mod(k)$.

In the setting of homological stability, we work with objects graded over the
natural numbers (non-negative integers) $\NN$, i.e., in $\Mod(k)^\NN$. We use $X_i$ to refer
to the $i$th component of some $X \in \Mod(k)^{\NN}$. The symmetric monoidal
structure is given via Day convolution, meaning that $(X\otimes Y)_t= \oplus_{i+j=t}X_i\otimes Y_j$.

\begin{definition}\label{definition:connectivity}
	Let $f$ be a function $\NN \to \RR$. Given an object $X$ in
	$\Mod(k)^{\NN}$, we say that $X$ is {\em $f$-bounded} if, for all $i, j\in \NN$ with
	$j > f(i)$, we have $\pi_iX_j = 0$.
\end{definition}

A key notion that we use to capture the notion of a homological stability is the notion of being bounded in a linear range.

\begin{definition}\label{defn:boundedlinearrange}
	Let $X \in \Mod(k)^{\NN}$. We say that $X$ is \textit{bounded in a
	linear range} if there are 
	real numbers $r_1\geq0$ and $r_2$ so that $f(i)
	= r_1 i + r_2$ and $X$ is $f$-bounded. 
	In this case, we also say that
	$X$ is $f_{r_1, r_2}$-bounded.
\end{definition}

\begin{example}
	\label{example:}
	Let $X$ be a graded $\EE_1$-algebra in spaces, and $x \in \pi_0X_t$. We say that $X$ has linear homological stability
	with respect to the operator $x$ if $H_i(X_j;k) \to H_i(X_{j+t};k)$ is an isomorphism for $j>Ai+B$, for some $A,B$ with $A>0$.
	
	The relationship between this notion and \autoref{defn:boundedlinearrange} is that $X$ has linear homological stability iff $C_*(X;\ZZ)/x$ is bounded in a linear range.
\end{example}
%
%

Proving the following theorem is the goal of this subsection:
\begin{theorem}\label{theorem:e1stability}	Let $R$ be an augmented
	connective $\EE_1$-algebra in $\Mod(k)^{\NN}$ such that $\pi_0R$ is generated in degrees $\leq d$. Let $M$ be a connective $R$-module. We
	make the following assumptions.
\begin{enumerate}
		\item[(a)](Bounded cells for $R$) There are real numbers
			$v$ and $w$ so that $k\otimes_{R}k$ is
			$f_{v,w}$-bounded.
		\item[(b)](Bounded cells for $M$) $k\otimes_RM$ 
			is $f$-bounded for some nondecreasing function $f$.
		\item[(c)](Uniform torsion for $\pi_iM$) If $I \subset \pi_0R$ kernel of the augmentation $\pi_0R \to k$, and $I_{>0}$ is the part of $I$ in positive degrees, then $I$ acts nilpotently on $\pi_iM$ for each $i$, and there exists $t \in \NN$ such that $I_{>0}^{t+1}$ acts by $0$ on $\pi_iM$ for each $i$.
	\end{enumerate}
	Define $g_0 := f(0)$.
	For $i>0$, define $g_i$ inductively by 
	\begin{align*}
	g_i = \max(f(i),\max_{0\leq j<i}\left((i-j+1)v+w+dt+g_{j})\right).
	\end{align*}
	Then $\pi_iM$ is generated as a $\pi_0R$-module
in degrees $\leq g_i$. In particular, if $h(i) = g_i+td$, then $M$ is
$h$-bounded. 
Moreover, if there are constants $\mu\geq 0$ and $b$ so that $f\leq
f_{\mu,b}$, then $M$ is $f_{r_1, r_2}$ bounded, where $r_1$ and $r_2$ only depend on
$\mu, b, v, w, d$ and $t$.
\end{theorem}

In order to prove the theorem, we first a couple lemmas relating boundedness properties
of
$\pi_i(N)$ with boundedness properties of $\pi_i(k \otimes_R N)$ in both
directions. 
Later we apply this in the situation where $N = \pi_i (M)$.

\begin{lemma}\label{lemma:assumption-a-implies-tensor-bounded}
	Suppose $R$ is an augmented connective $\mathbb E_1$-algebra in
       $\on{Mod}(k)^{\mathbb N}$ such that $\pi_0 R$ is generated in degree
       $\leq d$. 
	Assume $k \otimes_R k$ is $f_{v,w}$ bounded.
	Suppose that $N$ is a discrete graded $\pi_0R$-module such that $I$ acts nilpotently on $N$, and $I^{t+1}_{>0}N=0$, and $N$ is generated in grading $\leq L$ as a $\pi_0R$-module. 
	Then 
	$k\otimes_{R}N$ is $f_{v,w+td+ L}$-bounded.
\end{lemma}
\begin{proof}
	
The assumptions that $\pi_0R$ is generated in degrees $\leq d$, $I_{>0}^{t+1}$
	acts by $0$ on $N = \pi_0N$, and $N$ is generated in grading $\leq L$
	imply that $N$ is concentrated in gradings $\leq td+L$.
	By using the finite filtration of $N$ by the powers of the augmentation ideal acting on it, the associated graded of the induced filtration on $k\otimes_RN$ is
	$(k \otimes_R k) \otimes_k \mathrm{gr} N$ where $\mathrm{gr} N$ is the associated graded of $N$. Since
	Since $k \otimes_R k$ is $f_{v,w}$ bounded by assumption,
	we obtain that  $k \otimes_R N$ is $f_{v,w + td + L}$ bounded.
\end{proof}

\begin{lemma}\label{lemma:cellstogenerators}
	Suppose that $R$ is a augmented connective $\EE_1$-algebra in $\Mod(k)^{\NN}$ and suppose that $N$ is a connective $R$-module such that $I$ acts nilpotently on $\pi_0N$
	and such that $\pi_0 (k\otimes_RN)$ 
	vanishes in gradings larger than $C$. Then $\pi_0N$ is generated as an $\pi_0R$-module in gradings $\leq C$.
\end{lemma}

\begin{proof}
The map $k\otimes_RN \to k\otimes_{\pi_0R}\pi_0N$ is an equivalence on $\pi_0$, 
so $\pi_0(k\otimes_{\pi_0R}\pi_0N)$ vanishes above grading $C$.
	Since the map $\pi_0N \to \pi_0(k\otimes_{\pi_0R}\pi_0N)$ is surjective, 
	and the target is generated in degrees $\leq C$, we can choose lifts of generators of
	$\pi_0(k\otimes_{\pi_0R}\pi_0N)$ to $\pi_0N$, and define
	$N'\subset \pi_0N$ to be the submodule generated by these lifts. 
	It suffices to show
	$N'=\pi_0N$, or equivalently $(\pi_0N)/N' = 0$.
	The map $\pi_0(k\otimes_{\pi_0R}N') \to \pi_0(k\otimes_{\pi_0R}(\pi_0N))$ is surjective by construction of $N'$, so $k\otimes_{\pi_0R}((\pi_0N)/N')$ is $1$-connective.
	This implies 
$\pi_0(k\otimes_{\pi_0R}((\pi_0N)/N')) = 0$ or equivalently
$I(\pi_0N)/N' = (\pi_0N)/N'$. Since $I$ also acts nilpotently, we conclude that $(\pi_0N)/N'=0$.
\end{proof}

We are now ready to proceed with the proof of the theorem.
\begin{proof}[Proof of \autoref{theorem:e1stability}]
	We will prove by induction on $j\geq0$ that
	
	\begin{enumerate}
		\item $\pi_jM$ is generated as a $\pi_0R$-module in degrees $\leq g_j$ and
		\item $k\otimes_{R}\pi_jM$ is
			$f_{v,w+td+g_{j}}$-bounded.
	\end{enumerate} 

We now explain why the later conclusions of the theorem follow form the above
inductive claims. First, $M$ is $h$-bounded, for $h(j) = g_j+td$,
since $I$ acts nilpotently on $\pi_jM$, $I_{>0}^{t+1}$ acts by zero on $\pi_jM$, and $R$ is generated in degrees $\leq d$. 
Moreover, it is clear from the formula defining the $g_i$ that $g_i$ is upper bounded by a function of the form $Ai+B$ if $f$ is.
	
We first prove (1) in the case $j = 0$.
By \autoref{theorem:e1stability}(b),
we have that $\pi_0((k \otimes_R M)_j )$ vanishes for $j > f(0)$.
Therefore, \autoref{lemma:cellstogenerators} implies $\pi_0M$ is generated as a $\pi_0R$-module in degrees $\leq f(0)$.

We next prove that (1) for a fixed value of $j$ implies (2) for that same value
of $j$.
	Supposing that $\pi_jM$ is generated in degrees $\leq g_j$, we apply
\autoref{lemma:assumption-a-implies-tensor-bounded}
to learn that $k\otimes_{R}\pi_jM$
	is $f_{v,w+td+g_j}$-bounded. 
	
\begin{remark}
	\label{remark:}
	For the readers familiar with \cite{EllenbergVW:cohenLenstra}, the
	remainder of the proof is a version of the spectral sequence argument in
	\cite[Theorem 6.1]{EllenbergVW:cohenLenstra}. This argument also appears
	in
	\cite[Proposition 8.1]{randal-williams:homology-of-hurwitz-spaces}.
\end{remark}

It remains to prove that if (1) and (2) both hold for values less than $j$, then
(1) also holds for $j$.
	There is a cofiber sequence
	\begin{align}
		\label{equation:cofiber-tensor}
	k\otimes_{R}\Sigma^{-j-1}\tau_{\leq j-1}M \to
	k\otimes_{R}\Sigma^{-j}\tau_{\geq j}M \to k\otimes_R\Sigma^{-j}M.
	\end{align}
	We can filter $\tau_{\leq j-1}M$ by its Postnikov tower, to see that the
	first term has a finite filtration with associated graded
	$k\otimes_{R}\Sigma^{l-j-1}\pi_lM$. Using the inductive hypothesis, we
	see that $\pi_0(k\otimes_{R}\Sigma^{l-j-1}\pi_lM)$ vanishes in degrees
	larger than $(j-l+1) v+w+td+g_{l}$. 
	It follows that $\pi_0(k\otimes_{R}\Sigma^{-j-1}\tau_{\leq j - 1}M)$ 
	vanishes in degrees larger
	than $\max_{0\leq l < j-1}(j-l+1)v+w+td+g_{l}$.
	Additionally, $\pi_0 (k\otimes_R\Sigma^{-j}M)=\pi_j(k\otimes_RM)$
	vanishes above degrees $f(j)$.
	It then follows from \eqref{equation:cofiber-tensor}
	and the definition of $g_j$
	that $\pi_0(k\otimes_{R}\Sigma^{-j}\tau_{\geq j}M)$ vanishes beyond grading $g_j$. 
	Finally, by \autoref{lemma:cellstogenerators}, this means that
	$\pi_0\Sigma^{-j}\tau_{\geq j}M = \pi_jM$ must be generated as a
	$\pi_0R$-module  in degrees $\leq g_j$, proving (1) for $j$.
\end{proof}

\subsection{The case of Hurwitz spaces}\label{subsec:hurwitzstabquotient}
We next deduce consequences for Hurwitz spaces of our general results from earlier in
this section.

\begin{notation}\label{notation:hurwitzchains}
	For a finite rack $c$, let $A_c := C_*(\Hur^{c};\ZZ)$, which is
	naturally a multigraded ring, with one grading for each component of
	$c$. That is, $A_c$ is an $\EE_1$-algebra in $\Mod(\ZZ)^{\NN^{c/c}}$. Let $CA_c$ denote $C_*(\CHur^{c};\ZZ)$, which is naturally an $A_c$-bimodule since $\CHur^c$ is a $\Hur^c$-bimodule.
\end{notation}

Because we are interested in homological stability for multigraded rings, there are many different directions in which one can consider homological stability with respect to. In the convention below, we choose to work one grading at a time, so that by proving a homological stability result with respect to each grading individually, we prove a multigraded version of homological stability. It is convenient for later to work in a slightly more general setting, where we grade with respect to a collection of components of $c$.

\begin{notation}\label{notation:gradingcomponent}
	Fix a rack $c$ and let $c' \subset c$ denote a union of components of
	$c$. Throughout this section, we consider $A$ as a (singly) graded ring,
	by using the evaluation at the components of $c'$ map $\NN^{c/c} \to
	\NN^{c'/c} \to \NN$ to restrict to a single grading, where the first map
	is the projection and the second map is given by summing the
	coordinates.
\end{notation}

In order to state our desired weak form of stability in
\autoref{theorem:quotstability}, we 
first recall that an $\EE_2$-central element in an $\EE_1$-algebra is an element
in the homotopy ring such that left and right multiplication by it refines to a bimodule map.
The following is obtained from
by applying $C_*(-; \mathbb Z)$ to the bimodule map produced in 
\cite[Lemma 4.3.1]{landesmanL:the-stable-homology-of-non-splitting}.

\begin{lemma}\label{lemma:centrality}
	If $c$ is a finite rack and $\gamma \in \pi_0\Hur^c$ is central, then the associated class $\gamma \in \pi_0A$ is $\EE_2$-central. 
\end{lemma}

Given $x \in c$, 
we use $\alpha_x$ to denote the class
associated to $x$ in  $\pi_0A$.
We define $\ord(x)$ to be the order of the operator
$x\triangleright: c \to c$ given by $y \mapsto x \triangleright y$, so that $\alpha_x^{\ord(x)} \in \pi_0\Hur^c$
is central. It follows from \autoref{lemma:centrality} that $\alpha_x^{\ord(x)}$ is
$\EE_2$-central. Thus we may form the $A$-bimodule $A/(\alpha_x^{\ord(x)})$.
Occasionally, to emphasize the dependence on $c$, we may denote $\ord(x)$ by
$\ord_c(x)$.

\begin{definition}\label{definition:quotient}
	Choose a total ordering $x_1,\dots,x_{|c|}$ on the finite rack $c$ such
	that the $x_i$ in $c'$ appear first. 
	We define $A/(\alpha_c^{\ord_c(c)})$ to be the unital $A$-bimodule defined as the tensor product
	
	$$A/\alpha_{x_1}^{\ord_c(x_1)}\otimes_{A}A/\alpha_{x_2}^{\ord_c(x_2)}\dots\otimes_{A}A/\alpha_{x_{|c|}}^{\ord_c(x_{|c|})}.$$
	
	For $c'' \subset c$, with $c'' = \{x_{i_1}, \ldots, x_{i_{|c''|}}\}$, we also define
	$$A/(\alpha_{c'}^{\ord_c(c')}) := A/\alpha_{x_{i_1}}^{\ord_c(x_{i_1})}\otimes_{A}A/\alpha_{x_{i_2}}^{\ord_c(x_{i_2})}\dots\otimes_{A}A/\alpha_{x_{i_{|c''|}}}^{\ord_c(x_{i_{|c''|}})}.$$
\end{definition}
\begin{remark}
	\label{remark:}
	Technically speaking, the construction
$A/(\alpha_c^{\ord_c(c)})$
depends on a choice of ordering of the elements in $c$. This choice of ordering will not play a
substantial role in our proof and so we omit it from the notation.
More precisely, if we choose a different ordering, the proof of our main
homological stability theorems will go through
with possibly different
constants $I$ and $J$ from
\autoref{theorem:some-large-homology-stabilizes}.
\end{remark}

Our goal in this subsection is to prove the following result:

\begin{theorem}\label{theorem:quotstability}
	Let $c$ be a rack, $c' \subset c$ be a subrack which is a union of
	components of $c$, and let $A := C_*(\Hur^c;\mathbb Z)$ viewed as a
	graded ring as in \autoref{notation:gradingcomponent}, so that the
	grading keeps track of the number of labeled points in $c'$.
	Then
	$A/(\alpha_c^{\ord_c(c)})$ 
		is $f_{\mu,b}$ bounded, where $\mu$ and $b$ only depend on $|c'|$ and
		$\max_{x \in c'} \ord_c(x)$.
\end{theorem}

Our goal will be to apply \autoref{theorem:e1stability} to the ring $R = A$ and
module $M = A/(\alpha_c^{\ord_c(c)})$. 
We begin by checking condition $(c)$ from
\autoref{theorem:e1stability}.
\begin{lemma}\label{lemma:Jnilpotent}
	If $I$ is the augmentation ideal of $\pi_0R$, then left multiplication
	by $I^{1+\sum_{i=1}^{|c|}(2^{i} \ord(x_i)-1)}$ acts by $0$ on
	$(A/(\alpha_c^{\ord_c(c)}))$. Moreover if $I_{>0}$ is the part of the ideal in nonnegative degrees, then $I_{>0}^{1+\sum_{i=1}^{|c'|}(2^{i} \ord(y_i)-1)}=0$.
\end{lemma}

\begin{proof}
	We first claim that $\alpha_{x_i}^{2^{i}\ord(x_i)}$ acts by $0$ on
	$A/(\alpha_c^{\ord_c(c)})$. Indeed $\alpha_{x_i}^{2\ord(x_i)}$ acts by $0$ on
	$$A/\alpha_{x_i}^{\ord(x_i)}\otimes_{A}A/\alpha_{x_{i+1}}^{\ord(x_{i+1})}\dots\otimes_{A}A/\alpha_{x_{|c|}}^{\ord(x_{|c|})}$$
	by \cite[Lemma 3.5.2]{landesmanL:the-stable-homology-of-non-splitting}.
	By applying \cite[Lemma
	3.5.1]{landesmanL:the-stable-homology-of-non-splitting}, we find by
	descending induction on $j$ with $i \geq j\geq 1$ that
	$\alpha_{x_i}^{2^{i+1-j}\ord(x_i)}$ acts by $0$ on
	$$A/\alpha_{x_j}^{\ord(x_j)}\otimes_{A}A/\alpha_{x_{j+1}}^{\ord(x_{j+1})}\dots\otimes_{A}A/\alpha_{x_{|c|}}^{\ord(x_{|c|})}.$$
	The claim is the case that $j=1$.
	
	We conclude by showing the above claim implies that
	$I^{1+\sum_{i=1}^{|c|}(2^{i} \ord(x_i)-1)}$ acts by $0$.
	Any sequence $(\alpha_{x_1}, \ldots, \alpha_{x_\ell})$ with
	for $\ell > \sum_{i=1}^{|c|}(2^{i} \ord(x_i)-1)$, 
	must contain
	at least $2^i\ord(x_i)$ copies of $\alpha_{x_i}$ for some $i$, by the pigeonhole
	principle. Using the action of the braid group, we see that such a
	sequence is divisible by $\alpha_{x_i}^{2^i \ord(x_i)}$ for some $i$, 
	and hence acts by $0$ on $A/(\alpha_c^{\ord_c(c)})$. Any element of $\pi_0A$ is a linear combination of such sequences, so this proves the claim. 

	We claim the same argument shows that $I_{>0}^{1+\sum_{i=1}^{|c'|}(2^{i}
	\ord(y_i)-1)}=0$. Indeed, this follows upon noting that monomials appearing in elements of
	$I_{>0}$ must contain some element of $c'$ and our assumption from
	\autoref{definition:quotient} that the elements of $c'$ appear first
	among the $x_1, \ldots, x_{|c|}$.
\end{proof}

The following lemma is well known in the case $c$ is a conjugacy class of a group (see for example \cite[Theorem 6.2]{randal-williams:homology-of-hurwitz-spaces}). We outline the argument in general below.
\begin{lemma}\label{lemma:e1cellshurwitz}
	$\ZZ\otimes_A\ZZ$ is $f_{1,0}$-bounded.
\end{lemma}
\begin{proof}
	Note that
	$\ZZ\otimes_A\ZZ$ is obtained by applying the reduced $\NN$-graded
	chains functor to the bar construction in pointed $\NN$-graded spaces
	$*_+\otimes_{\Hur^c_+}*_+$,\footnote{Here $*_+$ denotes the unit in $\NN$-graded spaces, which has a single point other than the base point in grading $0$, and just the base point in positive gradings.} where the map of  $\EE_1$-algebras in $\NN$-graded pointed spaces $\Hur^{c}_+ \to *_+$ sends all of the nonidentity components to the base point. 
	The proof of \cite[Theorem
	A.4.9]{landesmanL:the-stable-homology-of-non-splitting} shows that this
	bar construction can be modeled by the ind-weak homotopy type of a
	quotient of the family of graded spaces denoted
	$\overline{Q}_\epsilon[*_+,\Hur^{c},*_+]$ in \cite[Definition A.4.1]{landesmanL:the-stable-homology-of-non-splitting}, 
	where if one of the boundary labels is the base point point, we identify the point with the base in its grading.
	However, this ind-weak homotopy type by inspection 
	is homotopy
	equivalent to the free associative algebra in graded pointed spaces on a
	wedge of circles 
	indexed over $c$, where the circle is in grading $1$ if
	it is in $c'$, and the circle is in grading $0$ otherwise.
	Thus the reduced homology of
$\overline{Q}_\epsilon[*_+,\Hur^{c},*_+]$, which is $\ZZ\otimes_A\ZZ$, is a tensor
	algebra over $\ZZ$ on classes indexed by $c$ in topological degree $1$,
	that are in grading $1$ if in $c' \subset c$ and in
	degree $0$ otherwise.
	It follows that
$\ZZ\otimes_A\ZZ$
	is $f_{1,0}$-bounded.
\end{proof}

We finish the proof of the main theorem.

\begin{proof}[Proof of \autoref{theorem:quotstability}]
	We apply \autoref{theorem:e1stability} in the case that $k = \ZZ$, $R = A$,
	and $M = A/(\alpha_c^{\ord_c(c)})$. Note $\pi_0R$ is generated in degrees
	$\leq 1$. We will verify conditions \autoref{theorem:e1stability}(a), (b),
	and (c), with parameters 
	$v, w, d, t, \mu, b$ from \autoref{theorem:e1stability} only 
	depending on $|c'|, \max_{x \in c'} \ord_c(x)$.
	First $\pi_0(A)$ is generated in degree $1$ by the elements of $c$, so
	we can take $d = 1$ in \autoref{theorem:e1stability}.
	Condition (a) is the content of \autoref{lemma:e1cellshurwitz},
	which moreover shows we can take $v = 1$ and $u = 0$ in 
	\autoref{theorem:e1stability}.
	Condition (b) follows from the fact
	that $M$ is generated under colimits by $A/(\alpha_{c'}^{\ord_c(c')})$, 
	which has finitely many cells only depending on $|c'|$ and $\max_{x \in
	c'} \ord_c(x)$.
	Finally,
	condition (c) follows from \autoref{lemma:Jnilpotent}, which also gives
	an explicit bound on $t$ depending only on $|c'|$ and $\max_{x \in c'}
	\ord_c(x)$.
\end{proof}

\section{Algebra preliminaries}
\label{section:algebra-preliminaries}

Our goal in this section is to prove \autoref{prop:rigiditymodules}, which gives
a criterion for checking a map is an equivalence, which we will use to verify
homological stability of Hurwitz spaces.

Recall that given a cosimplicial object $\Delta \to C$, we use $\Tot$ to denote the limit of this diagram in $C$ and $\Tot^n$ to denote the limit of the restriction of this diagram to $\Delta_{\leq n}$.

\begin{definition}\label{definition:nilpotentcomplete}
	Let $f:R\to S$ be a map of $\EE_1$-rings. We say that a left $R$-module $M$ is \textit{$f$-nilpotent complete} if the natural map
	
	$$M \to \Tot(S^{\otimes_{R}{\bullet+1}}\otimes_R M)$$ is an equivalence. The target of the map is called the $f$-nilpotent completion of $M$.
	
	If $R$ is a connective $\EE_1$-ring, and $I \subset \pi_0R$ is a
	two-sided ideal, we way that a left $R$-module $M$ is
	\textit{$I$-nilpotent complete} if it is nilpotent complete with respect
	to the map $R \to (\pi_0R)/I$.
\end{definition}

The proof of \cite[Proposition 2.14]{mathew2017nilpotence}\footnote{The cited reference proves the result in symmetric monoidal categories, but the proof works in general by replacing sets with linearly ordered sets in appropriate places.} shows 
the following well known result:

\begin{lemma}
	Let $C$ be an exactly\footnote{This means that the tensor product
	commutes with finite colimits in each variable.} monoidal stable $\infty$-category with unit $\unit$,
	and let $A\in \Alg(C)$. Then the tower 

	$$\cdots \fib(\unit \to \Tot^{n}(A^{\otimes \bullet+1})) \to \fib(\unit \to \Tot^{n-1}(A^{\otimes \bullet+1})) \to \cdots \to \fib(\unit \to \Tot^0(A^{\otimes \bullet+1})) $$ is equivalent to the tower
	$$\cdots I^{\otimes n+1} \to I^{\otimes n} \to \cdots \to I $$
where $I$ is the fiber of the map $\unit \to A$ and the maps are given by tensoring with the map $I \to A$ on one\footnote{In fact, the towers obtained from different choices of tensor factors are equivalent, because the two maps 
$I\otimes I \to I$ given by the inclusion of either factor into the unit are isomorphic in the slice category of $C$ over $I$.} of the tensor factors.
\end{lemma}

We apply the above lemma in case of an $\EE_1$-algebra $R \to S$, by viewing $S$ as an $\EE_1$-algebra in the category of $R$-bimodules, to get the lemma below:

\begin{lemma}\label{lemma:descenttower}
	Given an $\EE_1$-algebra map $R \to S$, we can identify the tower of $R$-bimodules $\fib (R \to \Tot^n(S^{\otimes_{R}{\bullet+1}}))$ with the tower $I^{\otimes_{R}n}$, where $I$ is the fiber of the map $R \to S$.
\end{lemma}

The following lemma gives a criterion for $I$-nilpotent completeness.
\begin{lemma}\label{lemma:nilpcompcrit}
	Suppose that $f:R\to R'$ is a $0$-connective (i.e., surjective on $\pi_0$)
	map of connective $\EE_1$-rings, and let $I \subset \pi_0R$ be the
	kernel of $\pi_0f$. If $M$ is a left $R$-module that is bounded below
	such that $I$ acts nilpotently on each $\pi_iM$, then $M$ is $f$-nilpotent complete.
\end{lemma}

\begin{proof}
	\autoref{lemma:descenttower} describes the tower $\fib (R \to \Tot^n(S^{\otimes_{R}{\bullet+1}}))$ as $I^{\otimes n+1}$, which is connective by the assumption that $f$ is $0$-connective. Since the $f$-nilpotent completion of an $R$-module $N$ is $\lim_n\Tot^n(S^{\otimes_{R}{\bullet+1}}\otimes_RN))$, we see that the map $N \to \Tot(S^{\otimes_R\bullet + 1}\otimes_RN)$ is $i$-connective if $N$ is, which in particular implies that the $f$-nilpotent completion functor preserves $i$-connectivity.
	It thus suffices to show that $\tau_{\leq n}M$ is $f$-nilpotent complete for any $n$. Because $f$-nilpotent complete objects are closed under extensions, we can assume that $M$ is a discrete $R$-module with $I$ acting by $0$.
	Now the map $I\otimes_RM \to M$ is null, so each map in the tower
	$I^{\otimes_Rn}\otimes_RM$ is zero. Thus  $\lim_nI^{\otimes
	n+1}\otimes_RM = 0$. By \autoref{lemma:descenttower},
$\lim_nI^{\otimes n+1}\otimes_RM$
	is the fiber of the map $M \to \Tot(S^{\otimes R^{\bullet + 1}}\otimes_RM)$. Thus 
	$M$ is $f$-nilpotent complete.
\end{proof}

The following is the main result of this section, which is a generalization of
a combination of 
\cite[Lemma 3.4.3 and Proposition
3.4.2]{landesmanL:the-stable-homology-of-non-splitting}.

\begin{proposition}\label{prop:rigiditymodules}
	Let $f:R\to S$ be a map of connective $\EE_1$-rings that is surjective on $\pi_0$.  Let $I_S$ be a two-sided ideal of $\pi_0S$, and let $I_R$ be the pullback of this to a two-sided ideal of $\pi_0R$. Let $M$ be a bounded below left $R$-module such that $I_R$ acts nilpotently on each $\pi_iM$. If $$((\pi_0R)/I_R\otimes_R(\pi_0R)/I_R) \to ((\pi_0S)/I_S\otimes_S(\pi_0S)/I_S)$$ is an isomorphism, then the map
	$$M \to S \otimes_R M$$ is an equivalence.
\end{proposition}
\begin{remark}
	\label{remark:}
	In the above statement,
the notation
$(\pi_0R)/I_R$ and 
$(\pi_0S)/I_S$ refers to the underived quotient of a ring by an ideal.
\end{remark}

\begin{proof}
	From the Tor spectral sequence $\mathrm{Tor}_{\pi_*R}^i(\pi_*M,\pi_*S) \implies \pi_*(S\otimes_RM)$
	we see that $S\otimes_RM$ also satisfies the condition that $I_S$ acts nilpotently on each homotopy group. It follows from \autoref{lemma:nilpcompcrit} that $M$ and $S\otimes_RM$ are $f$-nilpotent complete.
	
	Thus it is enough to show that the natural maps
	
	$$((\pi_0R)/I_{R})^{\otimes_R n}\otimes_R M \to ((\pi_0S)/I_{S})^{\otimes_S n}\otimes_R M$$
	are equivalences for each $n\geq1$ because the map $M\to S\otimes_RM$ is
	obtained as a totalization of these maps. It suffices to show that 
	\begin{equation}\label{equation:comparison}
		((\pi_0R)/I_{R})^{\otimes_R n} \to ((\pi_0S)/I_{S})^{\otimes_S n}
	\end{equation}
	are equivalences. For $n=1$, \eqref{equation:comparison} follows from the assumption that $R\to S$ is surjective on $\pi_0$. 

	We next induct on $n$. The base case $n = 1$ was done above. For the
	inductive step, we claim that for any right $\pi_0S/I_S$-module $N$, the natural map 
	$$N\otimes_R(\pi_0R/I_R) \to N\otimes_S(\pi_0S/I_S)$$ is an equivalence. Indeed, in the case $N=\pi_0S/I_S$, this follows from assumption, and a general module is built from this under colimits and desuspensions. Taking $N=(\pi_0S/I_{S})^{\otimes_S n-1}$, we see that \ref{equation:comparison} being an equivalence follows from induction on $n$, finishing the proof.
\end{proof}
\section{Homological stability for connected Hurwitz spaces}
\label{section:homological-stability}

The goal of this section is to prove homological stability integrally for
$\CHur^{c}$ for any finite rack $c$. The weaker stability results of
\autoref{section:stabquot} show homological stability after quotienting by many operators, and we prove our result by inductively reducing the number of operators needed in quotienting.

For any subrack $c'' \subset c$, we let $A_{c''} = C_*(\Hur^{c''})$ and let
$CA_{c''}
= C_*(\CHur^{c''})$. 
As in \autoref{notation:gradingcomponent}, we view these as
graded algebras using a single component $c' \subset c$.
A key observation we use to prove homological stability is the following lemma:

\begin{lemma}\label{lemma:inductquotient}
	Suppose that $X \in \Mod(\ZZ)^{\NN}$, and $v :X \to X$ is a map of
	degree $|v|$ in the sense that it sends the
	$j$th graded part to the $(j+|v|)$th graded part for some $|v| \in
	\mathbb N$. Suppose $X/v$ is $f$-bounded and $X[v^{-1}]=0$. Then $X$ is
	$\phi$-bounded for $\phi(i) = f(i+1)-|v|$.
\end{lemma}

\begin{proof}
	We have an sequence of homotopy groups that is exact in the middle
	
	$$\pi_{i+1}(X/v)_{j+|v|}\to \pi_{i}X_j \xrightarrow{v}\pi_iX_{j+|v|}.$$
	Because $X/v$ is $f$-bounded, the multiplication 
	by $v$ is injective if $j+|v|>f(i+1)$. This implies that in such
	degrees, the map $\pi_iX_j \to\colim_n \pi_i X_{j+n|v|}= \pi_iX_j[v^{-1}]=0$ 
	is injective. Hence, $\pi_iX_j = 0$ when $j + |v| > f(i+1)$, so $X$ is $\phi$-bounded.
\end{proof}

The following lemma shows that quotienting by graded operators in degree $0$
doesn't affect whether homological stability holds.

\begin{lemma}\label{lemma:twogradings}
	Let $X \in \Mod(\ZZ)^{\NN\times \NN}$. Suppose that $v:X[0,1] \to X$ is a map, where $(Y[i,j])_{k+i,l+j} = Y_{k,l}$, interpreting $Y_{k,l}$ to be $0$ when either $k$ or $l$ is negative. Suppose that when viewed via the first grading, $X/v$ is $f$-bounded. Then $X$ is $f$-bounded with respect to the first grading.
\end{lemma}

\begin{proof}
	It follows by induction on $n\geq1$ that $X/v^n$ is $f$-bounded 
	when viewed in the first grading, because of the cofiber sequence
	$$(X/v^{n-1})[0,1] \to X/v^n \to X/v.$$
	To see that $\pi_*X$ is $f$-bounded, it suffices to do this for each
	summand of $\pi_*X$ along the second grading (since $\pi_*X$ is a direct
	sum of these). But where the second grading is $\leq n-1$, $X$ agrees
	with $X/v^{n}$ since the source of the map $v^n$ is $X[0,n]$, so is concentrated in
	degrees $\geq n$ with respect to the second grading. Thus $X$ is $f$-bounded because all of the $X/v^{n}$ are $f$-bounded.
	\end{proof}

The following proposition is the key input we need to show why the stability
result in \autoref{theorem:quotstability} implies homological stability. We recall
that for a space $X$, $X_+$ denotes the pointed space obtained by adding a
disjoint base point to $X$. For a subrack $c'' \subset c$, the normalizer
$N_c(c'')$ is defined as $\{x \in c|x\triangleright y \in c'', \forall y \in
c''\}$.

\begin{proposition}[{\cite[Proposition 4.5.11]{landesmanL:the-stable-homology-of-non-splitting}}]\label{proposition:tensorproducthomotopy}
	Let $c$ be a rack, $c'' \subset c$ be a subrack with normalizer $N_c(c'')$,  
	and $X_+ := \pi_0\Hur^{c''}[\alpha_{c''}^{-1}]_+$, viewed as a $\Hur^{c''}_+$-bimodule. 
	Then the map
	\begin{align}
		\label{equation:c-to-prime-equivalence}
		X_+\otimes_{\Hur^{c}_+}X_+ \to X_+\otimes_{\Hur^{N_c(c'')}_+}X_+ 
	\end{align}
	is a homology equivalence.
\end{proposition}

In what follows, it will be helpful to introduce a filtration on $A_c$:

\begin{construction}\label{construction:hurfil}
	Given a rack $c$, we define a decreasing filtration on $\Hur^c$ as a $\Hur^c$-bimodule, by setting $F_i\Hur^{c}$ to be the union of
	components of $\Hur^{c}$ such that if $c''$ is the smallest subrack of
	$c$ such that the component comes from $\Hur^{c''}$, then $|c''|\geq i$.
	
	We use $F_*A_c$ to denote the filtration on $A_c$ induced by taking chains.
\end{construction}
The following lemma is immediate from construction:
\begin{lemma}\label{lemma:hurassgrdd}
	There is a natural isomorphism of $A_c$-bimodules
		$$F_iA_c/F_{i+1}A_c \cong \oplus_{c'' \subset
		c,|c''|=i}CA_{c''},$$
		where each $CA_{c''}$ as in \autoref{notation:hurwitzchains}
		is given the structure of an $A_{c}$-bimodule by having elements
		of $c$ not in $c''$ act by $0$.
\end{lemma}

Note that $F_0A = A$, and that $F_iA$ for $i \geq 0$ are, in particular, graded $A$-modules.

We are now ready to prove the main result of this section, which is a way of
expressing the sense in which the homologies of Hurwitz spaces stabilize.

\begin{theorem}
	\label{theorem:spectra-homology-stabilization}
	Let $c$ be a rack. Let $c' = c'_1 \cup \cdots \cup c'_r \subset c$ be a
	disjoint union of $r$ components of $c$.
	Let $s_i \in c'_i$. Then
	there are functions 
	$\mu,b: \mathbb N^2 \to \mathbb Z$ so that
	$CA_c/(\alpha_{s_1}^{\ord(s_1)}, \ldots, \alpha_{s_r}^{\ord(s_r)})$ is $f_{\mu(\max_{1 \leq i\leq r} \ord(s_i),|c'|),b(\max_{1 \leq i\leq r} \ord(s_i),|c'|)}$ bounded.
\end{theorem}

\begin{remark}
	It seems likely the dependence of $\mu$ and $b$ on $\ord(s)$ in the proof of 
	\autoref{theorem:spectra-homology-stabilization}
	can be removed (upon slight modification of the statement). One approach to doing this is to replace
	$A/(\alpha_x^{\ord(x)},x \in S)$ in the proof with an appropriate module
	`$A/(\alpha_x,x \in S)$'. Constructing this module would involve using
	certain invertible $A$-bimodules such that each $\alpha_x$ becomes a
	bimodule map.
	We believe it would be interesting to work this out precisely.
\end{remark}

\begin{proof}
	To abbreviate notation for the course of this proof, we use the notation $\ord(s) :=
	\max_{1 \leq i \leq r} \ord(s_i)$.
	We will show by descending induction on $|S|$ the following claim:
	
	\begin{enumerate}
		\item[$\star$] For any subset $S
		\subset c$ containing some element of each component of $c$,
		$CA_c/(\alpha_x^{\ord(x)},x \in S)$ is 
		$f_{\mu(\ord(s),|c'|),b(\ord(s),|c'|)}$ bounded.
	\end{enumerate} 
	
	We next explain why the above claim implies the result.
	If this claim is proven, we will learn that for some subset $S' \subset
	c$ with $S' \cap c' = \{s_1, \ldots, s_r\}$,  $CA_c/(\alpha_x^{\ord(x)},x \in S')$ is 
$f_{\mu(\ord(s),|c'|),b(\ord(s),|c'|)}$
bounded.
Using notation that generalizes \autoref{definition:quotient}, by
$(-)/(\alpha_x^{\ord(x)},x \in S')$ we mean the iterated quotient by each
$\alpha_x^{\ord(x)}$ with respect to an implicit ordering for which all elements
in $c'$ are first in the ordering. Nothing we prove will depend on this choice of ordering, so we do not indicate it in our notation.
	
	Define a bigrading on $CA_c$ so that
a point of Hurwitz
		space lies in
		bigrading $(g_1, g_2)$ if it has $g_1$ labels in $c'$ and $g_2$
	labels in $c - c'$.)
	Then by iteratively applying \autoref{lemma:twogradings}
	to this bigrading, we learn that 
since $CA_c/(\alpha_x^{\ord(x)},x \in S')$ is 
$f_{\mu(\ord(s),|c'|),b(\ord(s),|c'|)}$
bounded,
$CA_c/(\alpha_{s_1}^{\on{ord}(s_1)}, \ldots, \alpha_{s_1}^{\on{ord}(s_1)})$ is
	also $f_{\mu(\ord(s),|c'|),b(\ord(s),|c'|)}$
	bounded. 
	Thus, in order to conclude the proof, it remains to prove the claim $\star$.
	
	Recall from \autoref{construction:hurfil} and \autoref{lemma:hurassgrdd} that
	$A_c$ has a finite decreasing filtration $F_iA_c$ by graded rings such
	that the associated graded is $$F_iA_c/F_{i+1}A_c \cong \oplus_{c'' \subset
		c,|c''|=i}CA_{c''}$$ 
		where each term $CA_{c''}$ is an $A_c$-bimodule by having each
		element not in $c''$ act by $0$.

		First, we handle the base case of our induction where $|S| = |c|$.
	Recall \autoref{theorem:quotstability}, which shows
	there are functions $\mu^0, b^0: \mathbb N^2 \to \mathbb Z$ so that
	$A_c/(\alpha_x^{\ord(x)},x \in c)$, when graded by the component $c'
	\subset c$, with $s \in c'$, is $f_{\mu^0(|c'|,\ord(s)), b^0(|c'|,
	\ord(s))}$ bounded.

	Next, define 
$\mu^1(u,t) := \max(t, \max_{u'\leq u, t'\leq t} \mu^0(u',t')), 
b^1(u,t) := \max_{u'\leq u, t' \leq t} b^0(u',t') + ut + \mu^1(u,t)$.

	We will show that $CA_c/(\alpha_x^{\ord(x)},x \in c)$ is also
	$f_{\mu^1(|c'|,\ord(s)),b^1(|c'|,\ord(s))}$-bounded, which is the
	base case of our claim $\star$. This will be proven by an induction on the size of $c$. 
	
	The associated graded piece $(F_iA_c/F_{i+1}A_c)/(\alpha_x^{\ord(x)},x \in c)$ is given by a direct sum over subsets $c'' \subset
	c$ with $|c''|=i$ of exterior algebras over
	$CA_{c''}/(\alpha_x^{\ord_c(x)},x \in c'')$ on classes in bidegree
	$(\ord(s),1)$. Since the elements outside of $c''$ act by $0$ on this,
	and for $x \in c''$, the order of $x$ in $c''$ divides its order in $c$,
	it follows from the definition of $b^1$ that
	$CA_{c''}/(\alpha_x^{\ord_{c''}(x)},x
	\in c)$ is 
	$f_{\mu^1(|c'|,\ord(s)),b^1(|c'|,\ord(s))-\mu^1(|c'|, \ord(s))}$ bounded.
	Now, since the cofiber $Q$ of
	\begin{align}
		\label{equation:connected-to-all-cofiber}
		CA_c/(\alpha_x^{\ord(x)},x \in c)\to A_c/(\alpha_x^{\ord(x)},x \in c)
	\end{align}
	has a finite filtration with associated graded pieces which are sums of
$CA_{c''}/(\alpha_x^{\ord_{c''}(x)},x
\in c)$ for $c'' \subset c$,
we find $Q$ is also
$f_{\mu^1(|c'|,\ord(s)),b^1(|c'|,\ord(s))-\mu^1(|c'|, \ord(s))}$ bounded.
This that the fiber of \eqref{equation:connected-to-all-cofiber}, $\Sigma^{-1} Q$,
is $f_{\mu^1(|c'|,\ord(s)),b^1(|c'|,\ord(s))}$ bounded (since if we shift the
	line of slope $\mu$ and intercept $b-\mu$ to the left by $1$, we obtain
the line of slope $\mu$ and intercept $b$).
As 
$A_c/(\alpha_x^{\ord(x)},x \in c)$ is also 
$f_{\mu^1(|c'|,\ord(s)),b^1(|c'|,\ord(s))}$ bounded,
we obtain that 
$CA_c/(\alpha_x^{\ord(x)},x \in c)$ is 
$f_{\mu^1(|c'|,\ord(s)),b^1(|c'|,\ord(s))}$ bounded as well.

	We next wish to tackle the inductive step of the claim $\star$.
	Having established the base case that $|S| = |c|$, we next suppose that
	$CA_{c}/(\alpha^{\ord(x)}_x, x \in S')$ is 
	\begin{align*}
	f_{\mu^1(|c'|,\ord(s)),b^1(|c'|,\ord(s))+ |(c - S')\cap c'| \cdot
\mu^1(|c'|,\ord(s))}
	\end{align*}
bounded	for all $S'$ with $|S'| > |S|$ and verify that 
$CA_{c}/(\alpha^{\ord(x)}_x, x \in S)$ is 
\begin{align*}
	f_{\mu^1(|c'|,\ord(s)),b^1(|c'|,\ord(s))+|(c - S)\cap c'| \cdot
\mu^1(|c'|,\ord(s))}
\end{align*}
 bounded.
	By 
\autoref{lemma:inductquotient}
	(which we use to remove elements in $c'$ from the quotient), 
	and 
\autoref{lemma:twogradings}
	(which we use to remove elements in $c - c'$ from the quotient), 
	it
	suffices to show 
	$CA_{c}/(\alpha^{\ord(x)}_x, x \in S)[\alpha_y^{-1}] = 0$ for each $y
	\in c- S$.
\begin{align*}
	\mu(|c'|, \ord(s)) &:= \mu^1(|c'|,\ord(s)) \\
	b(|c'|,\ord(s) &:=
	b^1(|c'|,\ord(s))+|c'| \cdot \mu^1(|c'|,\ord(s)).
\end{align*}

%
As mentioned above, in the remaining part of the inductive step, it is enough
	to show that $CA_c/(\alpha_x^{\ord(x)},x \in S)[\alpha_y^{-1}]=0$, where
	$y \in c-S$.
	Note that here and in what follows, we freely use that inverting $\alpha_y$ commutes with tensoring with quotients \cite[Lemma 3.4.4]{landesmanL:the-stable-homology-of-non-splitting}; 
	here, \cite[Lemma 3.4.4]{landesmanL:the-stable-homology-of-non-splitting}
	applies because $\alpha_x^{\ord(x)}$ is $\EE_2$-central (\autoref{lemma:centrality}), and inverting a central element is base changing along a homological epimorphism (by \cite[Remark 3.3.2]{landesmanL:the-stable-homology-of-non-splitting}, the localized ring, which is always homological epimorphism by \cite[Example 3.3.1]{landesmanL:the-stable-homology-of-non-splitting}, is computed as the colimit along multiplication by $r$).
		
	By the inductive hypothesis, 
	$CA_c/(\alpha_x^{\ord(x)},x \in
	S)[\alpha_y^{-1}]/(\alpha_z^{\ord(z)})=0$ for each $z \in c-S-\{y\}$. 
	Thus by applying
	\autoref{lemma:inductquotient} and iteratively applying \cite[Lemma
	3.3.4]{landesmanL:the-stable-homology-of-non-splitting}, we learn that
	it is enough to show that $CA_c/(\alpha_x^{\ord(x)},x \in
	S)[\alpha_x^{-1}, x \in c-S]=0$. 

	First, suppose $c-S$ is not a subrack of $c$, then there is $x,y
	\in c-S$ with $x\triangleright y \in S$. But then since $\alpha_x
	\alpha_y =
	(\alpha_{x\triangleright y}) \alpha_x$ in $\pi_0\Hur^c$,
	$\alpha_{x\triangleright y}$ acts
	nilpotently and invertibly on $CA_c/(\alpha_x^{\ord(x)},x \in S)[\alpha_x^{-1}, x \in c-S]$, so 
	$CA_c/(\alpha_x^{\ord(x)},x \in S)[\alpha_x^{-1}, x \in c-S]= 0$.
	Thus we may assume that $c-S$ is a nonempty subrack of $c$.
	
	Recall that we write $N_c(c-S)$ for the normalizer of $c-S$ in $c$.
	Assuming $c - S \subset c$ is a nonempty subrack,
	we claim that the map 
	\begin{align}
		\label{equation:quotient-and-invert}
	f_{c,c-S}:A_c/(\alpha_x^{\ord(x)},x \in
	S)[\alpha_x^{-1}, x \in c-S] \to A_{N_c(c-S)}/(\alpha_x^{\ord(x)},x \in
	S)[\alpha_x^{-1}, x \in c-S]
	\end{align}
	is an equivalence.
	To see \eqref{equation:quotient-and-invert} is an equivalence, we will
	first show that the kernel of the map $\pi_0A_c[\alpha_x^{-1}, x \in c-S] \to
	\pi_0A_{c-S}[\alpha_x^{-1}, x \in c-S]$, the right ideal generated by $\alpha_x, x \in S$,
	acts nilpotently on the source and target. First,
 by \cite[Lemma 3.5.1, Lemma 3.5.2]{landesmanL:the-stable-homology-of-non-splitting}, each element $\alpha_x$ for $x \in S$ acts nilpotently on the source and target.
 
 An arbitrary element $w$ in the kernel is of the form $\sum_{x \in S}\alpha_x y_x$ for some elements $y_x$. Using the pigeonhole principle and the fact that $\alpha_x y = \phi_x(y)\alpha_x$ where $\phi_x$ is the automorphism that $x$ induces on $c$, we see that for any $i>0$, there is an $N$ such that $w^N$ is in the right ideal generated by $\alpha_x^i$ for $x \in S$. It follows that $w$ acts nilpotently, since each of the $\alpha_x$ does.

	We will apply \autoref{prop:rigiditymodules} to the map of rings
	$A_c[\alpha_x^{-1}, x \in c-S] \to
	A_{N_c(c-S)}[\alpha_x^{-1}, x \in c-S]$, the ideal $I$ generated by $\alpha_x$ for $x \in S$, and $M=A_c/(\alpha_x^{\ord(x)},x \in
	S)[\alpha_x^{-1}, x \in c-S]$. Note that the map in the Proposition is then exactly the map of interest \eqref{equation:quotient-and-invert}. The hypothesis of $I$ acting nilpotently was verified in the previous paragraph.
	Thus
	in order to show
	\eqref{equation:quotient-and-invert} is an equivalence, it is enough to see that the map 
	\begin{align*}
	&\pi_0A_{c-S}[\alpha_x^{-1}, x \in c-S] \otimes_{A_c[\alpha_x^{-1}, x \in
	c-S]}\pi_0A_{c-S}[\alpha_x^{-1}, x \in c-S] \\
	&\to \pi_0A_{c-S}[\alpha_x^{-1}, x \in c-S]
	\otimes_{A_{N_c(c-S)}[\alpha_x^{-1}, x \in
	c-S]}\pi_0A_{c-S}[\alpha_x^{-1}, x \in c-S]
	\end{align*}
is an equivalence. But this follows from applying reduced chains to the result of 
	 \autoref{proposition:tensorproducthomotopy} above.
	
	The map $f_{c,c-S}$ is compatible with the filtration from \autoref{construction:hurfil}. The $i$th associated graded piece of the source and target are
	\begin{equation}
	\begin{aligned}
		\label{equation:filtration-quotient-and-invert}
	&\oplus_{c'' \subset c,|c''|=i}CA_{c''}/(\alpha_x^{\ord(x)},x \in
	S)[\alpha_x^{-1}, x \in c-S] \\
	&\to \oplus_{c'' \subset
	N_c(c-S),|c''|=i}CA_{c''}/(\alpha_x^{\ord(x)},x \in S)[\alpha_x^{-1}, x
	\in c-S].
	\end{aligned}
	\end{equation}

	Note that all terms of the above sum where $c''$ doesn't contain $c-S$
	are zero, because then some $\alpha_x$ for $x \in c-S$ both acts
	invertibly and by zero. In degree $|c-S|$, both sides agree then, since
	they are both $CA_{c-S}/(\alpha_x^{\ord(x)},x \in S)[\alpha_x^{-1}, x
	\in c-S]$.
	By induction on $|c|$, we may assume that
	$CA_{c''}/(\alpha_x^{\ord(x)},x \in S)[\alpha_x^{-1}, x \in c-S]$
	vanishes for all $c''$ strictly smaller than $c$ but strictly containing
	$c-S$. It follows that the
	map $f_{c,c-S}$ is an isomorphism on all associated graded pieces of
	filtration from \autoref{construction:hurfil} below degree $|c|$. Thus
	since $f_{c,c-S}$ is an isomorphism and the filtration from \autoref{construction:hurfil}
	is finite, it follows that $f_{c,c-S}$ is also an isomorphism in associated graded degree $|c|$.
	
	The source of \eqref{equation:filtration-quotient-and-invert} in
	associated graded degree $|c|$ is $CA_c/(\alpha_x^{\ord(x)},x \in
	S)[\alpha_x^{-1}, x \in c-S]$. To complete the inductive step, we now wish to show
	the source of \eqref{equation:filtration-quotient-and-invert} vanishes.
	Since \eqref{equation:quotient-and-invert} is an equivalence,
	it suffices to show the target of
	\eqref{equation:filtration-quotient-and-invert} is $0$ in associated graded degree $i=|c|$. By our inductive hypothesis, the target of this map vanishes in degrees smaller than $i$. That is, it
	suffices to show $N_c(c-S) \neq c$, because then $N_c(c-S)$ has no subracks of size $|c|$.

	Note that $c-S$ is not a union of components of $c$,
	because $S$ contains some element from each component of $c$ by
	assumption.
	Since $c-S$ is not a union of components of $c$, there is some $x
	\in c-S$ and $y \in c$ so that $y \triangleright x \notin c-S$. Then $y
	\triangleright x$ lies in the same component as $x$. Hence, $N_c(c-S) \neq
	c$,
	completing the proof.
\end{proof}

\subsubsection{Proof of \autoref{theorem:some-large-homology-stabilizes}}
\label{subsubsection:some-large-homology-stabilizes-proof}

Let us now explain why \autoref{theorem:spectra-homology-stabilization} implies
\autoref{theorem:some-large-homology-stabilizes}.
We recall that our goal for the readers convenience.
We fix a rack $c$ with connected components
$c_1, \ldots, c_\upsilon$.
Fix one connected component $c_\lambda$.
We aim to show there are constants $I$ and $J$, depending only on
$|c_\lambda|$ and the order of any element of $c_\lambda$
so that for any
$i$ and $\lambda$, with $i \geq 0$, 
$1 \leq \lambda \leq \upsilon$, 
and $n_\lambda > Ii + J$,
the maps $[g]$ for $g \in c_\lambda$
all induce isomorphisms
$H_i(\cphurc {n_1, \ldots, n_\upsilon} c, \mathbb Z) \to H_i(\cphurc
{n_1,\ldots, n_{\lambda-1}, n_\lambda + 1, n_{\lambda+1}, \ldots, n_\upsilon} c, \mathbb Z)$.

After reordering, we may as well take $\lambda = 1$.
We then apply \autoref{theorem:spectra-homology-stabilization} where we take
$c'$ there to be the single component $c_1$ and choose a fixed element $s_1 \in c_1$.
It then follows from \autoref{theorem:spectra-homology-stabilization}
that there are constants $I$ and $J$ as desired for which $\alpha_{s_1}^{\ord(s_1)}$
induces isomorphisms
$H_i(\cphurc {n_1, \ldots, n_\upsilon} c, \mathbb Z) \to H_i(\cphurc
{n_1,\ldots, n_{\lambda-1}, n_\lambda + \ord(s_1), n_{\lambda+1}, \ldots, n_\upsilon} c, \mathbb Z)$.
Now, if a composite of injective maps is an isomorphism, then each of those
maps is an isomorphism. Hence, writing $\alpha_{s_1}^{\ord(s_1)}$ as a composite
of $\ord(s_i)$ iterates of $\alpha_{s_1}$ we obtain the desired statement.
\qed

\section{The dominant stable homology}
\label{section:dominant-stable-homology}

In this section we compute the dominant part of the stable homology of
$\CHur^c$, i.e., the homology obtained by stabilizing with respect to all of the elements of $c$.
We first build the comparison map that we wish to show is a stable isomorphism.
The following lemma shows that, after group completion, there is no difference between $\Hur^c$ and $\CHur^c$.

\begin{lemma}\label{lemma:hurtochur}
	The natural map $B\CHur^c \to B\Hur^c$ is an equivalence. In particular, the group completions of $\CHur^c$ and $\Hur^c$ agree.
\end{lemma}

\begin{proof}
	Note that the map $\pi_0\CHur^c \to \pi_0\Hur^c$ becomes an equivalence
	on group completions, since group completion inverts $\prod_{g \in
	c}[g]^{\ord(g)}$, and the map
	$\times \prod_{g \in c}[g]^{\ord(g)} : \Hur^c \to \Hur^c$
factors through
$\CHur^c$. By the group completion theorem, the result then follows, since it implies that the map $\Omega B\Hur^{c} \to \Omega B\CHur^{c}$ is a homology equivalence, and hence an equivalence since they are loop spaces.
\end{proof}

For understanding the following definition, we remind the reader about
\autoref{definition:abelianizationrack} regarding the orbits of a rack acting on
itself.

\begin{definition}
	Given a rack $c$, we define $D^c$ to be the pullback
	
	$$B\pi_0\Hur^{c}\times_{B\pi_0\Hur^{c/c}}B\Hur^{c/c}$$
	
	There is a comparison map of $\EE_1$-algebras
	\begin{align}
		\label{equation:comparison-map}
	v_c:\CHur^c \to \Omega D^c = \Omega B\pi_0\Hur^{c}\times_{\Omega
	B\pi_0\Hur^{c/c}}\Omega B\Hur^{c/c}.
	\end{align}
\end{definition}

\begin{remark}
	Note that $\Hur^{c/c}$ is the multicolored configuration space on the components of the rack $c$. 
The comparison map above essentially replaces each component of
	Hurwitz space with the corresponding component of multicolored
	configuration space.

	Additionally,
$\Hur^{c/c}$
	is a free $\EE_2$-algebra on the set $c/c$, and its homology is completely calculated (see \cite[Section 16]{galatius2018cellular} and \cite[Section 5.1]{lawson2020n} for modern references).
	However, for the purposes of this paper, we will not actually have to
	use that this homology is known.
	\end{remark}

We will need the following lemma to know that certain groups showing up will not
contribute to homology with $\ZZ[\frac 1 {|G^0_c|}]$-coefficients.

\begin{lemma}
	\label{lemma:components-factors}
	Let $c$ be a finite rack and let $G^0_c$ denote the reduced structure group
	of $c$ as defined in \autoref{definition:abelianizationrack}.
	Let $U(c)$ denote the
	group completion of the monoid $\pi_0(\on{CHur}^c)$.
	(This is also known as the structure group of $c$.)
	There is a map $U(c) \to U(c/c)\cong \mathbb Z^{c/c}$.
	Then
	$\ker(U(c) \to \mathbb Z^{c/c})$ is finite, and any prime dividing its order also divides $|G^0_c|$.
	\end{lemma}
\begin{remark}
	\label{remark:}
	In the case that $c$ is a union of conjugacy classes in a finite group,
	\autoref{lemma:components-factors} follows from \cite[Theorem
	2.5]{wood:an-algebraic-lifting-invariant}. It seems likely that this
	should generalize to the setting of racks without much difficulty. 
	In the case that $c$ is a union of conjugacy classes, this is related to
	understanding the constant in Malle's conjecture, and so in the general
	case seems likely to be related to understanding what the
	constant term in a generalized version of Malle's conjecture to racks
	should be.
	We believe it would be interesting to work out this generalization. 
\end{remark}
The following slick proof of \autoref{lemma:components-factors} was suggested to us by Pavel Etingof.
\begin{proof}
	The group $U(c)$ is presented by the generators $[x]$ for $x \in c$ and the
	relations $[x]^{-1}[y][x] = [x\triangleright y]$ for all $x, y \in c$. Indeed, this is because 
	$\pi_0( \on{Hur}^c)$ is the monoid with the same generators and relations $[y][x] = [x][x\triangleright y]$.
	This identification of the generators and relations of the group completion of
	the Hurwitz space monoid is also spelled out in \cite[Proposition
	4.17]{shusterman:the-tamely-ramified-geometric}.
	Let $K$ denote the commutator $[U(c),U(c)]$ and let $Z := \ker (U(c) \to G^0_c)$.
	We can identify $U(c)^{\ab} \simeq \mathbb Z^{c/c}$ since it is generated by the classes
	$[x]$ for $x \in c$ with the relations that two such classes are identified if
	they lie in the same orbit.
	We have two short exact sequences
	\begin{equation}
		\label{equation:}
		\begin{tikzcd}
			0 \ar {r} & Z \ar {r} & U(c) \ar {r} & G^0_c \ar {r} & 0 \\
			0 \ar {r} & K \ar {r} & U(c) \ar {r} & U(c)^{\ab} \ar {r} & 0 
	\end{tikzcd}\end{equation}

	We first claim that $K$ is finite.
	To see this, first note that $Z$ is a finitely generated abelian group,
	since it is an abelian finite index 
	subgroup of the finitely generated (nonabelian) group $U(c)$.
	To show $K$ is finite, we only need show that $Z$ has the same $\mathbb
	Z$-rank as $U(c)^{\ab}$, or equivalently that the map $Z \to U(c)^{\ab}$
	has finite kernel and finite cokernel.
	There is a finite index subgroup $Z' \subset Z$ generated by elements of the
	form $[x]^{\on{ord}(x)}$ for all $x \in c$.
	Thus it suffices to show $Z' \to U(c)^{\ab}$ has finite kernel and
	cokernel.
	The map is in fact injective and the image is, by construction,
	the finite index subgroup of
	$U(c)^{\ab}$ generated by images of $[x]^{\on{ord}(x)}$ for $x \in c$.
	Hence $K$ is finite.

	We wish to show any prime $\ell \nmid |G^0_c|$ also does not divide
	$|K|$.
	Let $S \subset Z$ denote the $\ell$ power torsion of the finitely generated
	abelian group $Z$ and $Q := Z/S$.
	Since $U(c)$ is a central extension, it corresponds to a class $\nu \in H^2(G^0_c,
	Z) \simeq H^2(G^0_c, S) \oplus H^2(G^0_c, Q)$
	Since $G^0_c$ and $S$ have coprime orders, 
	$H^2(G^0_c, S) = 0$ so we can view $\nu$ as lying in $H^2(G^0_c, Q)$. This means that the extension $U(c)$ is
	of the form $S \times H$ where $H$ is a central extension of $G^0_c$ by $Q$.
	This means that $U(c)^{\ab} \simeq S^{\ab} \times H^{\ab}$ and so the map $S
	\to U(c)^{\ab}$ is injective.
	Now, if $g \in K$ has order $\ell$, we obtain that $g$ maps to $0$ in $G^0_c$
	because $\ell \nmid |G^0_c|$, and therefore $g \in S$. But then $S$ maps injectively
	into $U(c)^{\ab}$ while $K$ is the kernel of $U(c) \to U(c)^{\ab}$ we obtain
	that $g = \id$, implying there are no elements of order exactly $\ell$ in $K$.
	\end{proof}

The following is a generalization of \cite[Proposition
4.5.1]{landesmanL:the-stable-homology-of-non-splitting} to arbitrary finite
racks. See \cite[Remark 4.5.2]{landesmanL:the-stable-homology-of-non-splitting}
for a discussion of the history of this result.

\begin{theorem}\label{theorem:dominantstablehomology}
	Let $c$ be a finite rack. Then the map
	\begin{align}
		\label{equation:desired-stable-homology}
		H_*(\CHur^c)[\frac 1 {|G^0_c|},\alpha_c^{-1}] \to H_*(\Omega D^c)[\frac 1 {|G^0_c|},\alpha_c^{-1}]	
	\end{align}
	induced by $v_c$ from \eqref{equation:comparison-map} is an equivalence.
\end{theorem}

\begin{proof}
	
	By \autoref{lemma:hurtochur} and the group completion theorem, the map
	\eqref{equation:desired-stable-homology} is the map induced by
	$\Omega B \Hur^c \to \Omega D^c$
	on $\ZZ[\frac 1 {|G^0_c|}]$-homology.
	
		We next set the stage to reduce to showing 
	$B \Hur^c \to D^c$ is a
	rational homology equivalence.
	We say a space $X$ is {\em weakly simple} if $\pi_1 X$ acts trivially on
	$\pi_i X$ for $i > 1$.
	As a first step, we claim
	$\Omega D^c$ is weakly simple. Indeed, 
	the action of $\pi_1(\Omega D^c)$ on $\pi_i(\Omega D^c)$ for $i > 1$
	factors through the action of $\pi_1B\Hur^{c/c}$ acting on $\pi_i
	(B\Hur^{c/c})$ for $i \geq 1$. Note that $B\Hur^{c/c}$ is the classifying space of a rack.
	The claim then follows from \cite[Proposition
	5.2]{fennRS:the-rack-space},
	where it is shown
	that for any rack $X$, $BX$ is weakly simple (so $B\Hur^c$ is also weakly simple).

We may now reduce to showing
$B \Hur^c \to D^c$ is a
	rational homology equivalence.
	Note that
	the map $\pi_1 B \Hur^c \to \pi_1 D^c$ is an equivalence by construction of $D^c$. 
	Applying \cite[Lemma 4.5.4]{landesmanL:the-stable-homology-of-non-splitting} to the map
	$B \Hur^c \to D^c$, we obtain that
	$B \Hur^c \to D^c$
	being a 
	a $\ZZ[\frac 1 {|G^0_c|}]$-homology equivalence implies
	$\Omega B\Hur^{c} \to \Omega D^c$
	is a $\ZZ[\frac 1 {|G^0_c|}]$-homology equivalence.
	
	We next claim that $D^c \to B\Hur^{c/c}$ is an equivalence on $\ZZ[\frac 1 {|G^0_c|}]$-homology.
	This can be seen from the Serre spectral sequence of the fiber sequence
	$$BK \to D^c \to B\Hur^{c/c},$$
	where $K$ is the kernel of the map $U(c) \to U(c/c)$. Indeed, by
	\autoref{lemma:components-factors}, $K$ is finite of order coprime to
	$|G^0_c|$, so $BK \to *$ is a $\ZZ[\frac 1 {|G^0_c|}]$-homology equivalence.

	Recall we are trying to show $B\Hur^c \to D^c$ is a 
	$\ZZ[\frac 1 {|G^0_c|}]$-homology equivalence
	and since we have shown
	$D^c \to B\Hur^{c/c}$ is an equivalence on $\ZZ[\frac 1 {|G^0_c|}]$-homology.
	Thus it suffices to show that
	\begin{equation}\label{eqn:compare3}
		B\Hur^c \to   B\Hur^{c/c}
	\end{equation} is a $\ZZ[\frac 1 {|G^0_c|}]$-homology equivalence. This is proven in \cite[Lemma 4.5.6]{landesmanL:the-stable-homology-of-non-splitting}.
\end{proof}

\section{The asymptotic Picard Rank Conjecture}
\label{section:picard}

\subsection{Stating the Picard rank conjecture}
Recall that the Picard rank conjecture predicts the rational Picard group of
certain Hurwitz spaces over $\mathbb{C}$ is trivial. In this section, we prove the Picard rank conjecture for
covers of sufficiently large genus. 
We thus consider this an asymptotic proof of the Picard rank conjecture.

Recall that we introduced the notation
$\churz {G} n c$ in \autoref{notation:hurwitz-picard}
to mean the Hurwitz space over $\mathbb{C}$ parameterizing geometrically
connected $G$ covers of a genus $0$ curve with
inertia in $c$ and a degree $n$ branch locus.
We also recall that for $G$ a group and $c$ a conjugacy class, $H_2(G,c)$,
as defined in \cite[Definition p. 3]{wood:an-algebraic-lifting-invariant},
is defined as the quotient of the group cohomology $H_2(G;\mathbb Z)$
by the image of all maps $H_2(\mathbb Z^2, \mathbb Z) \to H_2(G;\mathbb Z)$
induced by the maps $\mathbb Z^2 \to G, (i,j) \mapsto x^iy^j$ for all pairs of
commuting $x,y \in c$.
Here is our main result toward the Picard rank conjecture.

\begin{theorem}
       \label{theorem:stable-picard}
       Let $G$ be a finite group and $c \subset G$ a conjugacy class generating $G$.
       For $n$ large enough depending on $c$ and any component $Z \subset \churz G n c$, we have
$\pic(Z) \otimes \mathbb Z[\frac{1}{2|G|}] \simeq \left((\mathbb Z/(2n-2) \mathbb Z)
\otimes \mathbb Z[\frac{1}{2|G|}] \right)$.

       Let $\ord_{G^{\ab}}(c)$ denote the order of the image of any element of
       $c$ in $G^{\ab}$.
If $n$ is divisible by
       $\ord_{G^{\ab}}(c)$, 
$\pic(\churz G n c) \otimes \mathbb Z[\frac{1}{2|G|}] \simeq \left((\mathbb Z/(2n-2) \mathbb Z)
\otimes \mathbb Z[\frac{1}{2|G|}] \right)^{|H_2(G,c)|}$.
If $n$ is sufficiently large and not a multiple of 
$\ord_{G^{\ab}}(c)$,
$\churz G n c$ is empty.
\end{theorem}

We prove this later in 
\autoref{subsubsection:proof-stable-picard}.
The reader may wish to refer to \autoref{subsubsection:picard-rank-idea}
for a description of the idea of the proof.


\subsection{General lemmas on cohomology of stacks}

Before getting to the Picard rank conjecture, we record some technical lemmas
concerning Picard groups of Deligne-Mumford stacks.

\begin{lemma}
	\label{lemma:vanish-h1}
	Suppose $X$ is a smooth proper Deligne-Mumford stack over $\mathbb{C}$ and $U \subset X$ is an
	open substack, with complement of codimension at least $1$.
	Let $R$ be a localization of the integers.
	If $H^1(U; R) = 0$ then $H^1(X; R) = 0$.
\end{lemma}
\begin{proof}
	For any finite type Deligne-Mumford stack $Y$ over $\mathbb C$, such as $X$ or $U$,
	the cohomology is finitely generated in each degree. The universal coefficient theorem implies that
	$H^1(Y;R)$ is a free $R$ module of rank that agrees with that of
	$H^1(Y;\mathbb Z/\ell \mathbb Z)$ for all $\ell$ chosen so
	that $H_1(Y;\ZZ)$ has no $\ell$-torsion (which is all but finitely many primes by finite generation). Thus we know that $H^1(U;\ZZ/\ell\ZZ)$ vanishes for $\ell$ sufficiently large, and that it is enough to show that $H^1(X;\ZZ/\ell\ZZ)$ vanishes for sufficiently large $\ell$. In this case, this singular
	cohomology group is isomorphic to the corresponding \'etale
	cohomology group of the stack, as can be deduced by \'etale descent from the
	corresponding result for smooth schemes, which is proven in
	\cite[Expos\'e XI, Theorem 4.4]{SGA4}.
	Let $n := \dim X$.
	By Poincar\'e duality for $X$, (see, for example, \cite[Proposition
	4.4.2]{laszloO:the-six-operations-i},)
	it is equivalent to show the map
	$\phi: H^{2n-1}_{\on{c}}(U, \mathbb Z/\ell \mathbb Z) \to H^{2n-1}_{\on{c}}(X, \mathbb
	Z/\ell \mathbb Z) \simeq H^{2n-1}(X, \mathbb Z/\ell \mathbb Z)$
	on compactly supported cohomology
	is an isomorphism. 
	By assumption, the source vanishes. 
	The long exact sequence on compactly supported cohomology associated to
	the open $U \subset X$ and its closed complement $Z := X - U$ 
	implies
	that the cokernel of $\phi$
        injects into $H^{2n-1}(Z, \mathbb Z/\ell \mathbb Z)$. 
        Finally,
        $H^{2n-1}(Z, \mathbb Z/\ell \mathbb Z) = 0$ because $Z$ has dimension at most $n - 1$.
\end{proof}

\begin{proposition}
	\label{proposition:injection}
	Suppose $X$ is a smooth proper Deligne-Mumford stack 
	over $\mathbb C$
       with $H^1(X; \mathbb Q) = 0.$
       Then there is an injection $\pic(X) \to H^2(X; \mathbb Z)$.
       Additionally, the torsion in $H^2(X; \mathbb Z)$ lies in the
       image of this injection.
\end{proposition}
\begin{proof}
	Using the exponential exact sequence,\footnote{The exponential exact
		sequence works equally well for
	orbifolds as for analytic spaces, since its exactness can be verified on an
\'etale cover.}
	we have a short exact sequence $$H^1(X^{\on{an}}, \mathscr O_{X^{\on{an}}}) \to
       H^1(X, \mathscr O_{X^{\on{an}}}^\times)
       \xrightarrow{\alpha}
       H^2(X; \mathbb Z) \xrightarrow{\beta} H^2(X^{\on{an}}, \mathscr O_{X^{\on{an}}}).$$
       Since $H^2(X^{\on{an}}, \mathscr O_{X^{\on{an}}})$ is
       torsion free, any torsion element of $H^2(X; \mathbb Z)$
       vanishes under $\beta$.
       Therefore, to conclude the proof, it suffices to show $\alpha$ is an
       injection.
	Since $H^1(X, \mathscr O_{X^{\on{an}}}^\times)$ is identified
	with the Picard group, to prove the desired injection, we only need to show
	$H^1(X^{\on{an}}, \mathscr O_{X^{\on{an}}}) = 0$.
	Using GAGA for Deligne-Mumford stacks, 
	\cite[Proposition A.4]{hall:generalizing-the-gaga-principle},
	we have
	$H^1(X^{\on{an}}, \mathscr O_{X^{\on{an}}}) =
	H^1(X, \mathscr O_X)$.
	Since $H^1(X;\mathbb Q) = 0$, we also have $H^1(
	X;\mathbb C) = 0$, and hence we conclude 
	$H^1(X, \mathscr O_X) = 0$ using \cite[Corollary
       1.7]{satriano:de-rham-theory-for-tame-stacks}, which says that the Hodge
       de Rham spectral sequence degenerates for smooth proper Deligne-Mumford stacks.
\end{proof}

\subsection{Proving the stable Picard rank conjecture}
We now aim to prove \autoref{theorem:stable-picard}.
To do this, we next compute the first two stable cohomology groups of 
$\churz G n c$.
To do so, we need a basic lemma about the number of connected components of
Hurwitz spaces.
\begin{lemma}
	\label{lemma:stable-components}
	Let $G$ be a group and $c \subset G$ a conjugacy class generating $G$.
	For $n$ sufficiently large, the set of connected components
	of $\cphurc n c$ with boundary monodromy $g \in G$
	is either empty or forms a torsor under $H_2(G,c)$; it is nonempty 
	if and only if the
	image of $n$ in $G^{\ab}$ (under the map $\mathbb Z \to G^{\ab}$ sending
	the positive generator to the image of any element of $c$) agrees with the image of $g$
	in $G^{\ab}$.
\end{lemma}
Rephrasing the statement above, there are $H_2(G,c)$ many components if the image of
$n$ in $G^{\ab}$ agrees with the image of $g$, and $0$ components otherwise.
\begin{proof}
	This essentially follows from \cite{wood:an-algebraic-lifting-invariant}
	as we now explain.
	Indeed, using \cite[Theorem 2.5 and Theorem
		3.1]{wood:an-algebraic-lifting-invariant}
		we can identify the number of components of $\cphurc n c$ for $n$ sufficiently large
		with the set of elements in a certain reduced Schur cover $S_c
		\to G$ having the same image in $G^{\ab}$ as $n$.
		Moreover, the boundary monodromy of these components is the same
		as their image in $G$ under the map $S_c \to G$.
		The kernel of $S_c \to G$ is identified with $H_2(G,c)$
		and so connected components with boundary monodromy $g$ either
		form a
		torsor under $H_2(G,c)$ when the
		image of $n$ in $G^{\ab}$ agrees with the image of $g$, or else
		there are no such connected components.
\end{proof}

For the next lemma and its proof, the reader may wish to recall notation from
\autoref{notation:hurwitz-picard}.

\begin{lemma}
	\label{lemma:vanishing-open}
Let $G$ be a group, $c \subset G$ be a conjugacy class generating $G$, and
$R := \mathbb Z[1/2|G|]$.
For $n$ sufficiently large depending on $c$ and for each component $\widetilde{Z} \subset \churp G n c$,
with corresponding component $Z \subset \churz G n c$,
we have
	\begin{align*}
		H^1(\widetilde{Z}; R) = H^1(Z; R) &= 0, \\
		H^2(\widetilde{Z}; R) = H^2(Z; R) &=
	\left((\mathbb Z/(2n-2) \mathbb Z) \otimes R \right).
	\end{align*}
\end{lemma}
\begin{proof}
	Taking $g = \id$ in \autoref{lemma:stable-components}, we obtain 
	that for $n$ sufficiently large, both 
	$\chura G {n} c {\id}$ and
	$\churp G {n} c$
	have $|H_2(G,c)|$ many connected components.
	Indeed, the statement for $\chura G {n} c {\id}$ follows from
	\autoref{lemma:stable-components}
	and the fact that $G$ conjugation acts trivially on $H_2(G,c)$ as it
	is identified with the central kernel of $S_c \to G$ by definition.
	Since $ \chura G {n} c {\id}$ is dense open in 
	$\churp G n c$, we obtain $\churp G n c$ also has $H_2(G,c)$ many
	connected components.
	Additionally, 
	we claim $\chura G {n-1} c {c^{-1}}$ also has $H_2(G,c)$ many connected components.
	Indeed, the image of $n-1$ in $G^{\ab}$ agrees with $c^{-1}$ if and only
	if the image of $n$ is trivial, so we now assume these conditions are
	both satisfied. In this case,
	$\cphur {G} {n-1} c {\mathbb C}$ has $H_2(G,c)$ many connected
	components for each possible
	boundary monodromy.
	Also,
	$\chura G {n-1} c {c^{-1}}$ 
	is obtained by taking the components of
	$\cphurc {n-1} c$ with boundary monodromy in $c$, and quotienting by
	the $G$ conjugation action, which only permutes the boundary monodromy
	but acts trivially on $H_2(G,c)$. Combining the above with the
	assumption that $c$ is a single conjugacy
	class, we find
	$\chura G {n-1} c {c^{-1}}$ has $H_2(G,c)$ many components. 

	Note that we have an open inclusion $\chura G {n}
	c {\id} \subset \churp G {n}
	c$,
	corresponding to the locus of covers where all branch points lie in $\mathbb A^1 \subset
	\mathbb P^1$.
	The closed complement can be identified with the closed subset of $\churp G {n}
	c$, where one of the branch points of the associated cover lies at $\infty$. In other words, this
	closed complement is given by $\chura G {n-1} c{c^{-1}}$.

	We claim that for $n$ large enough,
	and each component $\widetilde{Z} \subset \churp G n c$,
	the map $H^i(\on{Conf}_{\mathbb P^1,n}; R)\to
	H^i(\widetilde{Z}; R)$ is an isomorphism for $i \in \{1,2\}$.
	To see this, note that there is at
least one component of $\chura G {n-1} c{c^{-1}}$ in the closure of any
component of
$\chura G {n} c{\id}$; since 
$\chura G {n-1} c{c^{-1}}$
and
$\chura G {n} c{\id}$
have the same number of components, there
must be a single component of $\chura G {n-1} c{c^{-1}}$ in the closure of any
component of $\chura G {n} c{\id}$.
From this we obtain maps $\churp G {n} c \to \coprod_{H_2(G,c)}
\on{Conf}_{\mathbb P^1,n}$ which is a bijection on connected components and
compatible with the above described stratification of the source.
We use $\widetilde{Z}' \subset \chura G n c {\id}$
and $\widetilde{Z}'' \subset \chura G {n-1} c{c^{-1}}$ for the components corresponding to
restrictions of $\widetilde{Z}$ under
the above stratification,

Hence, using the Gysin exact sequence, we have a map of exact sequences
	\begin{equation}
		\label{equation:excise-conf-to-hur}
\begin{tikzpicture}[baseline= (a).base]
\node[scale=.75] (a) at (0,0){
		\begin{tikzcd}[column sep=small]
			H^{2n-i-1}(\on{Conf}_{\mathbb
			A^1, n} ; R)\ar{r}\ar{d} &
H^{2n-i-2}(\on{Conf}_{\mathbb A^1, n-1} ; R) \ar {r} \ar {d} &
			H^{2n-i}(\on{Conf}_{\mathbb P^1,n} ; R)
			\ar {r} \ar {d} & H^{2n-i}(\on{Conf}_{\mathbb
			A^1, n} ; R)
 \ar
 {d}\ar{r} & H^{2n-i-1}(\on{Conf}_{\mathbb A^1, n-1} ; R) \ar{d} \\
 H^{2n-i-1}(\widetilde{Z}'; R) \ar{r} &
 H^{2n-i-2}(\widetilde{Z}''; R) \ar
			{r} & H^{2n-i}(\widetilde{Z}; R) \ar {r} & 
			H^{2n-i}( \widetilde{Z}'; R) \ar{r} &
			H^{2n-i-1}(\widetilde{Z}'' ; R).
	\end{tikzcd}};
\end{tikzpicture}
\end{equation}

	By \autoref{theorem:all-large-stable-homology}, and the universal
	coefficients theorem to relate homology to cohomology, the first two and
	last two vertical maps
	are isomorphisms for $n$ sufficiently large, depending on $c$.
Hence, the five lemma implies that the middle vertical map is an isomorphism for such $n$, as claimed.

For $\widetilde{Z} \subset \churp G n c$ a component,
we now deduce that for $n$ sufficiently large, $H^1(\widetilde{Z}; R)= 0$ and 
	$H^2(\widetilde{Z}; R)= ((\mathbb Z/(2n-2) \mathbb Z) \otimes
	R)$
	from the fact that
	$H^1(\on{Conf}_{\mathbb P^1,j}; R) = 0$ and
	$H^2(\on{Conf}_{\mathbb P^1,j}; R) = \mathbb Z/(2n-2) \mathbb Z \otimes
	R$,
	see, for example,
	\cite[Theorem 1.3]{schiessl:integral-cohomology-of-configuration-spaces-of-the-sphere}.

	Having established the claim for $\widetilde{Z}$, it remains to obtain
	the claim for its image $Z \subset \churz G n c$.
	Note that $\widetilde{Z} \subset \churp G n c$
	is a $\pgl_2$ bundle over $Z$.
	By \cite[Theorem 1.5]{brown:the-cohomology-of-bso}, the only integral cohomology of
	$B\on{SO}_3$ in degrees $1,2$, and $3$ is $2$-torsion.
	Since the cohomology of $B\on{SO}_3$ agrees with that of $B \on{PGL}_2$,
	we obtain that, since $2$ is invertible on $R$, 
	$h^i(B\pgl_2; R) = 0$ if $i \in \{1,2,3\}$.
	Then, 
	the Serre spectral sequence associated to the map $Z \to B\pgl_2$
	with fiber $\widetilde{Z}$ 
	implies 
	$H^i(Z; R) \simeq H^i(\widetilde{Z}; R)$ for $i \in \{1,2\}$,
	completing the proof.
\end{proof}

Recall notation from \autoref{notation:hurwitz-picard},
where we use $\churzb G n c$ to denote the Abramovich-Corti-Vistoli compactification
of $\churz G n c$, which is a smooth proper Deligne-Mumford stack for $n\geq3$, as mentioned
in \autoref{remark:compactification}.

We next use that $\churzb G n c$ is a smooth proper
Deligne-Mumford stack containing $\churz G n c$ as a dense open, 
and apply \autoref{proposition:injection}
to make the
idea in
\autoref{subsubsection:picard-rank-idea} rigorous.
In order to analyze the Picard group of $\churz G n c$ vanishes, we first analyze the
Picard group of its compactification.

\begin{lemma}
	\label{lemma:compactified-injection}
Let $G$ be a a group and $c \subset G$ a conjugacy class generating $G$, 
	For $n$ sufficiently large depending on $c$,
	there is an injection $\pic(\churzb G n c)
	\hookrightarrow	H^2(\churzb G n c; \mathbb Z)$.
\end{lemma}
\begin{proof}
	This follows from \autoref{proposition:injection} once we know
	$H^1(\churzb G n c; \mathbb Q) = 0$. 
	Observe that $H^1(\churz G n c; \mathbb Q) = 0$, for $n$ sufficiently
	large depending on $c$, by
	\autoref{lemma:vanishing-open}.
	Therefore, since $\churz G n c \subset \churzb G n c$ is
	an open substack, we obtain from \autoref{lemma:vanish-h1}, that for $n$ sufficiently large
	depending on $c$, 
	$H^1(\churzb G n c; \mathbb Q) = 0$. 
\end{proof}

\subsubsection{Proof of \autoref{theorem:stable-picard}}
\label{subsubsection:proof-stable-picard}
Recall we are trying to compute  
$\pic(Z) \otimes R$ for $R := \mathbb Z[1/2|G|]$
and $n$ sufficiently large,
depending on $c$,
and $Z \subset \churz G n c$ a component.
If $n \nmid \ord_{G^{\ab}}(c)$, 
$\churz G n c$ is empty by \autoref{lemma:stable-components}.
The computation of 
$\pic(\churz G n c) \otimes R$
then follows from the computation of the Picard group of each component, using 
\autoref{lemma:stable-components}.
For the remainder of the proof we assume $n$ is sufficiently large and $n \mid
\ord_{G^{\ab}}(c)$, and $Z \subset \churz G n c$ is a component.

Define $\partial := \overline{Z} - Z$, for $\overline Z$ the closure of $Z$ in
the compactification $\churzb G n c$. 
	Let $Q$ denote the submodule of $\pic(\churz G n c) \otimes R$
	spanned by the line bundles corresponding to irreducible components of $\partial.$
	The cycle class map yields a map of exact sequences
	\begin{equation}
		\label{equation:cycle-class-map}
		\begin{tikzcd}
			& Q \ar {r}{\alpha} \ar {d}{\widetilde{\beta}} &
			\on{Pic}(\overline{Z}) \otimes
			R \ar {r} \ar {d}{\gamma}
			& \on{Pic}(Z) \otimes R \ar {r} \ar
			{d}{\varepsilon} & 0 \\
			& Q \ar {r}{\delta} & H^2(\overline{Z}; R) \ar
			{r}{\zeta} &
			H^2(Z; R) \ar{r}{\xi} & H^1(\partial;R)
	\end{tikzcd}\end{equation}
	where the bottom exact sequence is the Gysin sequence associated to the
	closed substack $Z \subset \overline{Z}$, the top sequence is the
	excision sequence on Picard groups, and the vertical maps are given by
	the cycle class maps.
	We note excision for Picard groups follows from \cite[Proposition
	2.3.6]{kresch:cycle-groups-for-artin-stacks}, which gives excision for
	Chow groups, and 
	\cite[Proposition 1]{kresch:hodge-theoretic-obstruction}, which
identifies the first Chow group with the Picard group.

	We next show $\varepsilon$ is an injection.
	Using \autoref{lemma:compactified-injection}, $\gamma$ is an injection. 
	It follows from a diagram chase (see, for example,
		\cite[Exercise
	1.7.D]{vakil:foundations-of-algebraic-geometry-2017}) that the map
	$\varepsilon$ is an injection.
	
	Finally, we show $\varepsilon$ is a surjection.
	By \autoref{lemma:vanishing-open}, $H^2(Z; R)  \cong ((\mathbb
	Z/(2n-2) \mathbb Z) \otimes R)$ and in particular is
	torsion. 
	This implies $\gamma$ is a surjection by
	\autoref{proposition:injection}.
	It also implies that 
	the map $\xi$ in
	\eqref{equation:cycle-class-map} is $0$, 
	because $H^1(\partial;R)$ is torsion free.
	As $\xi =0$, $\zeta$ is surjective.
	Since $\zeta$ is surjective and $\gamma$ is
	surjective, their composite is surjective, and hence $\varepsilon$ is surjective as well.
	\qed

\section{Frobenius equivariance of the stabilization map}
\label{section:frobenius-equivariance}

In this section, we prove \autoref{theorem:frob-equivariant-stabilization},
which shows that the stabilization maps for cohomology of Hurwitz spaces are
suitably equivariant for the action of Frobenius.
The main consequence of this we will need in future sections is
\autoref{lemma:component-point-bound}, which uses this Frobenius equivariance to
relate point counts of Hurwitz spaces over finite fields.

We first state the main result in
\autoref{subsection:stating-frobenius-equivariance}.
We define the stabilization map in 
\autoref{subsection:defining-stabilization}.
We complete the proof of \autoref{theorem:frob-equivariant-stabilization} in
\autoref{subsection:topological-identification}.
Finally, we show that the hypotheses of
\autoref{theorem:frob-equivariant-stabilization}
are often satisfied in 
\autoref{subsection:verifying-hypotheses}
and also deduce some consequences.

\subsection{Stating the main result on Frobenius equivariance}
\label{subsection:stating-frobenius-equivariance}

We begin by introducing notation to state our main result.
\begin{notation}
	\label{notation:frobenius-stabilization}
		Suppose $c = c_1 \cup \cdots \cup c_\upsilon$ with each $c_i$ a
		conjugacy class in 
		a group $G$ and $g \in c_1$. Let $d := \on{ord}(g)$, and define $r$ to be the minimal integer so that
		$g^{q^r} = g$. Suppose $(g, g^q, \ldots,
	g^{q^{r-1}})$ are contained in the pairwise distinct conjugacy classes
	$c_1, \ldots, c_s$, and each such conjugacy class contains some such
	power of $g$.
	Note that each conjugacy class contains
	$r/s$ elements in the set $\{g, g^q, \ldots,
	g^{q^{r-1}}\}$.
\end{notation}

For the next theorem, it will be useful to recall the notion of the  boundary
monodromy of a component of Hurwitz space coming from a union of conjugacy
classes in a group, as defined in 
\autoref{definition:powering}
Loosely speaking, the boundary monodromy of a component is the element of $G$
appearing as the monodromy at $\infty$ for that component.

\begin{theorem}
	\label{theorem:frob-equivariant-stabilization}
	Fix a finite group $G'$ a normal subgroup $G \subset G'$, and a subgroup $K \subset G'$.
	We use notation from \autoref{definition:hurwitz-stack-over-b},
\autoref{definition:powering},
\autoref{notation:component-maps},
and
\autoref{notation:frobenius-stabilization}.
Choose $M > 0$.
Suppose
$W \subset \phur{G'}{Mdr} {c}{\mathbb F_q}$
is the component indexed by $\prod_{j=0}^{r-1}[g^{q^j}]^{Md}$
and $W(\mathbb F_q)\neq \emptyset$.
Then, the map 
\begin{align*}
	U^{g,q,M,K} := \sum_{\kappa \in K} \prod_{j=0}^{r-1}
[(\kappa^{-1}g\kappa)^{q^j}]^{Md} : 
H^i(\cquohur{G}{K}{n + Mdr} {c}{\mathbb
C}, \mathbb Q_\ell) \to H^i(\cquohur{G}{K}{n}{c}{\mathbb
C}, \mathbb Q_\ell)
\end{align*}
can be identified via specialization with a map
\begin{align*}
	U^{g,q,M,K}_{\overline{\mathbb F}_q}:
H^i(\cquohur{G}{K}{n + Mdr} {c}{\overline{\mathbb F}_q}, \mathbb Q_\ell) \to H^i(\cquohur{G}{K}{n} {c}{\overline{\mathbb
F}_q}, \mathbb Q_\ell).
\end{align*}
Moreover, 
$U^{g,q,M,K}_{\overline{\mathbb F}_q}$
is equivariant for the actions of Frobenius on the source and
target, coming from viewing them as the base changes of
$H^i(\cquohur{G}{K}{n + Mdr} {c}{\mathbb F_q}, \mathbb Q_\ell)$ and $H^i(\cquohur{G}{K}{n} {c}{\mathbb F_q}, \mathbb Q_\ell)$.
\end{theorem}

The proof will be somewhat involved, and we will complete it in
\autoref{subsection:equivariant-proof}.
The general strategy of proof will crucially use log geometry, and is similar to that of \cite[Theorem
A.5.2]{ellenbergL:homological-stability-for-generalized-hurwitz-spaces}.
Since this proof is similar to that one, we will be
somewhat brief, often referring the reader to analogous steps carried out there.
\begin{remark}
	\label{remark:}
	The appearance of the parameter $M$ in
	\autoref{theorem:frob-equivariant-stabilization} may seem mysterious.
	The point of including this is that when $M = 1$, we do not know how to
	rule out the possibility that
$W(\mathbb F_q)= \emptyset$.
However, we will see in \autoref{subsection:verifying-hypotheses} that when $M$ is sufficiently large, 
$W(\mathbb F_q)\neq \emptyset$.
\end{remark}

\subsection{Defining the stabilization map algebraically}
\label{subsection:defining-stabilization}

Our first goal toward proving \autoref{theorem:frob-equivariant-stabilization}
is to define a map of objects in positive characteristic that agrees with the
complex stabilization map of Hurwitz spaces.
This map will not be in the category of schemes or stacks, but rather in the
category of log schemes or log stacks. We define this map before quotienting by
$K$ in \autoref{lemma:gluing} and define the induced map on homology of
quotients by $K$ in \autoref{subsubsection:gluing-quotient}.

\subsubsection{Defining a space of stable maps}
We next define a space of stable maps which is a compactification of our
Hurwitz spaces.
\begin{notation}
	\label{notation:henselian}
	Throughout the remainder of this section, we will work over a fixed Henselian
dvr $B$ with residue field $\mathbb F_q$ and generic point of characteristic $0$. 
We also assume $|G|$ is invertible on $B$.
All schemes and stacks in this section will be considered over $B$.
\end{notation}

The reason we want $B$ to be Henselian in \autoref{notation:henselian} is primarily due to
\autoref{lemma:component-bijection}, which gives us a bijection between
components over $B$ and components over $\mathbb F_q$, which we are implicitly
using to make sense of the meaning of a component {\em indexed by} a tuple in \autoref{definition:powering}.

We assume $c \subset G$ is a union of conjugacy classes which is closed under
$q$th powering in the sense of \autoref{definition:powering}.
	Let $\mathcal K_{n+1,0}(\mathbb P^1 \times BG, 1)^c$
	denote the moduli stack of maps from a stable genus
	$0$ twisted curve $\mathcal X$ with $n+1$ marked sections to $\mathbb
	P^1 \times BG$,
	where the source is geometrically
	irreducible, the maps are balanced in the sense of
	\cite[\S2.1.3]{abramovichCV:twisted-bundles},
	the first $n$ of the $n+1$ sections have inertia in $c$,
	and such that
	pullback of $\mathscr O_{\mathbb P^1}(\infty)$ on $\mathbb P^1$ under the composite
	$\mathcal X \to \mathbb P^1 \times BG \to \mathbb P^1$ has degree $1$ on
	$\mathcal X$.

	Concretely, points of this stack correspond to $G$-covers of $\mathbb
	P^1$ ramified at $n+1$ marked points with inertia in $c$.
	We let 
	$[\mathcal K_{n+1,0}(\mathbb P^1 \times BG, 1)^c/S_{n}]$
	denote the quotient of the above stack by $S_n$, given by the action on
	the first $n$ marked points.
	Now, points of this stack 
	correspond to $G$-covers of $\mathbb P^1$ ramified at a degree $n+1$ divisor which contains a marked section, with inertia in $c$.

\subsubsection{Defining a compactification of the pointed Hurwitz space}
\label{subsubsection:compactification}
	Next, let
$\widetilde{\mathscr H}_{n}^c$
denote the closed substack of 
$[\mathcal K_{n+1,0}(\mathbb P^1 \times BG, 1)^c/S_n]$
where the last marked point maps to $\infty \in \mathbb P^1$.
(This is also a union of components of the stack from \cite[Notation
	B.1.1]{ellenbergL:homological-stability-for-generalized-hurwitz-spaces}
denoted there by $\mathcal K_{0,n}([\mathbb P^1/G], \infty, 1)$.)
Let $\overline{\mathscr H}_{n}^c \to \widetilde{\mathscr H}_{n}^c$
obtained by viewing the points of $\widetilde{\mathscr H}_{n}^c$ as
parameterizing twisted $G$ covers and marking a section over the final marked
point of the twisted curve. 
The construction here is analogous to the marked
section $t$ in \autoref{definition:hurwitz-over-b}.
More precisely, we can view a $T$ point of 
$\widetilde{\mathscr H}_{n}^c$ as corresponding to a finite Galois $G$-cover $X
\to \mathscr P$,
where $\mathscr P$ is a stacky curve with genus $0$ coarse space,
with a stacky point of order $r$ over the
final 
specified section $p_{n+1} : T \to \mathscr P$, such that the map $f: X \to \mathscr P$ is \'etale over 
$p_{n+1}$. The cover 
$\overline{\mathscr H}_{n}^c \to \widetilde{\mathscr H}_{n}^c$
is obtained by marking a section $T \to X \times_{f, \mathscr P, p_{n+1}} T$.

We let $\mathscr H_n^c \subset \overline{\mathscr H}_{n}^c$ denote the open
substack parameterizing smooth covers and $\mathscr D_n^c := \overline{\mathscr
H}_n^c - \mathscr H_n^c$ denote the boundary divisor parameterizing singular
covers.

\begin{remark}
	\label{remark:open-is-hurwitz}
	By construction, 
$\mathscr H_{n}^{c}$
is identified with
the pointed Hurwitz space
$\phur {G'} {n} c B$.
\end{remark}

\begin{lemma}
	\label{lemma:nc-compactification}
	The stack $\overline{\mathscr H}_n^c$ is smooth and proper over $B$ and $\mathscr D_n^c$ is a
	normal crossings divisor not containing any component in any fiber over $B$.
\end{lemma}
\begin{proof}
	This is a special case of 
	\cite[Corollary
	B.1.4]{ellenbergL:homological-stability-for-generalized-hurwitz-spaces}.
where we take $C = \mathbb P^1, Z = \infty \subset \mathbb P^1$. 
The final statement that it does not contain any component in any fiber over $B$
follows from applying 
\cite[Corollary
B.1.4]{ellenbergL:homological-stability-for-generalized-hurwitz-spaces}
again to the base change along any point of $B$.
\end{proof}

\begin{example}
	\label{example:configuration}
	Above, we also allow the possibility that $c= G = \id$, in which case 
	$\mathscr H_{n}^{c} \simeq \conf_{n}$ and  
$\overline{\mathscr H}_{n}^{c}$ is a compactification of
configuration space parameterizing configurations of points on nodal genus $0$ curves.
\end{example}

\subsubsection{Defining the gluing map in algebraic geometry}

\begin{figure}
	\centering
	\includegraphics[scale=.5]{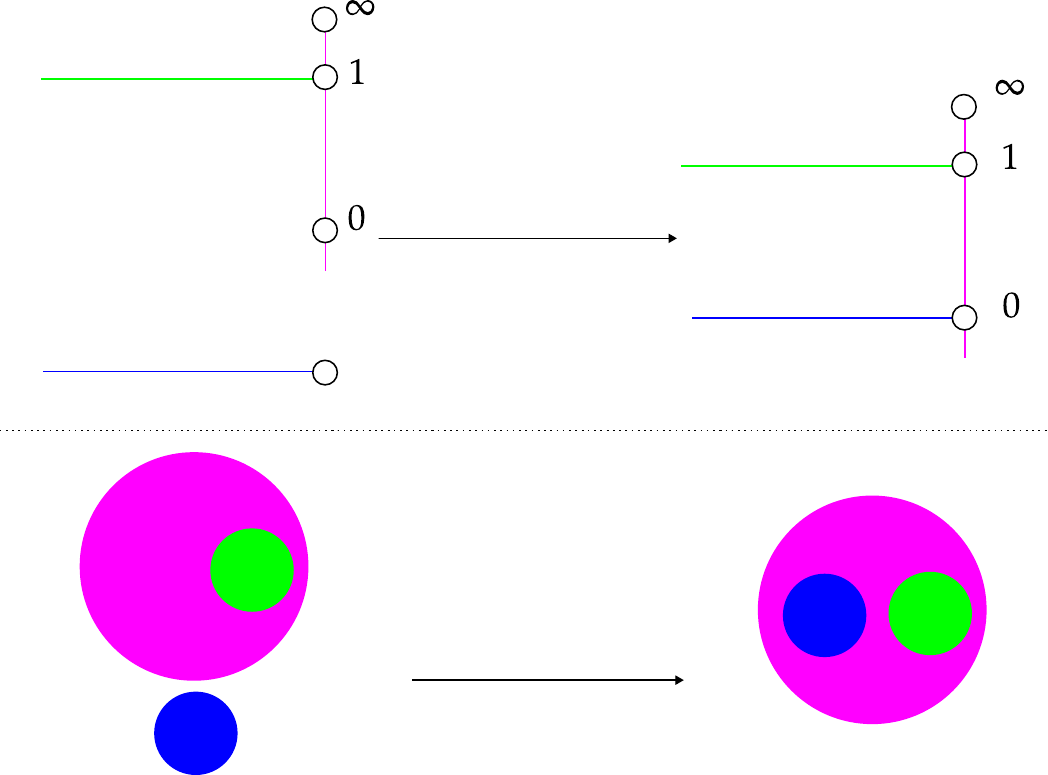}
	\caption{The top half of the diagram pictures the algebraic gluing map on the base
		of the cover, corresponding to the stabilization map.
		The map takes in 
		a point of $\mathscr H^c_n$, corresponding to 
		the blue line, and glues it to a fixed cover, corresponding to
		the green and pink lines,
		to obtain a point of $\overline{\mathscr H}^c_{n+Mdr}$.
		Here, there are $n$ additional branch points on the blue line, 
		$Mdr$ additional marked point on the green line, and no
		additional marked points on the pink line.
		This is meant to be an algebraic incarnation of the topological
		structure pictured coming from the little discs operad in the
		bottom half of the diagram. The algebraic map replaces each disc with a
		copy of $\mathbb P^1$, with the white circles on the algebraic
		picture corresponding to the
		boundary of the disc.
}\label{figure:algebraic-gluing}
\end{figure}

We next define a gluing map in \eqref{equation:scheme-gluing}, depending on a point $x$, which we now introduce
notation to describe.
The algebraic gluing map on the base of the cover is pictured in
\autoref{figure:algebraic-gluing}.
A topological incarnation is also pictured below in
\autoref{figure:topological-gluing}.
We will later see this map is compatible with the topological gluing map
$U^{g,q,M,K}$ from \autoref{theorem:frob-equivariant-stabilization}.
\begin{notation}
	\label{notation:section-lift}
	Fix a point $x \in W(\mathbb F_q)$, for $W \subset \phur{G'}{Mdr}{c}{\mathbb F_q}$
the component indexed by $\prod_{j=0}^{r-1}[g^{q^j}]^{Md}$ in the sense of
\autoref{definition:powering}.
First, we may lift $x$ to a point modulo any power of the maximal ideal
of the dvr corresponding to $B$ using smoothness of $W \subset \phur{G'}{Mdr}{c}{B}$
over $B$. Then, this compatible system of infinitesimal lifts algebraizes to a
$B$-point $x_B: B \to \phur{G'}{Mdr}{c}{\mathbb F_q}$
using
\cite[Corollary
8.4.6]{FantechiGIK:fundamentalAlgebraicGeometry}.
The base change of $x_B$ along the special fiber of $\spec \mathbb F_q \to B$
agrees with $x$ by construction.
We also use $x_{\mathbb C}: \spec \mathbb C \to \phur{G'}{Mdr}{c}{\mathbb C}$ to denote the base change of $x_B: B \to
\phur{G'}{Mdr}{c}{B}$
along $\spec \mathbb C \to B$.
\end{notation}

By definition of
the component $W$, the boundary monodromy of $x$ is trivial, so the cover
corresponding to $x_B$ is
unramified over $\infty$.
We view $x$ as corresponding to a cover $[x_B] : C \to \mathbb P^1$ together with a
marked section $p \in C$ mapping to $\infty$ in $\mathbb P^1$. 

We will next define a map 
\begin{align}
	\label{equation:scheme-gluing}
\Gamma_x : \mathscr H_{n}^{c} \to \overline{\mathscr H}^c_{n+Mdr}
\end{align}
over $B$, using our chosen point $x$ as above.

We will define this map precisely via a functorial construction on $T$ points.
The reader may wish to refer to \autoref{figure:algebraic-gluing} which pictures this gluing map.
It suffices to carry out the construction for connected $T$.
A $T$ point of $\mathscr H_{n}^{c}$ corresponds to a family of covers
$\psi : X \to \mathscr P_T$, where $\mathscr P$ is a root stack of
order $w$
of
$\mathbb P^1_B$ along $\infty$, with $w$ to be defined below, $\widetilde{\infty}_T$ is the base change of the natural section
$\widetilde{\infty} : B \to \mathscr P$ over $\infty$,
and additionally we have a specified section $t: T \to X \times_{\psi,
\mathscr P_T, \infty} T$.

Using this data and the point $x$, we construct a family of curves
over $T$ corresponding to a $T$ point of $\overline{\mathscr H}_{n+Mdr}$.
Since we are assuming $S$ is connected, and we have a marked section $t$,
we may use this section to identify the monodromy over $\infty$
with a fixed element $h \in G$.
The root stack $\mathscr P$ has order $w := \on{ord}(h)$ along $\infty$.
Moreover, the map $\psi$ induces a map $\mathscr P_T \to BG$.
Since the residual gerbe of $\mathscr P_T$ over $\infty$ is $B \mu_{w,T}$,
we obtain a composite map $B \mu_{w,T} \to \mathscr P_T \to BG$.
This map induces a map $\mu_{w,T} \to G$ on inertia stacks, which has trivial
kernel and image $(\mathbb Z/w \mathbb Z)_T$, as can be verified on an \'etale cover.
Hence we obtain an isomorphism
$\mu_{w,T} \to (\mathbb Z/w \mathbb Z)_T$.

Now, there is a cover $f': \mathbb P^1_T \to \mathbb P^1_T$ which is ramified over
$\infty$ to order $w$.
Let $\mathscr P_{0, \infty}$ denote the root stack of order $w$ of $\mathbb P^1_T$ over the
sections $0$ and $\infty$ of $\mathbb P^1$. Then, $f'$ factors through $f'':
\mathbb P^1_T \to \mathscr P_{0, \infty}$. 
Consider the finite \'etale cover 
$f: \cup_{i = 1}^{|G|/w} \mathbb P^1_T \to \mathscr P_{0, \infty}$ given
by taking a disjoint union of $\frac{|G|}{w}$ many copies of $f''$.
We view this as a $G$ cover by viewing $f$ as a $\mathbb Z/w \mathbb
Z$ cover with base point over $\infty$,
and using the inclusion $\mathbb Z/w \mathbb Z \to G$ sending $1$ to $h$.

Additionally, choose a fixed marked section of one of the source copies of
$\mathbb P^1$ over the point $1 \in \mathscr P_{0,\infty}$ and glue this to the
marked point $p \in C$ coming from our point $x_B$ from
\autoref{notation:section-lift}.
Since we have an isomorphism $\mu_{w,T} \simeq (\mathbb Z/w \mathbb Z)_T$.
the fiber of $f$ over the section $1 : T \to \mathscr P_{0,\infty}$
is a disjoint union of $w$ sections.
Hence, we may glue all other points of the fiber of $f$ over $1$ with all other points of the
fiber of the cover $[x_B]$ over $\infty$, compatibly with the $G$ actions on both.
Choose a marked point $\alpha$ in the fiber of $f$ over $\infty$ and a
marked point $\beta$ in the fiber of $f$ over $0$.
Note that $f$ has inertia generated by $h^{-1}$ over $0$
as it is inverse to the inertia over $\infty$.
Glue $\beta$ to the marked section $t$. Then, glue
all other points of the fiber of $f$ over $0$ to the fiber of $\psi$ over
$\infty$ compatibly with the $G$ actions on both.
Altogether, this yields
a cover of curves where the base curve has coarse space with three rational components.
The cover is ramified over a scheme over a scheme contained in the smooth locus
of the target which has degree $n + Mdr$ and it has a marked section 
$\alpha$ over $\infty$. This data yields a $T$ point of 
$\overline{\mathscr H}_{n+Mdr}$.

\begin{remark}
	\label{remark:}
	Above, it is important to assume the monodromy associated to $\psi$ over
	$\infty$ is inverse to the monodromy $0$ at the point $\alpha$ in order to
obtain that the glued cover is balanced in the sense of
\cite[\S2.1.3]{abramovichCV:twisted-bundles}. This balanced condition is crucial
for proving that $\mathscr H^c_n \subset \overline{\mathscr H}^c_{n}$ meets each
irreducible component in each fiber in \autoref{lemma:nc-compactification}.
\end{remark}

\subsubsection{Defining logarithmic structures}
We next upgrade the above gluing map \eqref{equation:scheme-gluing} to a map in logarithmic geometry.
Recall that a Deligne-Faltings log stack (or scheme) can be described as 
a Deligne-Mumford stack (or scheme), together with $v$ line bundles $L_1, \ldots, L_v$ on
$X$ and $v$ sections $\sigma_i : \mathscr O_X \to L_v$ for some $v \geq 0$.
For further background on log stacks (or schemes) pertinent to this context, we suggest the
reader consult
\cite[\S\,A.2.3]{ellenbergL:homological-stability-for-generalized-hurwitz-spaces}.

In general, if $X$ is a Deligne-Mumford stack and $D \subset X$ is a divisor,
the {\em log structure defined by $D$} on $X$ corresponds to the log stack with underlying stack $X$, $v = 1$, line bundle
$\mathscr O_X(D)$, and the tautological section $\sigma_1: \mathscr O_X \to \mathscr O_X(D)$
corresponding to the divisor $D$.
We define the log scheme
$(\overline{\mathscr H}_{n}^{c})^{\on{log}}$ to be the log structure
on
$\overline{\mathscr H}_{n}^{c}$ defined by the divisor
$\mathscr D_{n}^{c}$.

Next, for any scheme or Deligne-Mumford stack $X$, we define the log stack
$X^{\on{std}}$
to denote the log stack with underlying scheme
$X$
and with the {\em standard log structure}, i.e., the single line bundle $\mathscr O_X$, together with the $0$ section.

We now upgrade our gluing map $\Gamma_x$ above to a map of log stacks.

\begin{lemma}
	\label{lemma:gluing}
	The gluing map $\Gamma_x$ from \eqref{equation:scheme-gluing} induces a
	strict map of log stacks
	\begin{align}
		\label{equation:log-gluing-map}
		\alpha_x: 
		\left(\mathscr H_{n}^{c}\right)^{\on{std}} \to
\left( \overline{\mathscr H}_{n+Mdr}^{c} \right)^{\on{log}}.
	\end{align}
\end{lemma}
\begin{proof}
	We have already exhibited the underlying map of schemes $\Gamma_x$.
	We use $\mathscr O_{\overline{\mathscr
	H}_{n+Mdr}^{c}}(\mathscr D^c_n)$ 
	to denote the line bundle defining the log structure on the target and
	$s: \mathscr O_{\overline{\mathscr
	H}_{n+Mdr}^{c}} \to
	\mathscr O_{\overline{\mathscr
	H}_{n+Mdr}^{c}}(\mathscr D^c_n)$ to denote the section associated to the
	divisor $\mathscr D^c_n$.
	By definition of a strict morphism, it suffices to show $\mathscr O_{\overline{\mathscr
	H}_{n+Mdr}^{c}}(\mathscr D^c_n)$ 
	pulls back under $\Gamma_x$ to the trivial line bundle and $s$ pulls back
	to the $0$ section. 
	The latter is clear because the source maps into the locus of reducible
	genus $0$ curves on the target, and hence maps into the locus where $s$
	vanishes.
	It is trickier to show the pullback of $\mathscr O_{\overline{\mathscr
	H}_{n+Mdr}^{c}}(\mathscr D^c_n)$ 
 is trivial, but
	the argument for this is analogous to that demonstrating \cite[Lemma
	A.2.5]{ellenbergL:homological-stability-for-generalized-hurwitz-spaces},
	with the key input being \cite[p. 346, line
	2]{arbarelloCG:geomtry-of-algebraic-curves-ii}.
	We omit further details.
\end{proof}

\subsubsection{Defining a gluing map on log quotients by $K$}
\label{subsubsection:gluing-quotient}

For any scheme $T$ mapping to $B$, the map $\alpha_x$ from \eqref{equation:log-gluing-map},
	induces a map on cohomology
$H^i\left(\left( \overline{\mathscr H}_{n+Mdr, T}^{c}\right)^{\on{log}}, \mathbb
Q_\ell\right)
\to
H^i\left( \left(\mathscr H^{c}_{n, T}\right)^{\std}, \mathbb Q_\ell\right)$.
By transfer along the quotient by $K$, we also obtain a map on cohomology
	\begin{align}
		\label{equation:log-gluing-map-quotient}
		\alpha_x^* : 
H^i\left(\left( [\overline{\mathscr H}_{n+Mdr, T}^{c}/K] \right)^{\on{log}},
\mathbb Q_\ell \right)
\to
H^i \left( \left([\mathscr H^{c}_{n, T}/K] \right)^{\std}, \mathbb Q_\ell \right),
	\end{align}
	where the log structure on the target is the standard log structure and
	the log structure on the source is the log structure defined by the
	boundary divisor $\mathscr D^c_{n+Mdr}$.

\subsection{Identifying the algebraic stabilization map with the topological
stabilization map}
\label{subsection:topological-identification}

We next aim to compare the map on log stacks we have constructed above to the
usual gluing map over the complex numbers coming from the $\mathbb E_2$ algebra
structure on configuration space.
To set up this comparison, we introduce a few names for relevant maps on
cohomology.
By \cite[Corollary 7.5]{illusie:an-overview},
whose hypotheses are satisfied by \cite[7.3(b)]{illusie:an-overview} and
the normal crossings compactification in \autoref{lemma:nc-compactification},
upon base changing to any spectrum of a field $T \to B$,
there is
an identification
\begin{align*}
	\delta: H^i( [{\mathscr H}_{n+Mdr,T}^{c}/K] , \mathbb Q_\ell)
	\simeq H^i(\left( [\overline{\mathscr H}_{n+Mdr,T}^{c}/K]
	\right)^{\on{log}}, \mathbb Q_\ell).
\end{align*}
Recall that we use $\alpha_x^*$ for the gluing map as defined in
\eqref{equation:log-gluing-map-quotient}.
Upon base changing this to any spectrum of a field $T \to B$,
we obtain the composite map $\alpha_x^* \circ \delta$
\begin{equation}
	\label{equation:cohomology-composite-to-standard}
\begin{aligned}
	H^i( [{\mathscr H}_{n+Mdr, T}^{c}/K] , \mathbb Q_\ell) &\xrightarrow{\delta} 
H^i(\left( [\overline{\mathscr H}_{n+Mdr, T}^{c}/K] \right)^{\on{log}}, \mathbb Q_\ell)
\\
&\xrightarrow{\alpha^*_x} 
H^i( \left([\mathscr H^{c}_{n, T}/K] \right)^{\std}, \mathbb Q_\ell).
\end{aligned}
\end{equation}

To state the next result, we use $\Sigma_{g,p}^b$ to denote a genus $g$ Riemann surface with $p$ punctures
and $b$ boundary components.
We now define a gluing map. It may be helpful to refer to
\autoref{figure:topological-gluing} for a pictorial description of this gluing
map. \autoref{figure:algebraic-gluing} may also be helpful.
\begin{construction}
	\label{construction:gluing}
	The gluing map takes in the
	data:
	\begin{enumerate}
		\item a direction $\tau$ on the unit circle,
		\item a (ramified) $G$-cover $X_1$ of $\Sigma_{0,0}^1$ with a
			trivialization along the boundary, corresponding to a point
			of $\mathscr H_{n}^{c}$ with monodromy $h$ along the
			boundary,
		\item a (ramified) $G$-cover $X_2$ of $\Sigma_{0,0}^1$ with a
			trivialization along the boundary, corresponding to
			the point $x_{\mathbb C}$ as defined in
			\autoref{notation:section-lift},
			which has trivial monodromy over $\infty$,
		\item an unramified $G$ cover $X_3$ of $\Sigma_{0,0}^3$ with a
			trivialization along the first boundary, that has
			monodromy $h$ over the first boundary of $\Sigma_{0,0}^3$, trivial
			monodromy over the second boundary of $\Sigma_{0,0}^3$, and monodromy
			$h^{-1}$
			over the third boundary component of $\Sigma_{0,0}^3$,
		\item a fixed identification of the boundary of $X_2$ with
			the second boundary of $X_3$, compatible with the $G$
			actions on both,
		\item a specified identification of one of the
			boundary components of $X_1$ with $S^1$,
		\item a specified identification of one of the boundary
			components of $X_3$ over the third boundary of
			$\Sigma_{0,0}^3$ with $S^1$.
	\end{enumerate}
	The gluing map then glues $X_2$ with $X_3$ as specified in $(5)$, 
	glues the two copies of $S^1$ in (6) and (7) via a
	rotation by $\tau$ from (1) compatibly with the projections to $\Sigma_{0,0}^1$
	and $\Sigma_{0,0}^3$, and glues the remaining boundary components of
	$X_1$ with components of $X_3$ over the third boundary component of
	$\Sigma_{0,0}^3$ in a $G$-equivariant fashion.
\end{construction}

\begin{figure}
	\centering
	\includegraphics[scale=.5]{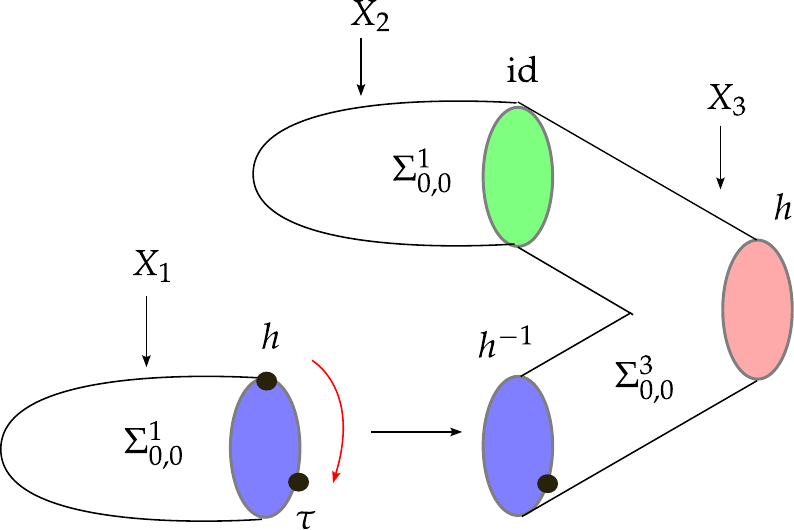}
	\caption{A figure depicting the gluing construction described in 
		\autoref{construction:gluing},
		inducing the
		Frobenius equivariant stabilization map.
}\label{figure:topological-gluing}
\end{figure}

The proof of the next lemma is completely analogous to that of
\cite[Lemma A.4.3]{ellenbergL:homological-stability-for-generalized-hurwitz-spaces},
so we omit it.
The key point is to use the identification of the log schemes we have defined
with their corresponding Kato-Nakayama spaces.

\begin{lemma}
	\label{lemma:}
	Taking $T = \spec \mathbb C$ in
	\eqref{equation:cohomology-composite-to-standard}, the composite
	$\alpha_x^* \circ \delta$ there can be identified with the map induced
	after taking cohomology and applying transfer to the quotient by $K$ of
	the gluing map from \autoref{construction:gluing}.
\end{lemma}

Next, we consider the stabilization operator
$\prod_{j=0}^{r-1}
[(g)^{q^j}]^{Md} : \cphur {G'} {n} {c}{\mathbb C} \to
\cphur {G'} {n+Mdr}{c}{\mathbb C}$.
This induces a corresponding map on cohomology, and applying transfer to the
quotient by $K$, we obtain a map on cohomology
\begin{align*}
	U^{g,q,M,K} := \sum_{\kappa \in K} \prod_{j=0}^{r-1}
[(\kappa^{-1}g \kappa)^{q^j}]^{Md} : 
H^i(\cquohur{G}{K}{n+Mdr} {c}{\mathbb C}, \mathbb Q_\ell)
\to H^i(\cquohur{G}{K}{n} {c}{\mathbb C}).
\end{align*}
Now, the space of stable maps
$\overline{\mathscr H}^c_{n}$
constructed in 
\autoref{subsubsection:compactification}
maps to a compactification of configuration space
(corresponding to the version of 
$\overline{\mathscr H}^c_{n}$ associated to the identity group in place of $G$,
see \autoref{example:configuration})
with complement a 
normal crossings divisor, again using
\autoref{lemma:nc-compactification}.
By
\cite[Proposition 7.7]{EllenbergVW:cohenLenstra},
(applied in the case that the map $\pi$ there is the trivial cover,)
the specialization map
induces vertical isomorphisms in \eqref{equation:transfer-u} below. We then
define the map $U^{g,q,M,K}_{\overline{\mathbb F}_q}$ to be the unique map making the
diagram below commute:
\begin{equation}
	\label{equation:transfer-u}
	\begin{tikzcd} 
	H^i(\cquohur{G}{K}{n} {c}{\mathbb C}; \mathbb Q_\ell) \ar
	{d} & H^i(\cquohur{G}{K}{n+Mdr} {c}{\mathbb C}; \mathbb Q_\ell)  \ar {d} \ar{l}{U^{g,q,M,K}} \\
	H^i(\cquohur{G}{K}{n} {c}{\overline{\mathbb F}_q};
	\mathbb Q_\ell)  & H^i(\cquohur{G}{K}{n+Mdr} {c}{\overline{\mathbb F}_q};
	\mathbb Q_\ell) \ar{l}{U^{g,q,M,K}_{\overline{\mathbb F}_q}}.
\end{tikzcd}\end{equation}

For $T$ a scheme over $B$,
upon endowing 
$\cquohur{G}{K}{n} {c}{T}$
with the trivial log structure, corresponding to no line bundles,
there is a map of log stacks 
$\left(\cquohur{G}{K}{n} {c}{T} \right)^{\std} \to \cquohur{G}{K}{n} {c}{T}$, which induces a map on
cohomology
\begin{align*}
	H^i(\cquohur{G}{K}{n} {c}{T}, \mathbb Q_\ell) \xrightarrow{\gamma} H^i(\left(\cquohur{G}{K}{n} {c}{T} \right)^{\std}, \mathbb
	Q_\ell).
\end{align*}

\begin{proposition}
	\label{proposition:splitting}
	Suppose $B$ is the spectrum of a Henselian dvr  with residue field
	$\mathbb F_q$ and generic characteristic $0$. 
	Suppose $W \subset \phur{G'}{Mdr} {c}{\mathbb F_q}$
is the component indexed by  $\prod_{j=0}^{r-1}
[g^{q^j}]^{Md}$,
as in 
\autoref{theorem:frob-equivariant-stabilization}, with
$x \in W(\mathbb F_q)$.
	If $T$ is either $\spec \overline{\mathbb F}_q$ or $\spec \mathbb C$,
	there is a canonical splitting
	\begin{align*}
		H^i(\left(\cquohur{G}{K}{n} {c}{T} \right)^{\std}, \mathbb Q_\ell) \xrightarrow{\varepsilon} H^i( \cquohur{G}{K}{n} {c}{T}, \mathbb
	Q_\ell).
	\end{align*}
	of $\gamma$, i.e., $\varepsilon \circ \gamma = \id$.
	In the case $T=\spec \overline{\mathbb F}_q$, this splitting is
	Frobenius equivariant.

	Moreover, the composition of
	\eqref{equation:cohomology-composite-to-standard} with $\varepsilon$
	yields a map $H^i([{\mathscr H}_{n+Mdr, T}^{c}/K], \mathbb Q_\ell) \to 
H^i(\cquohur{G}{K}{n} {c}{T}, \mathbb Q_\ell)$ which agrees with the restriction of $U^{g,q,M,K}$ as in
\eqref{equation:transfer-u}
when $T = \spec \mathbb C$ and
agrees with $U^{g,q,M,K}_{\overline{\mathbb F}_q}$ 
when $T = \spec \overline{\mathbb
F}_q$.
\end{proposition}
\begin{proof}
	The proof of this is analogous to \cite[Proposition
	A.4.4]{ellenbergL:homological-stability-for-generalized-hurwitz-spaces}
	and we omit the details.
	(However, it is slightly simpler because the map $\eta$ in 
	\cite[Proposition
	A.4.4]{ellenbergL:homological-stability-for-generalized-hurwitz-spaces}
	does not show up for us. Said another way, we can treat the map $\eta$
	defined there as the identity.)
\end{proof}

Combining our work so far easily yields a proof of 
\autoref{theorem:frob-equivariant-stabilization}.

\subsubsection{Proof of
\autoref{theorem:frob-equivariant-stabilization}}
\label{subsection:equivariant-proof}
By 
\autoref{proposition:splitting},
the composite of the map 
\eqref{equation:cohomology-composite-to-standard}
with
$\varepsilon$ defined in 
\autoref{proposition:splitting}
yields a map
which agrees with the restriction of
the map $U^{g,q,M,K}_{\overline{\mathbb F}_q}$ from
\eqref{equation:transfer-u}.
Hence, it suffices to show
$\varepsilon$ and
\eqref{equation:cohomology-composite-to-standard},
which in turn is the composite of the maps $\delta$ and $\alpha^*_x$,
are both equivariant for the actions of Frobenius.
The equivariance of 
$\varepsilon$ was stated in \autoref{proposition:splitting}.
Next, $\alpha^*_x$ is equivariant for Frobenius 
because it was induced from the base change of a map of log schemes over
$\mathbb F_q$;
hence the action of Frobenius on cohomology is equivariant because it is
functorially induced by an equivariant action
of Frobenius on these log schemes.
Finally, $\delta$ is equivariant for the actions of Frobenius because it is the
base change of an isomorphism over $\mathbb F_q$ coming from
\cite[Corollary 7.5]{illusie:an-overview}.
\qed
\newline

\subsection{Complements to Frobenius stabilization}
\label{subsection:verifying-hypotheses}

As a complement to \autoref{theorem:frob-equivariant-stabilization},
we would like to show its hypotheses are often verified. That is, we would like
to show that the component $W$ defined there frequently has many $\mathbb F_q$
points.

\begin{lemma}
	\label{lemma:w-irreducible}
	For $M$ sufficiently large, the component $W$ from
	\autoref{theorem:frob-equivariant-stabilization} is geometrically
	irreducible.
\end{lemma}
\begin{proof}
	By \cite[Corollary 12.6]{liuWZB:a-predicted-distribution} we can identify
	geometric components of Hurwitz spaces where all conjugacy classes
	appear sufficiently many times
	with their lifting invariants in the sense of
	\cite[Theorem 5.2]{wood:an-algebraic-lifting-invariant}.
	We use the notation $U(G,c)$ for what we defined as $U(c)$ in
	\autoref{lemma:components-factors}.
	Let $c \subset G$ denote 
	the union of the $G$ conjugacy classes containing $g, g^q, \ldots,
	g^{q^{r-1}}.$
	By \cite[Theorem 12.1]{liuWZB:a-predicted-distribution}(2)
	we wish to show that the Frobenius action on 
	the lifting invariant (corresponding to descent data for the component
	from $\overline{\mathbb F}_q$ to $\mathbb F_q$) associated to $W$ is trivial.
	In other words, if $q^{-1} * g$ denotes the discrete action of $q^{-1}$
	on $g \in U(G, c)$, as defined in
	\cite[\S4, p. 8]{wood:an-algebraic-lifting-invariant},
	we wish to show that 
	$q^{-1} * \left(\prod_{j=0}^{r-1} [g^{q^j}]^{Md} \right) = \prod_{j=0}^{r-1} [g^{q^j}]^{Md}$.
Indeed, 
using that $[g]^{\on{ord(g)}}$ is central in $U(G,c)$
and $q$th powering is invertible on $c$,
so $g^{q^j}$ has the same order as $g$,
\begin{align*}
	q^{-1} * \left(\prod_{j=0}^{r-1} [g^{q^j}]^{Md} \right) &=
\left(\prod_{j=0}^{r-1} (([(g^{q^j})^{q^{-1}} ])^{q})^{Md}\right)^{q^{-1}} \\
&= \left(\prod_{j=0}^{r-1} (([(g^{q^j})^{q^{-1}} ])^{Md})^q) \right)^{q^{-1}} \\
&= \prod_{j=0}^{r-1} ([(g^{q^j})^{q^{-1}} ])^{Md} \\
&= \prod_{j=0}^{r-1} ([g^{q^{j-1}} ])^{Md} \\
&= \prod_{j=0}^{r-1} ([g^{q^{j}} ])^{Md}.
\qedhere
\end{align*}
\end{proof}
The above shows the component $W$ is geometrically irreducible, so we would next
like to show geometrically irreducible components have many $\mathbb F_q$
points.
\begin{lemma}
	\label{lemma:exponential-cohomology-bound}
	For any finite rack $c$,
	there is some constant $K$, depending on $c$, so that 
	$\dim H_i(\cphurc n c ;\mathbb Q) \leq K^{i+1}$.
\end{lemma}
\begin{proof}
	First, 	suppose $n = \sum_{j=0}^\upsilon n_j$ and $Z \subset \cphurc {n_1,
	\ldots, n_\upsilon} c$.
Using \autoref{theorem:some-large-homology-stabilizes}, 
if $n_\lambda > Ii + J$, we can identify $\dim H_i(\cphurc {n_1,
\ldots,n_\lambda+1,\ldots,  n_\upsilon} c ; \mathbb Q)$ with
	$\dim H_i(\cphurc {n_1,\ldots, n_\lambda, \ldots, n_\upsilon} c ;\mathbb
	Q)$, and so we may assume that $n_1, \ldots, n_\upsilon \leq Ii + J$,
	and hence $n \leq \upsilon(Ii+J)$.
	Therefore, it suffices to bound
	$\dim H_i(\cphurc n c ;\mathbb Q) \leq K^{i+1}$ for $n <
	\upsilon(Ii+J)$.
	By the same argument as in \cite[Proposition 2.5]{EllenbergVW:cohenLenstra},
	$\dim H_i(\cphurc n c ;\mathbb Q) \leq (2 |c|)^n,$
	so it suffices to find some $K$ so that $(2|c|)^n < K^{i+1}$ for all $n
	< \upsilon(Ii + J)$.
	This holds upon taking $K := \max((2|c|)^{\upsilon J}, (2|c|)^{\upsilon I})$.
\end{proof}

\begin{remark}
	\label{remark:transfer}
	In what follows, we will need to repeatedly use that
	there is an isomorphism between the cohomology of our Hurwitz spaces
	over $\overline{\mathbb F}_q$ and over $\mathbb C$.
	This follows from
	\cite[Proposition 7.7]{EllenbergVW:cohenLenstra},
	which requires the existence of a normal crossings compactification of
	Hurwitz space, which is provided by 
	\cite[Corollary
	B.1.4]{ellenbergL:homological-stability-for-generalized-hurwitz-spaces}.
\end{remark}

We next obtain a bound on the number of finite field points of a
component of Hurwitz space.
For $X$ a groupoid, we use the notation $| X | := \sum_{x \in G}
\frac{1}{ |\aut(x)|}$ to denote the groupoid cardinality of $X$.
We also use $\|
x\|$ to denote the absolute value of a complex number $x$.
\begin{lemma}
	\label{lemma:weak-point-bound}
	Fix a finite group $G'$, a normal subgroup $G \subset G'$, and a union of conjugacy classes $c \subset G$ as in
	\autoref{notation:frobenius-stabilization}. 
	Let $H \subset G'$ be a (possibly trivial) subgroup. Let $q$ be a prime
	power with $\gcd(q, |G|) = 1$ such that $c$ is closed under $q$th
	powering. Fix a geometrically connected component 
	$Z \subset \cquohur {G}{H} n {c} {\mathbb F_q}$.	
Suppose $q^{1/2} >
	K$, for $K$ as in \autoref{lemma:exponential-cohomology-bound}.
	Then, is some constant $D_q$, depending on $c$ and $q$ but not on $n$ or $Z$, so that 
	$\| |Z(\mathbb F_q)| -q^n \| < D_q q^{n-1/2}$.
	Moreover $D_q$ is bounded as a function of $q$ as $q$ tends to $\infty$.
\end{lemma}
\begin{proof}
	This follows in a standard fashion from the Grothendieck-Lefschetz trace formula, Deligne's
	bounds on the eigenvalues of Frobenius, and
	\autoref{lemma:exponential-cohomology-bound}. We now spell out the
	details.
	The Grothendieck Lefschetz trace formula yields
	\begin{align*}
		\frac{|Z(\mathbb F_q)|}{q^n} &= \sum_{i=0}^{2n}(-1)^i \on{tr}\left(
		\on{Frob}_q^{-1}| H^i(Z_{\overline{\mathbb F}_q}, \mathbb
	Q_\ell)	\right).
	\end{align*}
	Since $Z$ is geometrically irreducible, and the eigenvalue of Frobenius on $H^0(Z_{\overline{\mathbb F}_q}, \mathbb
	Q_\ell)$ can be read off from its action on the geometric
	components of $Z$, we find that the eigenvalue of Frobenius on
	$H^0(Z_{\overline{\mathbb F}_q}, \mathbb
	Q_\ell)$ is $1$.
	Combining this with
Sun's generalization to algebraic stacks of Deligne's bounds on the
eigenvalues of Frobenius \cite[Theorem 1.4]{sun:l-series-of-artin-stacks},
	we obtain
\begin{align*}
	\left \| \frac{|Z(\mathbb F_q)| -q^n}{q^n} \right \|  &= \left \| \sum_{i=1}^{2n}(-1)^i \on{tr}\left(
		\on{Frob}_q^{-1}| H^i(Z_{\overline{\mathbb F}_q}, \mathbb
	Q_\ell)	\right) \right \| \\
	&\leq \sum_{i=1}^{2n} q^{-i/2} \dim H^i(Z_{\overline{\mathbb F}_q}, \mathbb Q_\ell)
	\\
	&\leq \sum_{i=1}^\infty q^{-i/2} K^{i+1} \\
	&\leq K^2 q^{-1/2} \sum_{i=0}^\infty q^{-i/2} K^i \\
	&\leq K^2 q^{-1/2} \frac{1}{1- Kq^{-1/2}}.
\end{align*}
The second inequality above uses 
\autoref{lemma:exponential-cohomology-bound}
and the fact that
$\dim H^i(Z_{\overline{\mathbb F}_q}, \mathbb Q_\ell)$
can be bounded above by the dimension of 
the $i$th cohomology of a component of
$\cphur {G'}n {c} {\overline{\mathbb F}_q}$ over $\overline{\mathbb F}_q$,
which can in turn be identified with the $i$th singular cohomology of a
component of the complex variety
$\cphurc n {c}$
using \autoref{remark:transfer}.
We can then take $D_q := \frac{K^2 }{1- Kq^{-1/2}}$.
As $q$ grows, $D_q$ tends to the constant $K^2$.
\end{proof}

We now record the main consequence of
\autoref{theorem:frob-equivariant-stabilization} we will need for future
applications. This gives a good approximation of the number of $\mathbb F_q$
points of components of Hurwitz spaces, and shows this number is periodic, in a
suitable sense, using the Frobenius equivariant stabilization from
\autoref{theorem:frob-equivariant-stabilization}.

\begin{lemma}
	\label{lemma:component-point-bound}
	Fix a prime power $q$.
	Suppose $G'$ is a finite group, $G \subset G'$ is a normal subgroup, and
	$c \subset G'$ is a union of conjugacy classes of $G'$ which is moreover
	contained in $G$ and closed under $q$th powering.
	Let $H \subset G'$ be a (possibly trivial) subgroup.
	Using notation from
	\autoref{notation:frobenius-stabilization},
	fix $g \in c_1 \subset c$ and let $s$ denote the associated constant defined
	in 
	\autoref{notation:frobenius-stabilization}.
	There are constants $C, I$ and $J$
	with the following properties.
	\begin{enumerate}
		\item Fix a non-negative integer $i$.
			Write $(\alpha, (n_1, \ldots, n_\upsilon))$ to index a
			component of
			$\cphurc {n_1, \ldots, n_{\upsilon}} {c}$ with $n = n_1 + \cdots
			+ n_\upsilon$.
			Suppose that $j$ is an integer satisfying $s \leq j \leq \upsilon$, 
	$q > C$ is a prime power with $\gcd(q, |G|) = 1$,
	and $n_1, \ldots, n_j > Ii + J$.
	Assume $Z$ is a geometrically irreducible component of 
	$\cquohur{G}{H}{n} {c}{\mathbb F_q}$
	corresponding to the $H$ orbit of $(\alpha, (n_1, \ldots, n_\upsilon))$, in the
	sense of \autoref{lemma:g-irred-components}.
	Then,
	there is a constant $\phi_{(\alpha, (n_1, \ldots, n_\upsilon)),c,G, H,q}$,
	so that
\begin{align}
	\label{equation:point-count-to-phi}
	\left | \frac{|Z(\mathbb F_q)|}{q^n} - \phi_{(\alpha, (n_1, \ldots,
	n_\upsilon)),c,G,H,q} \right | \leq
	\frac{2C}{1-\frac{C}{\sqrt{q}}} \left( \frac{C}{\sqrt{q}}
	\right)^{\frac{n-J}{I}}.
\end{align}
\item Using notation as in the previous part, 
	\begin{align*}
	\phi_{(\alpha, (n_1, \ldots, n_\upsilon)),c, G,H, q} =
	\phi_{(\alpha', (n_1 + dr/s, \ldots, n_s + dr/s, n_{s+1}, \ldots, n_\upsilon),
		c,G,H, q},
	\end{align*}
where the image of the component associated to $\alpha$ under the map
$U_{\overline{\mathbb F}_q}^{g,q,M+1,K}$
from \autoref{theorem:frob-equivariant-stabilization}
agrees with the image of the component associated to $\alpha'$ under
$U_{\overline{\mathbb F}_q}^{g,q,M,K}$.
In particular,
\begin{align*}
	\phi_{(\alpha, (n_1, \ldots, n_\upsilon)),c, G,H, q} =
\phi_{(\alpha'', (n_1 + |G|^2, \ldots, n_s + |G|^2, n_{s+1}, \ldots, n_\upsilon)), c,G,H,
q}
\end{align*}
for a suitable component $\alpha''$.
	\end{enumerate}
\end{lemma}
\begin{remark}
	\label{remark:}
	The constant $\phi_{(\alpha, (n_1, \ldots,
	n_\upsilon)),c,G,H,q}$ can be interpreted as the trace of arithmetic Frobenius on
	the stable cohomology of
	the component $\alpha$, where one stabilizes ``in the direction of
	$g$.''
\end{remark}

\begin{proof}
	For the purposes of proving (1), we start by defining a constant $L$.
	Using \autoref{lemma:weak-point-bound}, there is some $L$ so that for
	any $n> L$,
	any geometrically integral
	component $Z \subset \cquohur {G}{H} n {c}  {\mathbb
	F_q}$, satisfies $Z(\mathbb F_q) \neq
	\emptyset$ once $q > C$, where we can take $C = K^2$ with $K$ as in
	\autoref{lemma:weak-point-bound}.

	We will explain why (1) follows from 
	\cite[Lemma 5.2.2]{landesmanL:the-stable-homology-of-non-splitting}
	applied to any sequence of spaces $Y_n$, where $Y_n$ is
	indexed by
	the geometrically irreducible component of 
	$\cquohur {G}{H} {n} c {\mathbb F_q}$
	corresponding to $(\alpha_n, n_1, \ldots, n_\upsilon)$ 
		with $n_{s+1}, \ldots, n_\upsilon$ fixed but $n_1, \ldots, n_s$ varying,
	such that $n_1, \ldots, n_s > Ii + J$,
	$n := \sum_{j=1}^\upsilon n_j$, 
	with the component of $\alpha_{i_1}$ mapping to the component of
	$\alpha_{i_2}$ under some $U^{g,q,M,H}$,
	and
	such that there are certain residues $r_1, \ldots, r_s$ modulo
	$dr/s$ 
	with
	$n_t \equiv r_t \bmod rd/s$ for $1 \leq t \leq s$.

	Fix a prime $\ell$ which is prime to $q$ and $|G'|$.
	If we
fix values of $n_{s+1}, \ldots, n_\upsilon$
	and let $n_1, \ldots, n_s > Ii + J$ such that $n_t \equiv r_t \bmod dr/s$ for
	$1 \leq t \leq j$,
	we obtain from \autoref{theorem:some-large-homology-stabilizes} 
and \autoref{theorem:frob-equivariant-stabilization} (to compare $\mathbb C$
with $\overline{\mathbb F}_q$)
that the multiplication map 
$U^{g,q,M,H}_{\overline{\mathbb F}_q}$
on $i$th cohomology 
in 
\autoref{theorem:frob-equivariant-stabilization},
is an isomorphism.
The hypothesis of 
\autoref{theorem:frob-equivariant-stabilization}
that $W(\mathbb F_q) \neq \emptyset$
is verified for $Mdr > L$ using \autoref{lemma:weak-point-bound}
and \autoref{lemma:w-irreducible}.
Since
\autoref{theorem:frob-equivariant-stabilization},
shows the isomorphism 
$U^{g,q,M,H}_{\overline{\mathbb F}_q}$
is also frobenius equivariant, 
	we can choose constants $I$ and $J$ so that 
	the action of $\on{Frob}_q^{-1}$ on 
	$H^i(Z_{\overline{\mathbb F}_q}, \mathbb Q_\ell)$
	is independent of the choice of $n_1, \ldots, n_\upsilon$ so long as $n_{s+1}, \ldots, n_\upsilon$
	are fixed, $n_1, \ldots, n_s > Ii + J$ and $n_t \equiv r_t \bmod dr/s$ for
	$1 \leq t \leq s$.
	Applying the above statement both for 
$U^{g,q,M+1,H}_{\overline{\mathbb F}_q}$ and for
$U^{g,q,M,H}_{\overline{\mathbb F}_q}$,
we find that the traces of $\on{Frob}_q^{-1}$ on the stable cohomology of the
		component
		$(\alpha_n, (n_1, \ldots,
	n_\upsilon))$ agrees with that on 
	$(\alpha_{n+Mdrs+dr}, (n_1+M dr+dr/s, \ldots, n_s+Mdr+dr/s,n_{s+1}, \ldots, n_\upsilon))$,
	which in turn agrees with that on 
	$(\alpha_{n+dr}, (n_1+dr/s, \ldots, n_s+dr/s,n_{s+1}, \ldots, n_\upsilon))$.

	We now fix $n_1, \ldots, n_\upsilon$.
	We can then take $V_t$ to be the vector space with action of geometric
	Frobenius, $\on{Frob}_q$ equal to $H^t(Z_{\overline{\mathbb F}_q},
	\mathbb Q_\ell)$
	for $Z$ some geometrically irreducible component corresponding to
	$(\alpha_n, (n'_1, \ldots, n'_s, n_{s+1},\ldots, n_\upsilon))$ with 
	$n'_1, \ldots, n'_s > It + J$ and $n'_r-n_r = Mrd/s$ for some value of
	$M$ and $1 \leq
	r \leq s$.
	Take $\phi_{(\alpha_n, (n_1, \ldots, n_\upsilon)),c,G,H,q} := \sum_{i=0}^\infty
	(-1)^i \tr\left( \on{Frob}_q^{-1}|V_i \right)$.
	By construction we see the first statement in part (2) of this lemma holds.
	The second statement in $(2)$ follows from the first statement in $(2)$
	because $\frac{r}{s} \mid |G|$ and $d \mid |G|$ so
	$\frac{dr}{s} \mid |G|^2$.

	We conclude by proving part (1).
This will follow from 
\cite[Lemma 5.2.2]{landesmanL:the-stable-homology-of-non-splitting}
once we verify the hypotheses (1) and (2) there.
Indeed, hypothesis (1) there follows from what we have done so far in this
proof.
To verify hypothesis (2), it wish to bound $\dim H^i(Y_{n, \overline{\mathbb F}_q}, \mathbb
Q_{\ell}) \leq C'\cdot C^i$ for some constant $C$, which follows from
\autoref{lemma:exponential-cohomology-bound} taking $C = C' = K$.
\end{proof}

We will also use the following consequence of Frobenius stabilization for
computing the moments predicted by Cohen--Lenstra--Martinet, which
compares the point counts of two related Hurwitz spaces.
\begin{lemma}
	\label{lemma:component-comparison}
	Suppose $G_1$ and $G_2$ are two finite groups and $c_1 \subset G_1$ and
	$c_2 \subset G_2$ are two unions of $\upsilon$
	conjugacy classes, both closed under $q$th powering.
	Assume we have a group homomorphism $G_1 \to G_2$ inducing a map $c_1 \to
	c_2$.
	There are constants $C, I$, and $J$ (depending on both $c_1$ and $c_2$)
	with the following property.
	Let $q$ be a prime power with $q > C$, and $\gcd(q, |G_1||G_2|) = 1$.
	Suppose
	$n_1, \ldots, n_\upsilon > Ii + J$.
Suppose
$Z_1 \subset\cphur {G_1} {n_1, \ldots, n_\upsilon} {c_1} {\mathbb F_q}$
is a geometrically irreducible component,
indexed by $\alpha_1$,
mapping to 
$Z_2 \subset\cphur {G_2} {n_1, \ldots, n_\upsilon} {c_2} {\mathbb F_q}$,
indexed by $\alpha_2$.
Using notation from \autoref{lemma:component-point-bound},
$\phi_{(\alpha_1, (n_1, \ldots, n_\upsilon)), c_1, G_1, \id, q} = \phi_{(\alpha_2, (n_1,
	\ldots, n_\upsilon)), c_2, G_2,\id,
q}$.
\end{lemma}
\begin{proof}
First, we claim
$H^i(Z_1, \mathbb Q_\ell) \to H^i(Z_2, \mathbb Q_\ell)$
is an isomorphism.
To see this, we have maps
$Z_{1, \overline{\mathbb F}_q} \to Z_{2, \overline{\mathbb F}_q} \to 
\conf_{n_1, \ldots, n_\upsilon, \overline{\mathbb
F}_q}$.
The composite induces an isomorphism on $H^i$ and the second map induces an
isomorphism on $H^i$ by
\autoref{theorem:all-large-stable-homology} and \autoref{remark:transfer}.
Hence,
$H^i(Z_1, \mathbb Q_\ell) \to H^i(Z_2, \mathbb Q_\ell)$
is an isomorphism.

Since 
$H^i(Z_1, \mathbb Q_\ell) \to H^i(Z_2, \mathbb Q_\ell)$
is an isomorphism, the trace of geometric Frobenius on 
$H^i(Z_{m,\overline{\mathbb F}_q}, \mathbb Q_\ell)$ is independent of $m \in
\{1,2\}$, for
$n_1, \ldots, n_j > iI + J$.
From this it follows that
$\phi_{(\alpha_m, (n_1, \ldots, n_\upsilon)), c_m, G_m, \id,q}$
is independent of $m$, as desired.
\end{proof}

\section{Application to the Cohen-Lenstra-Martinet heuristics}
\label{section:clm}

In this section, we prove our main result toward the Cohen--Lenstra--Martinet
heuristics.
In order to give the proof, we first introduce some notation.

\begin{notation}
	\label{notation:stable-components}
	Fix a finite group $G$ and a union of $\upsilon$ conjugacy classes $c = c_1
	\cup \cdots \cup c_\upsilon$ for $c_i \subset G$ conjugacy classes so that
	$c$ generates $G$.
	Following \cite{wood:an-algebraic-lifting-invariant},
	we define $S_c \to G$ to be some reduced Schur cover for $c$ as in
	\cite[Definition, p. 3]{wood:an-algebraic-lifting-invariant}
	and $\hat{G} := S_c \times_{G^{\ab}} \mathbb Z^\upsilon$.
	Note that $\hat{G}$ implicitly depends on $c$.
	By \cite[Theorem 2.5 and Theorem
	3.1]{wood:an-algebraic-lifting-invariant}, elements of $\hat{G}$ such
that the projection to $\mathbb Z^\upsilon$ is $(n_1, \ldots, n_\upsilon)$ with all $n_i$ sufficiently large are in bijection
	with components of
	$\cphurc {n_1, \ldots, n_\upsilon} c$.
\end{notation}

We now aim to prove \autoref{theorem:clm}.
We recall the statement now.

\clm*

The proof is very similar to that of \cite[Theorem
1.1]{liu:non-abelian-cohen-lenstra-in-the-presence-of-roots} except that we use
our main homological stability results to obtain better control of the finite field point counts and
thereby remove the large $q$ limit appearing in 
\cite[Theorem 1.1]{liu:non-abelian-cohen-lenstra-in-the-presence-of-roots}.
It will also be crucial to know this homological stability holds Frobenius
equivariantly, as we established in
\autoref{theorem:frob-equivariant-stabilization}.

\subsubsection{Proof of \autoref{theorem:clm}}
\label{subsubsection:proof-clm}
	Let $G := H \rtimes \Gamma$.
	Let $c_2 := \Gamma - \{\id\} \subset \Gamma$ and let $c_1 \subset
	G-\{\id\}$ denote the set of elements with the same order as their image
	in $\Gamma$.
	Suppose $c_2$ consists of $P$ conjugacy classes.
	It is argued in the first two paragraphs of the proof of
	\cite[Theorem 10.4]{liuWZB:a-predicted-distribution}, using admissibility
	of the $\Gamma$ action, that $c_1$
	is also a union of $P$ conjugacy classes.
	
	Fix $\delta: \hat{\mathbb Z}(1)_{(|\Gamma| q)'} \to H_2(H \rtimes
	\Gamma, \mathbb Z)_{(|\Gamma| q)'}$ and let
	$Z^\delta_{q,n}$ denote the union of geometrically irreducible
	components of $\cphur {G} n {c_1}
	{\mathbb F_q}$ whose $\overline {\mathbb F}_q$ points have $\omega$
	invariant $\delta$, in the sense of \cite[Definition
	2.13]{liu:non-abelian-cohen-lenstra-in-the-presence-of-roots}.
		It follows from
	\cite[Lemma 4.6]{liu:non-abelian-cohen-lenstra-in-the-presence-of-roots}
	that
	\begin{align}
		\label{equation:numerator-clm-count}
		\sum_{K
			\in E_\Gamma(q^n, \mathbb F_q(t)} \left | \left\{ \pi \in \surj_\Gamma(
	G^\sharp_\emptyset(K), H) : \pi_* \circ \omega_{K^\sharp/K} = \delta
\right\}  \right |
&=
\frac{|Z_{q,n}^\delta (\mathbb F_q)|}{[H:H^\Gamma]}.
	\end{align}
	Moreover, when we take $H$ to be the trivial group,
	it follows from 
	\cite[Lemma 4.6]{liu:non-abelian-cohen-lenstra-in-the-presence-of-roots}
	that
	\begin{align}
		\label{equation:denominator-clm-count}
		| E_{\Gamma}(q^n, \mathbb F_q(t)) | = 
		|\cphur {\Gamma} n {c_2} {\mathbb F_q}(\mathbb F_q)|.
	\end{align}
	Therefore, plugging \eqref{equation:numerator-clm-count} and
	\eqref{equation:denominator-clm-count} into \eqref{equation:clm-ratio},
		to complete the proof, it suffices to prove
	\begin{equation}
		\label{equation:ratio-clm-count}	
		\frac{\sum_{n \leq N} |Z_{q,n}^\delta(\mathbb F_q)|}{\sum_{n
		\leq N} |\cphur {\Gamma} n {c_2} {\mathbb F_q}(\mathbb F_q)|}
	= 1 + O_G(N^{-1}).
	\end{equation}
		
Let $\pi_{G,c_1}^\delta(q,n)$ denote the number of components
	of $Z^\delta_{q,n}$ and let 
	$\pi_{\Gamma,c_2}(q,n)$ denote the 
	number of geometrically irreducible components of
	$\cphur {\Gamma} n {c_2} {\mathbb F_q}$.
Let $d_{G,c_1}(q)$ be the number of orbits of the $q$th powering action on
$c/G$, which sends a conjugacy class to the conjugacy class of its $q$th power.
Then, it follows by combining 
\cite[Lemma 4.4 and Lemma 4.5]{liu:non-abelian-cohen-lenstra-in-the-presence-of-roots}
with
\cite[Proposition 12.7]{liuWZB:a-predicted-distribution}
that 
$\pi_{G,c_1}^\delta(q,n) = \pi_{\Gamma, c_2}(q,n) + O(n^{d_{G,c_1}(q)-2})$,
while
$\pi_{G,c_1}^\delta(q,n)$ is a polynomial in $n$ of degree 
$d_{G,c_1}(q)-1$.
Combining the above with 
\cite[Corollary 12.9]{liuWZB:a-predicted-distribution},
we also obtain $d_{\Gamma,c_2}(q) = d_{G, c_1}(q)$.

Next, we make two observations.
First, 
by \autoref{lemma:weak-point-bound},
there is some constant $D_q$, independent of $n$ so that for any geometrically
connected component $W$ either in $Z_{q,n}^\delta$ or in
$\cphur {\Gamma} n {c_2} {\mathbb F_q}$, we have
$|W(\mathbb F_q)| \leq D_q q^n$.

Recall from \autoref{notation:stable-components} that the stable components of
Hurwitz space are indexed by elements $(\alpha, (n_1, \ldots, n_\upsilon)) \in \hat{G}.$
Our second observation is that there is a subset of $O_G(n^{d_{G,c_1}(q)-2})$
many geometrically irreducible components of
$Z_{q,n}$ and of $\cphur {\Gamma} n {c_2} {\mathbb F_q}$
and constants
$C, I, J$
such that the following holds:
for any geometrically irreducible $W$ outside of that set,
associated to $\alpha$,
there is a specific number
$\phi_{(\alpha, (n_1, \ldots, n_\upsilon)), c_1,G,\id,q}$
from \autoref{lemma:component-point-bound}
such that, if $q > C$ is a prime power with $\gcd(q, |G|) = 1$,
\begin{align}
	\label{equation:stable-component-estimate}
	\left | \frac{|W(\mathbb F_q)|}{q^n} - \phi_{(\alpha, (n_1, \ldots,
	n_\upsilon)), c_1,G,\id,q} \right | \leq
	\frac{2C}{1-\frac{C}{\sqrt{q}}} \left( \frac{C}{\sqrt{q}}
	\right)^{\frac{n-J}{I}}.
\end{align}
This follows from \autoref{lemma:component-point-bound}, using that the number
of geometrically irreducible components with some 
component $n_j \leq Ii + J$ accounts for 
$O_G(n^{d_{G,c_1}(q)-2})$ many geometrically irreducible components of
$Z_{q,n}$ and of $\cphur {\Gamma} n {c_2} {\mathbb F_q}$, as is explained in the third paragraph of the
proof of \cite[Proposition 12.7]{liuWZB:a-predicted-distribution}.

To use the above two observations, we next note that
the proof of \cite[Lemma 4.5]{liu:non-abelian-cohen-lenstra-in-the-presence-of-roots}
not only shows
$\pi_{G,c_1}^\delta(q,n) = \pi_{\Gamma, c_2}(q,n) + O(n^{d_{G,c_1}(q)-2})$ but this equality
is induced by a specific bijection obtained from the map $c_1 \to c_2$.
To state this precisely, we introduce some more notation.
For each $\delta: \hat{\mathbb Z}(1)_{(|\Gamma| q)'} \to H_2(H \rtimes
\Gamma, \mathbb Z)_{(|\Gamma| q)'}$, let $S^\delta_{q,n}$ denote the set of geometrically irreducible components of
$\cphur {G} n {c_1} {\mathbb F_q}$ with $\omega$ invariant $\delta$, which are in the set to which 
\eqref{equation:stable-component-estimate} applies.
Similarly, let
$S_{q,n}$ denote the set of geometrically irreducible components of
$\cphur {\Gamma} n {c_2} {\mathbb F_q}$ to which 
\eqref{equation:stable-component-estimate} applies.
By modifying these sets $S^{\delta}_{q,n}$ and $S_{q,n}$ by at most
$O_G(n^{d_{G,c_1}(q)-2})=O_G(n^{d_{\Gamma,c_2}(q)-2})$ elements, we can arrange
that $S^{\delta}_{q,n}$ maps bijectively to $S_{q,n}$
under the map induced by $c_1 \to
c_2$,
as follows from the proof
of 
\cite[Lemma 4.5]{liu:non-abelian-cohen-lenstra-in-the-presence-of-roots}.

We now combine the above two observations to 
complete the verification of
\eqref{equation:ratio-clm-count}.
The above two observations imply
we can estimate $|Z^\delta_{q,n}(\mathbb F_q)|$ as
\begin{equation}
\begin{aligned}
	\label{equation:z-estimate}
|Z^\delta_{q,n}(\mathbb F_q)|
=
&\sum_{(\alpha_1, (n_1, \ldots, n_\upsilon)) \in S^\delta_{q,n}}
\phi_{(\alpha_1, (n_1, \ldots, n_\upsilon)), c_1,G,\id, q}  \left(1 + 
M\frac{2C}{1-\frac{C}{\sqrt{q}}} \left( \frac{C}{\sqrt{q}}
	\right)^{\frac{n-J}{I}} \right)q^n \\
	&+ O_G(n^{d_{G,c_1}(q)-2})q^n
\end{aligned}
\end{equation}
for some $M \in [-1,1]$.
Similarly, we can estimate
\begin{equation}
\begin{aligned}
	\label{equation:hur-gamma-estimate}
	|\cphur {\Gamma} n {c_2} {\mathbb F_q}(\mathbb F_q)| =
	&\sum_{(\alpha_2, (n_1, \ldots, n_\upsilon)) \in S_{q,n}}
	\phi_{(\alpha_2, (n_1, \ldots, n_\upsilon)), c_2,\Gamma,\id, q}\left(1 + M'
\frac{2C}{1-\frac{C}{\sqrt{q}}} \left( \frac{C}{\sqrt{q}}
	\right)^{\frac{n-J}{I}} \right) q^n \\
	&+ O_G(n^{d_{\Gamma,c_2}(q)-2})q^n
\end{aligned}
\end{equation}
for some $M' \in [-1,1]$.
Recall here that 
$\sum_{(\alpha_1, (n_1, \ldots, n_\upsilon)) \in S^\delta_{q,n}}
\phi_{(\alpha_1, (n_1, \ldots, n_\upsilon)), c_1,G,\id, q}$
as well as
$\sum_{(\alpha_2, (n_1, \ldots, n_\upsilon)) \in S_{q,n}}
	\phi_{(\alpha_2, (n_1, \ldots, n_\upsilon)), c_2,\Gamma,\id, q}$
both grow as
$q^n n^{d_{\Gamma,c_1}(q)-1}=q^n n^{d_{\Gamma,c_2}(q)-1}$.
Hence, the error terms 
$O_G(n^{d_{\Gamma,c_i}(q)-2})q^n$
are roughly a factor of $1/n$ smaller than the main term.
Moreover, for $n$ sufficiently large, we may bound 
$M' \frac{2C}{1-\frac{C}{\sqrt{q}}} \left( \frac{C}{\sqrt{q}}
	\right)^{\frac{n-J}{I}}$
	and
$M \frac{2C}{1-\frac{C}{\sqrt{q}}} \left( \frac{C}{\sqrt{q}}
	\right)^{\frac{n-J}{I}}$
by $O_G(1/n)$ since we are assuming $q$ is sufficiently large so that $C <
\sqrt{q}$, and so eventually this exponentially decaying term will be dominated
by the polynomially decaying function $1/n$.
This means that for $n$ sufficiently large, we may simplify
\eqref{equation:z-estimate} and \eqref{equation:hur-gamma-estimate}
to
\begin{equation}
\begin{aligned}
	\label{equation:z-estimate-simplified}
|Z^\delta_{q,n}(\mathbb F_q)|
=
&\sum_{(\alpha_1, (n_1, \ldots, n_\upsilon)) \in S^\delta_{q,n}}
\phi_{(\alpha_1, (n_1, \ldots, n_\upsilon)), c_1,G,\id, q} q^n + O_G(n^{d_{G,c_1}(q)-2})q^n
\end{aligned}
\end{equation}
and
\begin{equation}
\begin{aligned}
	\label{equation:hur-gamma-estimate-simplified}
	|\cphur {\Gamma} n {c_2} {\mathbb F_q}(\mathbb F_q)| =
	&\sum_{(\alpha_2, (n_1, \ldots, n_\upsilon)) \in S_{q,n}}
	\phi_{(\alpha_2, (n_1, \ldots, n_\upsilon)), c_2,\Gamma,\id, q}q^n+ O_G(n^{d_{\Gamma,c_2}(q)-2})q^n.
\end{aligned}
\end{equation}

We showed above that
the map $c_1 \to c_2$ induces a bijection $S^\delta_{q,n} \to S_{q,n}$
and also that $d_{\Gamma,c_2}(q) = d_{G, c_1}(q)$.
Since
$\phi_{(\alpha_1, (n_1, \ldots, n_\upsilon)), c_1,G,\id, q}
=
\phi_{(\alpha_2, (n_1, \ldots, n_\upsilon)), c_2,\Gamma,\id, q}$
by 
\autoref{lemma:component-comparison},
it follows that upon summing \eqref{equation:z-estimate-simplified}
over $n \leq N$
and dividing it by the sum of \eqref{equation:hur-gamma-estimate-simplified} over $n \leq
N$, we obtain
\eqref{equation:ratio-clm-count}.
\qed

\section{Malle's conjecture}
\label{section:malle}

Our final application is to prove several versions of Malle's conjecture.
While the inverse Galois problem predicts the number of $G$-extensions of
$\mathbb Q$ or $\mathbb F_q(t)$ is nonzero,
Malle's conjecture goes further and predicts the asymptotic number of
$G$-extensions.

We first state our main results toward Malle's conjecture in
\autoref{subsection:notation-malle}, including a version without constants in
\autoref{theorem:turkelli} and a simplified version with constants in
\autoref{theorem:malle-g-connected}.
We introduce notation to reorganize the components of Hurwitz space according to
a given counting invariant in
\autoref{subsection:invariant-hurwitz-spaces}.
We then prove the simplified version of Malle's conjecture with constants in 
\autoref{subsection:reduced-discriminant-malle}
and the more comprehensive version without constants in 
\autoref{subsection:invariant-counting}.

\subsection{Notation and statements of versions of Malle's conjecture}
\label{subsection:notation-malle}

We next introduce some notation to state generalizations of Malle's conjecture, and T\"urkelli's
modification of Malle's conjecture.
For a nontrivial finite group $G \subset S_d$
and a global field $K$,
Malle's conjecture predicts the number of $G$ extensions 
of discriminant bounded by $X$,
(viewed as a
permutation group,) to be of the form
$C(K, G-\id) X^{\frac{1}{a(G-\id)}}(\log X)^{b_{M}(K,G, G-\id)-1}$,
for some unspecified $C(K,G-\id)$ and specified $a(G-\id)$ and $b_{M}(K,G, G-\id)$.
We will discuss a more general version of this conjecture depending on a
counting invariant, which
we define next.
In what follows we will restrict ourselves to global fields $K$ of the form $\mathbb F_q(t)$
for some prime power $q$.

\begin{definition}
	\label{definition:invariant}
	Fix a finite group $G$. Let $\inv: G - \id \to \mathbb Z_{>0}$ be a
	function which is constant on conjugacy classes (so $\inv(g) =
	\inv(h^{-1}gh)$ for any $g,h \in G$) and such that $\inv(g) = \inv(g^j)$
for any $j$ relatively prime to $\on{ord}(g)$.
We refer to any such function $\inv$ as a {\em counting invariant}.
If $c_i \subset G$ is a conjugacy class, we use $\inv(c_i) := \inv(g)$ for any
$g \in c_i$.

Given a field $\kappa$ and a $G$ cover $f: Y \to \mathbb A^1_\kappa$,
which we sometimes think of as
a $\kappa$ point of the Hurwitz space
$x \in \cquohur {G}{G}{n} {G-\id} {\mathbb Z[\frac 1 {|G|}]}(\kappa)$,
if $f$ has $n_i$ geometric points whose inertia lies in the conjugacy class $c_i \subset
G$,
we define the {\em invariant} of the cover associated to the counting invariant
$\inv$ by $\invc(x) := \sum_i n_i \inv(c_i)$.
We also use the notation $\invc(f)$ to mean the same thing as $\invc(x)$.
\end{definition}

\begin{example}
	\label{example:discriminant}
	Perhaps the most ubiquitous example of a counting invariant as above is the
	function $\Delta(g) := |G| -r(g)$, where $r(g)$ is the
	number of orbits of $g$ acting on $G$. Then, the 
	associated invariant function yields the degree of the discriminant of the cover.
	More generally, if $G \to S_d$ is a permutation group, one may consider
	the invariant $\inv(g) := d- r(g)$, where $(g)$ is the number of orbits
	of $g$ on $\{1, \ldots, d\}$. This counting invariant corresponds to the
	discriminant of a degree $d$ cover corresponding to $G$ acting via the
	above permutation representation.
\end{example}
\begin{example}
	\label{example:reduced-discriminant}
	Another example of a counting invariant is the counting invariant
	$\on{rDisc}(g) = 1$ for every $g \in G -\id$. This is the {\em reduced
	discriminant} and corresponds to counting by
	the radical of the discriminant of the cover, i.e., each branch point is
	counted according to its degree.
	Many people also refer to counting by the
	reduced discriminant as counting by ``the product of ramified primes.''
\end{example}

As mentioned, Malle's conjecture predicts values for certain constants $a(G-\id,\inv)$ and $b(K,
(G-\id)_{\inv})$ related to counting the number of $G$ extensions. We now define
these constants.

\begin{notation}
	\label{notation:malle}
	We fix a global field $K$ of the form $\mathbb F_q(t)$ with field of constants $\mathbb F_q$
and a nontrivial finite permutation group $G \subset S_d$
whose order is invertible on
$K$.
Also fix a counting invariant $\inv: G - \id \to \mathbb Z_{>0}$ as in
\autoref{definition:invariant}. 
For $c \subset G - \id$ a subset closed under conjugation, define
\begin{align*}
	a(c, \inv) := \min_{g \in c} \inv(g).
\end{align*}
For $c \subset G - \id$ a subset closed under conjugation, define $c_{\inv}
\subset c$ to be the subset of elements $g \in c$ with $\inv(g)
= a(c,\inv)$.

Suppose $q$ is a prime power with $\gcd(q,|G|) = 1$.
Let $N \subset G$ be a normal subgroup.
Choose a subset $c \subset N - \{\id\}$ closed under $G$-conjugation which generates $N$. 
Suppose $c$ is closed under conjugation by
elements of $G$. 
Also assume $c$ is closed under $q$th powering in the sense of
\autoref{definition:powering}; that is, if $g \in c$ so
is $g^q$.
We will next define a constant $b_{M}(K, N, c)$.
We let $c/N$ denote the set of $N$-conjugacy classes of $c$
and $c/G$ to denote the set of $G$-conjugacy classes of $c$.
Since $\gcd(q,|G|) = 1$, $q$th powering acts invertibly on $c$, so we have a
bijective operation given by $q^{-1}$ powering, which we define to be the inverse of $q$th
powering.
For $h \in G/N$ a generator, we let $\rho(K, N, c, h)$ be the quotient of $c/N$ by 
the equivalence relation given by $x \sim h x^{q^{-1}}h^{-1}$ for all $x
\in c/N$.
(This means that if we choose a representative $\tilde{x} \in N$ for $x$ and
$\tilde{h} \in G$ for $h$, then we take the $N$-conjugacy class of
$\widetilde{h} \widetilde{x}^{q^{-1}} \widetilde{h}^{-1}$; one can check
the resulting conjugacy class is independent of the choice of lifts.)  
We then define 
\begin{align*}
b_{M}(K, N, c) := \max_{h \in G/N, \langle h \rangle = G/N} |\rho(K,N, c,h)|.
\end{align*}

Finally, we define
\begin{align*}
	b_{T}(K, c) := \max_{K', N} b_{M}(K', N, c),
\end{align*}
where the maximum is taken over all $(K', N)$ where $a(c \cap N,\inv) = a(c,\inv)$, $N \subset G$
is a normal subgroup with $G/N$ cyclic, and $K' = \mathbb F_{q^{|G/N|}}
\otimes_{\mathbb F_q} K$.
\end{notation}

\begin{remark}
	\label{remark:}
	In the case $\inv = \Delta$ from \autoref{example:discriminant},
	if there are within a constant factor of $X^{\frac{1}{a(G-\id,\Delta)}}(\log
	X)^{b(K,G-\id)-1}$ many field extensions of $K$ of
	discriminant at most $X$ as $X \to \infty$,
	then one may interpret Malle's prediction in the function field setting
	to be that $b(K,G)$ is $b_{M}(K, G,
	(G-\id)_{\Delta})$.
	Hence, $b_M(K,N,(G-\id)_{\inv} )$ can be thought of as a generalization of Malle's
	predicted value of $b(K,G-\id)$. The value $b_{T}(K,(G-\id)_{\Delta})$ is the value which
	T\"urkelli predicted for $b(K,(G-\id)_{\Delta})$. 
	We note that T\"urkelli was in fact the first to make the predictions for 
	$b_{M}(K, N,(G-\id)_{\Delta})$ when $N \neq G$, but, to avoid confusion,
	we still opt to label these constants with the subscript $M$.
\end{remark}

\begin{remark}
	\label{remark:}
	In the case $\inv = \Delta$, (the discriminant invariant from
	\autoref{example:discriminant},) Malle originally used the notation
	$\frac{1}{a(G)}$ to denote what we are calling $a(G-\id, \inv)$. The
	reader may wish to keep in mind that there is some inconsistency
	for this notation throughout the literature.
	We prefer the definition we give so that $a(G-\id,\inv)$ is
	an integer. Additionally, essentially all other authors use $a(G)$ in place of $a(G-\id,
	\inv)$,
	but we prefer to use the notation $a(G-\id,\inv)$ to indicate that we are really dealing
with the rack which omits the identity element from $G$.
\end{remark}

We next record some notation for counting the number of $G$ extensions.

\begin{notation}
	\label{notation:}
	Let $G \subset S_d$ be a group. Let $K$ be a global function field of
	the form $\mathbb F_q(t)$ with field
	of constants $\mathbb F_q$ with a
	single infinite place and
	$O_K := \mathbb F_q[t]$.
Suppose $\gcd(q,|G|) = 1$.
We let $c \subset G$ be a union of conjugacy classes, and
$\inv(K, c, X)$ denote the groupoid cardinality of the groupoid of $G$ field extensions $L$ of $K$
corresponding to an finite extension $f: \spec \mathscr O_L \to \spec \mathscr O_K$, 
where $\mathscr O_L$ denotes
the normalization of $\mathscr O_K$ in $L$,
with inertia in $c$, and such that
$q^{\invc(f)} \leq X$.
Here, by groupoid cardinality, we mean that each extension is counted inversely
proportionally to its automorphisms.

We let $\inv(K, N, c, X)$ denote the groupoid cardinality of extensions $L$ over
$K$ as above, but with the
additional condition $\spec L \otimes_{\mathbb F_q} \overline{\mathbb F}_q$ has
$|G|/|N|$ connected components, any pair of which are isomorphic, and so that
each such component corresponds to an $N$ extensions of $K \otimes_{\mathbb F_q} \overline{\mathbb
F}_q(t)$.
\end{notation}

The next result establishes T\"urkelli's revision of Malle's conjecture as in
\cite[Conjecture 6.7]{turkelli:connected-components-of-hurwitz-schemes}
over $\mathbb F_q(t)$ when one takes $c = G - \id$.
As mentioned in the introduction, 
\cite[Theorem 1.3]{wang:counterexamples-for-turkelli}
shows 
\cite[Conjecture 6.7]{turkelli:connected-components-of-hurwitz-schemes}
is wrong over $\mathbb Q$, but nevertheless we are able to establish many cases
of it over function fields.
We prove a generalized version of T\"urkelli's conjecture where we only allow
ramification to only lie in a union of conjugacy classes $c$ closed under
$q$th powering.
We also count by a general invariant instead of restricting ourselves to the
discriminant.
\begin{theorem}
	\label{theorem:turkelli}
	Using notation from \autoref{notation:malle},
	fix a finite group $G$ and a union of conjugacy classes
	$c \subset G$ closed under $q$th powering. 
	Let $N \subset G$ denote a normal subgroup such that $G/N$ is cyclic.
	There is a constant $C$
	depending on $G$ so that for $q > C$ and $\gcd(q, |G|) = 1$,
	there are positive constants $C_-$ and $C_+$ depending on $q$ and $c$ so that for $X$ sufficiently
	large,
	\begin{align}
		\label{equation:sub-malle-count}
		&C_- X^{\frac{1}{a(c \cap N,\inv)}} (\log X)^{b_M(\mathbb F_{q^{|G|/|N|}}(t),
		N, (c \cap N)_{\inv})-1} \leq
			\inv(\mathbb F_q(t),N, c, X) \\
			&\leq C_+
			X^{\frac{1}{a(c \cap N,\inv)}} (\log
			X)^{b_M(\mathbb F_{q^{|G|/|N|}}(t), N, (c \cap N)_{\inv}) -1}.
	\end{align}
	In particular, 	
	\begin{align}
		\label{equation:all-malle-count}
		C_- X^{\frac{1}{a(c,\inv)}} (\log X)^{b_T(\mathbb F_q(t), c_{\inv})-1} \leq
		 \inv(\mathbb F_q(t),c , X) \leq C_+ X^{\frac{1}{a(c,\inv)}} (\log
		X)^{b_T(\mathbb F_q(t), c_{\inv}) -1}.
	\end{align}
\end{theorem}
We prove \autoref{theorem:turkelli} in \autoref{subsubsection:turkelli-proof}

\begin{remark}
	\label{remark:}
	Although
$\inv(\mathbb F_q(t),N, c, X)$ in
\autoref{theorem:turkelli} counts the number of extensions, weighted inversely
proportionally to their automorphisms, the same statement holds if one counts
all extensions with weight $1$, because any such extension has between $1$ and $|G|$
automorphisms.
\end{remark}

Although Malle's prediction for the $b$ constant is not correct in general, we
also prove it is correct if one restricts to geometrically connected covers, 
counts by $\on{rDisc}$, and takes $q$ sufficiently large.
Suppose $c = c_1 \cup \cdots \cup c_\upsilon$, with $c_i \subset G$ pairwise
distinct conjugacy classes. For the statement of the next result only, we use 
$\cquohur {G}{G}{n_1, \ldots, n_\upsilon} {c} {\mathbb F_q}$ to denote the union of
geometrically connected components of 
$\cquohur {G}{G} {n_1 + \cdots + n_\upsilon} {c} {\mathbb F_q}$
whose basechange to $\overline{\mathbb F}_q$ corresponds to the $G$-orbit of a
component of 
$\cphur {G} {n_1,\ldots, n_\upsilon} {c} {\overline{\mathbb F}_q}$.

\begin{theorem}
	\label{theorem:malle-g-connected}
	Fix residues $r_1, \ldots, r_\upsilon \bmod |G|^2$.
	Suppose $c = c_1 \cup \cdots \cup c_\upsilon$ is a disjoint union of $\upsilon$ conjugacy
	classes in $G$.
	There is a constant $C$, depending on $G$ and $c$,
	and a constant $C_{r_1, \ldots, r_\upsilon, c,G,q}$,
	depending on $r_1, \ldots, r_\upsilon, G$, and $q$,
	so that for $q > C$, $\gcd(q,
	|G|) = 1$, and $c$ closed under $q$th powering, 	
	\begin{align}
		\label{equation:malle-g-connected}
		\lim_{n \to \infty} \frac {\sum_{\substack{n_1, \ldots, n_\upsilon \\ n_1 + \cdots + n_\upsilon = n
		\\ n_j \equiv r_j \bmod |G|^2 \text{ for } 1 \leq j \leq \upsilon}}
	\left | [\cphur {G} {n_1, \ldots, n_\upsilon} {c} {\mathbb
F_q}/G](\mathbb F_q) \right | }{q^n n^{b_M(\mathbb F_q(t), G, c)-1}} = C_{r_1,
	\ldots, r_\upsilon, c, G,q}.
	\end{align}
	Moreover, there is some tuple of residue classes 
$r_1, \ldots, r_\upsilon \bmod |G|^2$
for which 
$C_{r_1, \ldots, r_\upsilon, c, G,q}\neq 0$.
\end{theorem}
We prove this in \autoref{subsubsection:connected-malle-proof}.

\begin{remark}
	\label{remark:}
	We note that \autoref{theorem:malle-g-connected} can be viewed as a
	periodic version of Malle's conjecture where one counts by the reduced
	discriminant from \autoref{example:reduced-discriminant}, and restricts
	to counting geometrically connected covers.
	It seems likely that a similar statement to
	\autoref{theorem:malle-g-connected} could be proven for all $G$
	extensions
	rather than the geometrically connected ones.
	We believe it would be
	interesting to verify this.
	See also \autoref{question:extend-constant}.
\end{remark}

\subsection{Defining Hurwitz spaces for counting by invariant}
\label{subsection:invariant-hurwitz-spaces}

We next want to define Hurwitz spaces which count points by a given invariant.
In order to define these, we first show that these invariants of a cover is constant
along irreducible components. 
\begin{lemma}
	\label{lemma:invariants-on-components}
	Suppose $c \subset G$ is a conjugacy invariant subset.	Suppose $B$ is a local henselian scheme with residue field $\kappa$
	containing $\mathbb F_q$ that has
	characteristic prime to $|G|$. If $\kappa$ is
	algebraically closed, for any irreducible component $Z \subset \phur {G} {n} {c} {B}$
	and $x, x'$ points of $Z$,
	we have $\invc(x) = \invc(x')$.
	If $c$ is closed under $q$th powering, this also holds when $\kappa = \mathbb F_q$.

	With the same hypotheses as above, if $c \subset N \subset N' \subset G$ with $N$ and $N'$
	normal subgroups of $G$, for $Z \subset \quohur {N}{N'}{n} {c} {B}$ any irreducible
	component and $x, x' \in Z$, we have $\invc(x) = \invc(x')$.
\end{lemma}
\begin{proof}
	We will first prove the statement about components of
	$\phur {G}{n} {c} {B}$.
	First, we handle the case $\kappa$ is algebraically closed.
	To verify this case, it suffices so show there
	is a well defined map
	$Z \to \conf_{n_1, \ldots, n_\upsilon,B}$ for any component $Z
	\subset\phur {G}{n}
	{c} {B}$
	with some point of $Z$ parameterizing covers branched at $n_i$ points
	with inertia in $c_i$. Let $n := \sum_i
	n_i$.
	We certainly have a well defined map $Z \to \conf_{n,B}$ and we wish to show
	this factors through 
	$\conf_{n_1, \ldots, n_\upsilon,B}$.
	Using \autoref{lemma:component-bijection},
	we can deduce the statement for arbitrary henselian $B$ from the case
	that $B = \spec \kappa$ using that 
	the
	specialization map on geometric components of Hurwitz space is compatible with the
	specialization map on geometric components of configuration space.
	This factorization is induced over $B= \spec \mathbb C$ via the map described in
	\autoref{definition:rack-pointed-hurc}.
	For $B = \spec \kappa$  characteristic $0$, we then obtain the claim via base
	change to a common algebraically closed field containing $\kappa$ and
	$\mathbb C$.
	To conclude the proof, it suffices to treat algebraically closed fields
	$\kappa$ of positive characteristic prime to $|G|$.
	In turn, via base change, one can reduce to verifying this for fields of
	the form
	$\kappa = \overline{\mathbb F}_q$. 
	In this case, one can take $B$ to be a
	henselian dvr with residue field 
	$\overline{\mathbb F}_q$ and generic characteristic $0$, in which case
	the specialization map induces a bijection on 
	components by
	\autoref{lemma:component-bijection}, and again this specialization is
	compatible with the corresponding specialization map on configuration
	space.

	Next, we verify the statement for $Z \subset \phur {G}{n} {c} {B}$
when $B = \spec \mathbb F_q$.
	Next, note that the Frobenius action preserves the value of $\inv(x)$ by
	\cite[Theorem 12.1(2)]{liuWZB:a-predicted-distribution}, which describes
	how Frobenius acts on components of Hurwitz spaces,
	and \cite[Remark 5.3]{wood:an-algebraic-lifting-invariant}, which
	explains why this action preserves the invariant of the cover using the
	assumption that $\inv(g) = \inv(g^q)$ for any $q$ with $q$
	prime to $|G|$.
	It follows that 
	$\invc(x) = \invc(x')$ for any $x, x' \in \phur {G}{n} {c} {\mathbb F_q}$.

	We next verify the statement for 
	$Z \subset \phur {G}{n} {c} {B}$
for an arbitrary henselian scheme $B$ with
	residue field $\mathbb F_q$.
	Indeed, let $B'$ denote the strict henselization of $B$. In order to
	verify the statement for some irreducible component $Z \subset \phur {G}{n}
	{c} {B}$,
	it suffices to show that for any $x_1, x_2 \in Z \times_B B'$ we have
	$\inv(x_1) = \inv(x_2)$. 
	Choose the two irreducible components $Z_1, Z_2 \subset Z
	\times_B B'$, such that $x_1 \in Z_1, x_2 \in Z_2$. Take $x'_1 \in
	(Z_1)_{\overline{\mathbb F}_q}$ and
	$x'_2 \in (Z_1)_{\overline{\mathbb F}_q}$. By the settled case when $B$
	is strictly henselian we have $\invc(x_1) = \invc(x'_1)$ and $\invc(x_2)
	= \invc(x'_2)$ and by the settled
	case when $B = \spec \mathbb F_q$, we have $\invc(x'_1) = \invc(x'_2)$.
	Hence, $\inv(x_1) = \inv(x_2)$.
	
The final part of the statement about components of 
$\quohur {N}{N'}{n} {c} {B}$
follows similarly to the above, using that the invariant
is constant on conjugacy classes, and hence conjugating a point of Hurwitz space
by the $G$ action
(i.e., changing the marked point of a given cover over $\infty$)
will preserve the value of the invariant.
\end{proof}

We next define the relevant Hurwitz spaces we will use
where we reorder the components according to a given counting invariant.
To prove \autoref{theorem:malle-g-connected}, we will only concern ourselves
with geometrically connected covers, so we will only need the case $N = G$
in the definition below to prove \autoref{theorem:malle-g-connected}. However, it will be convenient to make the following
notation for more general $N$ in order to prove \autoref{theorem:turkelli}.

\begin{notation}
	\label{notation:invariant-hur}
	Fix a finite group $G$, normal subgroups $N \subset N' \subset G$, and $c \subset N$ a conjugacy invariant subset.
	Let $B$ be a Henselian local scheme on which $|G|$ is invertible.
	Assume the residue field of $B$ is either $\mathbb F_q$ and $c$ is
	closed under $q$th powering, or just assume the residue field is algebraically closed.
We define 
$\cquohur N {N'} {\invc \leq n} c B$ to be the union of components of $\cup_{n'
\geq 0}\cquohur N {N'} {n'} c B$ parameterizing covers whose invariant is at
most $n$.
This is well defined by
\autoref{lemma:invariants-on-components}.
\end{notation}

\begin{notation}
	\label{notation:connected-quotient-hur}
	Continuing with notation as in \autoref{notation:invariant-hur},
define
$\cgquohur N G {n} c B \subset \cquohur N G {n} c B$
to be the union of components which are geometrically irreducible
and whose preimage in 
$\cquohur N {N'} {n} c B$
does not consist of $|G|/|N'|$ geometrically irreducible components
for any $N' \subsetneq G$ containing $N$.

We also define 
$\cgquohur N {G} {\invc \leq n} c B$
to denote the union of components of
$\cup_{n' \geq 0} \cgquohur N {G} {n'} c B$
parameterizing covers
whose invariants are at most $n$.
\end{notation}

\begin{notation}
	\label{notation:connected-hur-by-invariants}
	Continuing to use notation as in
	\autoref{notation:connected-quotient-hur},
when  $B$ is an algebraically closed field, 
we use 
$\cgquohur N G {n_1,\ldots, n_\upsilon} c B$
to denote the union of components of 
$\cgquohur N G {n_1+\ldots+ n_\upsilon} c B$
which correspond to $G$-conjugation orbits of components of 
$\quohur N G {n_1+\ldots+ n_\upsilon} c B$
which additionally parameterize covers with $n_i$ branch points with inertia in
$c_i$. (Here,we still impose the condition that these components are
geometrically irreducible and their preimage in 
$\cquohur N {N'} {n} c B$
does not consist of $|G|/|N'|$ geometrically irreducible components
for any $N' \subsetneq G$ containing $N$.)
\end{notation}

Before continuing, we prove a quick lemma explaining that 
$\quohur N G {\invc \leq n} {c\cap N} {\mathbb F_q}$ has components
parameterizing points that could potentially correspond to connected $G$ covers.

\begin{lemma}
	\label{lemma:connected-cover-observation}
	Fix a prime power $q$.
	Let $G$ be a group, $N \subset G$ be a normal subgroup, and 
$c \subset N-\id$ 
	a union of
	conjugacy classes generating $N$ and closed under $q$th powering.
	Any 
	$x \in \quohur N G {\invc \leq n} c {\mathbb F_q}(\mathbb F_q)$
corresponding to a connected $G$ cover which geometrically becomes a connected
$N$ cover must lie in
$\cgquohur N G {\invc \leq n} c {\mathbb F_q}(\mathbb F_q)$.
\end{lemma}
\begin{proof}
	We first note $x$ must lie in a geometrically irreducible component
$Z \subset \quohur N G {\invc \leq n} c {\mathbb F_q}$, as any component $Z$ which is not geometrically irreducible has no $\mathbb
F_q$ points.
Second, if there is some $N' \subsetneq G$ so that 
$\quohur N {N'} {\invc\leq n} c {\mathbb F_q}(\mathbb F_q)$
has $|G|/|N'|$ geometrically irreducible components over such a component $Z$,
then those components are permuted by the $G/N'$ action, and each map
isomorphically to $Z$ in the quotient.
Therefore, all $\mathbb F_q$ points of $Z$ lift to some $\mathbb F_q$ point of a
component of 
$\quohur N {N'} {\invc\leq n} c {\mathbb F_q}$,
which implies that $Z$ has no $\mathbb F_q$ points corresponding to a connected
$G$ cover, as each such cover has covering group contained in $N'\subsetneq G$.
\end{proof}
\begin{remark}
	\label{remark:}
	\autoref{lemma:connected-cover-observation}
	is meant to partially explain the superscript $G$ on the $C$  in the notation
$\cgquohur N G {\invc \leq n} {c\cap N} {\mathbb F_q}$.
Namely, it expresses
	that the covers should not only be geometrically connected $N$ covers,
	but the components should even contain many connected $G$ covers.
In fact, we will later show that all components of 
$\cgquohur N G {n} {c\cap N} {\mathbb F_q}$,
for $n$ sufficiently
large, contain connected covers, which follows from
\autoref{lemma:n-extensions-constant}.
\end{remark}

\subsection{Counting extensions by reduced discriminant}
\label{subsection:reduced-discriminant-malle}

Our next goal is to prove \autoref{theorem:malle-g-connected}.
This will be substantially easier than \autoref{theorem:turkelli}
due to two simplifications. First, we are only counting geometrically connected
covers, so we do not need to worry about subtleties related to connected covers
which are disconnected geometrically. Second, we count by the product of
ramified primes, for which it is simpler to organize the information from the perspective of our
homological stability results. In the case of counting by a more general
invariant, we will need to use a Tauberian theorem from analytic number theory.

Note that
$\cgquohur N G {\invc\leq n} c {\mathbb F_q}$
is defined over $\mathbb F_q$ and so there is descent data for
$\cgquohur N G {\invc \leq n} c {\overline{\mathbb F}_q}$
along $\overline{\mathbb F}_q$ over $\mathbb F_q$.
Hence it makes sense to ask whether one of its components, or, in turn, a
component of 
$\cgquohur N G {n_1,\ldots, n_\upsilon} c {\overline{\mathbb F}_q}$ is fixed by
this descent data.
Being fixed corresponds to being geometrically irreducible.
The next lemma shows this is often the case, and also that irreducible
components occur with a certain periodicity, in a suitable sense.

\begin{lemma}
	\label{lemma:geom-irred-mod-g}
	We use notation from \autoref{notation:connected-hur-by-invariants}.
	Fix a prime power $q$.
	Let $G$ be a finite group, $N$ a normal subgroup, and $c \subset N -
	\id$ a union of conjugacy classes closed under $q$th powering.
\begin{enumerate}
	\item 	Suppose $h$ generates $G/N$,
	we have an equality of formal sums
	$\sum_{i=1}^j n_i c_i= \sum_{i=1}^j n_i h c_i^{q^{-1}}h^{-1}$, and
	$n_1, \ldots, n_j \equiv 0 \bmod |G|$. Then there is some component of
$\cgquohur N G {n_1,\ldots, n_j, 0, \ldots, 0} c {\overline{\mathbb F}_q}$
which is fixed by descent data for $\overline{\mathbb F}_q$ over $\mathbb F_q$,
	\item Moreover, there is a constant $J$ depending on $c$ so that
whether a component of
$\cgquohur N G {n_1,\ldots, n_j, 0, \ldots, 0} c {\overline{\mathbb F}_q}$
is fixed by the above descent data only depends on 
the residue classes of $n_1, \ldots, n_j$ modulo $|G|$, so long as
$n_1, \ldots, n_j > J$.
	\item In particular, those components fixed by such descent data correspond to
geometrically irreducible components of 
$\cgquohur N G {n_1 + \cdots + n_j} c {{\mathbb F}_q}$.
\end{enumerate}

\end{lemma}
\begin{proof}
	Statement (3) is a general fact about descent.

	We next verify (1).
	We now follow notation from 
\cite{wood:an-algebraic-lifting-invariant}.
In particular, we will use $\underline{m}$ to denote a tuple $(n_1, \ldots,
n_j) \in \mathbb Z^j$ and we use $S_c\to G$ to denote a particular finite group
called a reduced Schur cover for $G$. We will use that $\ker(S_c \to G)$ has
exponent dividing $|G|$.
We use $q^{-1} *$ to denote the discrete action of $q^{-1}$ on $S_c
\times_{G^{\on{ab}}} \mathbb Z^{G/c}$.
The action of such descent data on the components of $\cquohur N {N'} {n_1,
\ldots, n_j, 0, \ldots, 0} c {\overline{\mathbb F}_q}$
can be deduced from the explicit description of the discrete action in
\cite[p. 8, (4)]{wood:an-algebraic-lifting-invariant}.
(Technically, it is assumed there that $c$ is closed under invertible powering,
but the same proof works if we only assume $c$ is closed under $q$th powering.)
Namely, if we take $h = \id \in S_c$ and $\underline{m} = (n_1,
\ldots, n_j)$ with all $n_1, \ldots, n_j \equiv 0 \bmod |G|$, 
we find $q^{-1} * (\id, \underline{m}) = (\id, \underline{m}^{q^{-1}})$,
using that each $n_i$ divides the order of $|G|$
and the components of 
$\cquohur N {N'} {n_1,
\ldots, n_j, 0, \ldots, 0} c {\overline{\mathbb F}_q}$
correspond to $N'$ orbits of such data.
In our setting, where 
$\sum_{i=1}^j n_i c_i= \sum_{i=1}^j n_i h c_i^{q^{-1}}h^{-1}$,
we claim there cannot be $|G|/|N'|$ orbits for any $N' \subsetneq G$,
so that 
the set of $h \in G/N$ for which 
$\sum_{i=1}^j n_i c_i= \sum_{i=1}^j n_i h c_i^{q^{-1}}h^{-1}$
consists precisely of those $h \in N'/N$. 
Indeed, the definition
$\cgquohur N G {n_1,\ldots, n_j, 0, \ldots, 0} c {\overline{\mathbb F}_q}$
implies some generator
$h \in G/N$ satisfies the above, so $h \notin N'/N$.
This proves the first statement.

Finally, we verify (2). This follows from the explicit description of the
action of the descent data on the components given in
\cite[p. 8, (4)]{wood:an-algebraic-lifting-invariant},
whose first component only depends on the values of $n_1, \ldots, n_\upsilon \bmod
|G|$ using that the exponent of $S_c$ divides $|G|$, see also \cite[Remark
4.1]{wood:an-algebraic-lifting-invariant}.
\end{proof}

The proof of 
\autoref{theorem:malle-g-connected} now follows fairly straightforwardly from the above
lemma and our results on Frobenius equivariant stabilization of cohomology of
Hurwitz spaces.

\subsubsection{Proof of \autoref{theorem:malle-g-connected}}
\label{subsubsection:connected-malle-proof}

	We first show
	\eqref{equation:malle-g-connected} holds.
	This will follow from \autoref{lemma:component-point-bound} once we verify
	the number of geometrically irreducible components of 
	$[\cphur {G} {n} {c} {\mathbb F_q}/G]$ with all $n_1, \ldots, n_\upsilon > J$
	such that $n_j \equiv r_j \bmod |G|^2$
	is a polynomial in $n$ of degree $b_M(\mathbb F_q(t), G, c)-1$.

	Using \autoref{lemma:geom-irred-mod-g}(2) and (3), in order to verify
	\eqref{equation:malle-g-connected},
	it suffices to prove the number 
	of tuples $\sum_{i=1}^j n_i c_i$ with
	$\sum_{i=1}^j n_i c_i =\sum_{i=1}^j n_i c_i^{q^{-1}}$
	and $n_j \equiv r_j \bmod |G|^2$
	is a polynomial in $n$ of degree 
	$b_M(\mathbb F_q(t), G, c)-1$.
	Note that here, since $N = G$, we have $hc_i h^{-1} = c_i$ for any $h
	\in G$.
	We next observe that any such tuple must be of the form 
	$\sum_i m_i \mathscr O_i$, for some $m_i \in \mathbb N$,
	where the $\mathscr O_i$ denote the orbits of $c$ under the $q^{-1}$
	powering action; indeed, the conjugacy classes in each of the orbits
	$\mathscr O_i$ are cyclically permuted by $q^{-1}$ powering and so the condition that the
	coefficient of $c_i$ agrees with the coefficient of $c_i^{q^{-1}}$
	implies all coefficient of conjugacy classes in this orbit must appear
	in such a tuple.
	Hence, the number of such 
	$(n_1, \ldots, n_\upsilon)$ is a polynomial whose degree is one less than the number of
	such orbits $\mathscr O_i$.
	(If we counted all $n_1, \ldots, n_\upsilon$ with $n_1 + \cdots +
		n_\upsilon \leq n$ and $n_j \equiv r_j \bmod |G|^2$, we would get a polynomial of degree equal to the
		number of orbits, but since we only want
		$n_1 + \cdots + n_\upsilon = n$ we get a polynomial of one degree
	less.)
		By definition, the number of such orbits is precisely 
		$b_M(\mathbb F_q(t), G, c)$.
		This proves \eqref{equation:malle-g-connected}.

		The final statement of
		\autoref{theorem:malle-g-connected}
		that $C_{r_1, \ldots, r_\upsilon, c, G,q}
		\neq 0$ follows from
		\autoref{lemma:geom-irred-mod-g}(1).
		\qed

\subsection{Counting extensions by invariant}
\label{subsection:invariant-counting}

We next embark on some preparations to prove \autoref{theorem:turkelli}.
We first aim to produce an upper bound. In order to do so, we will need a
Tauberian theorem, and some notation to state the Tauberian theorem.

\begin{notation}
	\label{notation:twisted-conf}
	Fix a normal subgroup $N \subset G$ and let $c \subset N - \id$
be a subset closed under conjugation by $G$ and closed under $q$th powering.
Write $c = c_1 \cup \cdots \cup c_j$ with $c_i \subset N$ pairwise distinct
$N$-conjugacy classes.
Consider the free $\mathbb Z$ module $\mathbb Z^{c/N}$ where $c/N$ denotes the
quotient of $c$ by the $N$ conjugation action.
We define actions of $\widehat{\mathbb Z}$
and $G$ on $\mathbb Z^{c/N}$.
First, $G$ acts on $c$ by conjugation, and this induces an action on $c/N$,
hence on $\mathbb Z^{c/N}$; explicitly $g \in G$ sends the basis vector indexed by $xN \in c/N$ to
the basis vector indexed by $gxg^{-1}N \in G/N$.
The topological generator $1 \in \widehat{\mathbb Z}$ acts on $c$ by sending
sending $g \mapsto g^q$.
Again, this induces an action on 
$\mathbb Z^{c/N}$.
\end{notation}

\begin{lemma}
	\label{lemma:tuple-counting}
	We use notation as in
	\autoref{notation:malle} and
	\autoref{notation:twisted-conf}.
	Let $r, n_1, \ldots, n_j, \delta \in \mathbb N$.
	Fix $h \in G$.
	Define $b_{r,\delta}$ as the number of tuples
	tuples of $N$-conjugacy classes of formal sums $\sum_{i=1}^j n_i c_i$ 
	so that
	$r=n_1 + \cdots + n_j$,
	$\inv(c_1)n_1 + \cdots + \inv(c_j) n_j = \delta$,
	and
	the action of $-1 \in \widehat{\mathbb Z}$
	on $\mathbb Z^{c/N}$,
	which is given by $q^{-1}$ powering,
	sends
	$\sum_{i=1}^j n_i c_i$
	to the tuple of $N$-conjugacy classes
	$\sum_{i=1}^j n_i c_i^{q^{-1}}$, which we assume agrees with the tuple $\sum_{i=1}^j n_i h^{-1} c_i h$.
Define	
	$a_\delta := \sum b_{r,\delta} q^\delta$.
	Then, the function
	$n\mapsto \sum_{\delta \leq n} a_\delta$ is bounded above and below by a constant
	multiple of 
	$q^n n^{|\rho(K, N,c_{\inv},h)|-1} + O_G(n^{|\rho(K,
	N,c_{\inv},h)|-2})$.
	\end{lemma}
	Our proof closely follows \cite[Theorem
	5.7]{turkelli:connected-components-of-hurwitz-schemes}.
\begin{proof}
	Consider the Dirichlet series
	$R(s) := \sum_{\delta \geq 1} a_\delta q^{- \delta s}$.
	We let $t = q^s$ and define the function $f(t) := a_\delta t^s$ so that
	$R(s) = f(t)$.
	By \cite[Lemma 5.8]{turkelli:connected-components-of-hurwitz-schemes},
	(or alternatively see \cite[Theorem 17.4]{rosen:number-theory-in-function-fields},)
	it suffices to show the Dirichlet series $f(t)$ has a pole of order 
	$|\rho(K, N,c_{\inv},h)|$ at $t = q^{\frac{1}{a(c,\inv)}}$ and no poles with $|t| >
	q^{\frac{1}{a(c,\inv)}}$.

	To compute the number of tuples $\sum_{i=1}^\upsilon n_i c_i$ as in the
	statement,
	any such tuple must be of the form 
	$\sum_{i=1}^v m_i \mathscr O_i$, for some $m_i \in \mathbb N$
	where $\mathscr O_1, \ldots, \mathscr O_v$ denote the orbits of
	$N$-conjugacy classes in $c$ under the equivalence relation
generated by
$x \sim hx^{q^{-1}} h^{-1}$.
Moreover, any such tuple gives an $N$-conjugacy class satisfying this
constraint.
We use $|\mathscr O_i|$ to denote the number of $N$-conjugacy classes in $\mathscr
O_i$.
We use $\inv(\mathscr O_i)$ to denote the index of any $N$-conjugacy class in $c$ in
the orbit $\mathscr O_i$; we note that all of these $N$-conjugacy classes have the
same invariant since they all lie in the same $G$ conjugacy class.
Then, observe that the dimension of the component associated to 
$\sum_i m_i |\mathscr O_i|\mathscr O_i$
is $\sum_i m_i$ and the invariant of that component is $\sum_i m_i |\mathscr
O_i|\inv(\mathscr O_i)$.
Therefore,
\begin{align*}
	\sum_{\delta =1}^\infty a_\delta q^{- \delta s} =
	\sum_{(m_1, \ldots, m_v)} q^{ \sum_{i=1}^v m_i |\mathscr O_i| } q^{- \sum_{i=1}^v
	m_i |\mathscr O_i|\inv(\mathscr O_i)  }
	= \prod_{i=1}^v \frac{1}{1 - q^{|\mathscr O_i| (1 - \inv(\mathscr O_i)
	s)}}.
\end{align*}
This function has a pole of order $|\rho(K, N,c_{\inv},h)|$ 
at
$t = q^{\frac{1}{a(c,\inv)}}$ 
(as is immediate from the definition of $\rho(K, N,c_{\inv},h)$ given in 
\autoref{notation:malle})
and no smaller poles, which proves the result.
\end{proof}

We are now prepared to prove our upper bound on the number of extensions.

\begin{lemma}
	\label{lemma:component-count-malle-upper}
	For $c \subset N-\id$ and $N \subset G$ normal as in \autoref{notation:malle},
	let $K = \mathbb F_q(t)$.
	Suppose $c = c_1 \cup  \cdots \cup c_\upsilon$, with $c_i$ pairwise
	distinct $N$-conjugacy classes.

There is a constant $C> 0$ depending on $c$ so that
if $q> C$ is a prime power with $\gcd(q, |G|) = 1$,
there is a
positive constant $D_+(K,N,c)$
so that 
\begin{align*}
	|\cgquohur N G {\invc \leq n} c {\mathbb F_q}(\mathbb F_q)| \leq 
	q^{\frac{n}{a(c,\inv)}} \cdot (D_+(K,N,c) n^{b_{M}(K,N,c_{\inv})-1} +
O_G(n^{b_{M}(K,N,c_{\inv})-2})).
\end{align*}
\end{lemma}
\begin{proof}
	Using \autoref{lemma:component-point-bound},
	there is some fixed constant $R$ only depending on $G, c$, and $q$, so that any
	geometrically irreducible component
	$Z \subset \cgquohur N G {\invc \leq n} c {\mathbb F_q}(\mathbb F_q)$
	has $Z(\mathbb F_q) \leq R q^{\dim Z}$.
	Hence, for the purposes of proving this result, we may assume that $R =
	1$. That is, we may assume $Z(\mathbb F_q) \leq q^{\dim Z}$.

	First, by \autoref{lemma:component-bijection}, there is a bijection between
	components of Hurwitz spaces over $\overline{\mathbb F}_q$ (when
	$\gcd(q, |G|) = 1$) and $\mathbb
	C$.
Then, combining this with \autoref{theorem:some-large-homology-stabilizes}
	yields some constant $J$ so that for any $n_1, \ldots, n_j$, there are
	some $n_1', \ldots, n_j' \leq J$ and a bijection between the components of
	$\cgquohur N G {n_1, \ldots, n_\upsilon} c {\overline{\mathbb F}_q}$
	and the components of
	$\cgquohur N G {n'_1, \ldots, n'_\upsilon} c {\overline{\mathbb F}_q}$.

	We next claim that, after possibly
	increasing $J$, the Frobenius action on
	$\cgquohur N G {\invc \leq n} c {\overline{\mathbb F}_q}$
fixes a component of 
	$\cgquohur N G {n_1, \ldots, n_\upsilon} c {\overline{\mathbb F}_q}$
	if and only if it fixes the corresponding component of
	$\cgquohur N G {n'_1, \ldots, n'_\upsilon} c {\overline{\mathbb F}_q}$
	for some $n'_1, \ldots, n'_\upsilon \leq J$.
	Indeed, in this range
	we already know from 
	\autoref{theorem:some-large-homology-stabilizes}
	and \autoref{remark:transfer}
	that the map
$U^{g,q,M,G}_{\overline{\mathbb F}_q}$ is an
isomorphism, and it is moreover Frobenius equivariant by 
\autoref{theorem:frob-equivariant-stabilization}. Applying this for $i = 0$
implies that the 
	Frobenius on the components of these spaces must stabilize (periodically
	in the values of $n_i$) as well.

	Therefore, 
	taking the maximum over all $n'_1, \ldots, n'_\upsilon \leq J$, we obtain a
	upper bound, uniform in $n_1, \ldots, n_\upsilon$, for the number of
	components $\overline{Z}$ of
	$\cgquohur N G {n_1, \ldots, n_\upsilon} c {\overline{\mathbb F}_q}$ for
	which there exists a component $Z$ of
	$\cup_n \cgquohur N G {\invc \leq n} c {\mathbb F_q}$
	with $\overline{Z} = Z_{\overline{\mathbb F}_q}$.

	There is a finite \'etale cover
$\cgquohur N G {n_1, \ldots,
	n_\upsilon} c {\overline{\mathbb F}_q}$ over a configuration space
$[\on{Conf}_{n_1, \ldots, n_\upsilon,\overline{\mathbb F}_q}/G]$ where
the covers being parameterized have $n_i$ branch points in conjugacy class
$c_i$.
Moreover, for each such component which is the base change of a component of
$\cgquohur N G {n_1+ \cdots + n_\upsilon} c {\overline{\mathbb F}_q}$,
there is descent data along the extension $\overline{\mathbb F}_q$ over $\mathbb
F_q$ corresponding to the Frobenius
action.
Hence, the number of components of dimension $r$ we are trying to count
	is bounded above by a constant factor times the number of 
	fixed components under the descent data for Frobenius of
$[\on{Conf}_{n_1, \ldots, n_\upsilon,\overline{\mathbb F}_q}/G]$
with $n_1 + \cdots + n_\upsilon = r$,
and $\inv(c_1)n_1 + \cdots + \inv(c_\upsilon) n_\upsilon \leq n$.
Such components can exactly be identified with elements of 
$\mathbb Z^{c/N}$ where $c/N$ denotes the quotient of $c$ by the
$N$-conjugation action, with such tuples considered up to $G$-conjugation,
which are fixed under 
the descent data for Frobenius acting on each $c_i$ by the $q^{-1}$ 
powering action.
The claim on the total number of $\mathbb F_q$ points then follows from
\autoref{lemma:tuple-counting}.
\end{proof}

The next step in our proof will be to produce a lower bound on the number of
geometrically irreducible components.

\begin{lemma}
	\label{lemma:component-count-malle-lower}
	For $c \subset N-\id$ and $N \subset G$ normal as in \autoref{notation:malle},
	let $K = \mathbb F_q(t)$.
	Suppose $c = c_1 \cup \cdots \cup c_\upsilon$, with $c_i$ pairwise
	distinct $N$-conjugacy classes.
	There is a constant $C$ depending on $G$ so that for $q >C$ and $\gcd(q, |G|) = 1$,
there is a positive constant $D_-(K,N,c)$
so that 
\begin{align*}
	|\cgquohur N G {\invc \leq n} c {\mathbb F_q}(\mathbb F_q)|
\geq q^{\frac{n}{a(c,\inv)}} \cdot (D_-(K,N,c) n^{b_{M}(K,N,c_{\inv})-1} +
O_G(n^{b_{M}(K,N,c_{\inv})-2})).
\end{align*}
\end{lemma}
\begin{proof}
For $G$ a group and $c \subset G$ a union of conjugacy classes, following
\cite[\S2]{wood:an-algebraic-lifting-invariant},
we define $U(G,c)$ to be the group with generators $[g]$ for $g \in c$ and
relations $[x][y][x]^{-1} = [xyx^{-1}]$ for $x, y \in c$.

	If all the $n_\upsilon$ are sufficiently large, 
	the components of $\cphur {N} n {c} {\mathbb F_q}$ can be described as 
	elements of $U(N,c)$
	\cite[Theorem 3.1]{wood:an-algebraic-lifting-invariant}.
	Now, $G$ acts by conjugation on 
	$U(N,c)$, with the action given by $g \cdot [n] = [gng^{-1}]$ for
	$g \in G$ and $n \in N$.
	Suppose $c_1,\ldots, c_j$ are the $N$ conjugacy classes
	of minimal index in $c$ so that $c_{\inv} := c_1 \cup \cdots
	\cup c_j$. We will produce
	$D_-(K,N,c) n^{b_{M}(K,N,c_{\inv})-1} +
	O_G(n^{b_{M}(K,N,c_{\inv})-2})$ geometrically irreducible components of
	$\cgquohur N G {\invc \leq n} {c{\inv}} {\mathbb F_q}$,
	whose dimension is between $\frac{n}{a(c,\inv)}$ and
	$\frac{n}{a(c,\inv)} - j|G|$,
	and so the result will follow from
	\autoref{lemma:component-point-bound}.

If we restrict to $n_1, \ldots, n_j \equiv 0 \bmod |G|$, 
and whose sum is as close to $\frac{n}{a(c,\inv)}$ as possible
(meaning it will be at least $\frac{n}{a(c,\inv)} - j|G|$)
it follows from \autoref{lemma:geom-irred-mod-g},
that it suffices to verify there are
at least
$D_-(K,N,c) n^{b_{M}(K,N,c_{\inv})-1}$
many tuples 
$\sum_{i=1}^j n_i c_i$
which agree with
$\sum_{i=1}^j n_i h c_i^{q^{-1}}h^{-1}$
for some generator $h \in G/N$.
Indeed, this holds because any such tuple must be of the form 
$\sum_i m_i \mathscr O_i$, for some $m_i \in \mathbb N$,
where $\mathscr O_i$ denote the orbits of $c'$ under the equivalence relation
generated by
$x \sim hx^{q^{-1}} h^{-1}$.
By definition, there are $|\rho(K, N,c_{\inv},h)|$ such orbits.
Hence, there is a polynomial in $n$ of degree 
$|\rho(K, N,c_{\inv},h)|$
such orbits if we restrict to the region where $n_1 + \ldots+ n_j \leq n$.
This will then become a polynomial of degree 
$|\rho(K, N,c_{\inv},h)|-1$ if we also assume $n_1 + \cdots + n_j \geq n - j
|G|$, completing the proof.
\end{proof}
\begin{remark}
	\label{remark:}
	The geometric meaning of the end of the proof of
	\autoref{lemma:component-count-malle-lower} is as follows:
$\cgquohur N G {n_1, \ldots, n_\upsilon, 0, \ldots, 0} c {\overline{\mathbb F}_q}$
covers a twisted configuration space
$\conf_{n_1, \ldots, n_j,0, \ldots, 0,\overline{\mathbb F}_q}$ which has a
compatible Galois action induced by sending the component
of the configuration space indexed by
$\sum_{i=1}^j n_i c_i$  to 
$\sum_{i=1}^j n_i c_i^{q^{-1}}$.
Above, we are computing the number of fixed components under the descent data
for Frobenius of this twisted
configuration space as we vary over $n_1, \ldots, n_j$.
\end{remark}

Having controlled the number of geometrically irreducible components in
$\cgquohur N G {\invc \leq n} c {\mathbb F_q}$
from below,
we next wish to relate that number to the number of connected $G$ extensions.
To do this, we will investigate when connected $G$ extensions become
geometrically disconnected.

\begin{lemma}
	\label{lemma:n-extensions-constant}
	There is a constant $C$ so that for $q > C$ with $\gcd(q,|G|) =1$, the
	following statement holds.
	Let $c \subset N -\id \subset G$ be a subset closed under conjugation by
	$G$ and closed under $q$th powering.
	For $n$ sufficiently large and 
	any geometrically irreducible component $Z \subset \cgquohur N G {\invc \leq n} c {\mathbb F_q}$
	parameterizing points of invariant exactly $n$, the number of points of
	$Z(\mathbb F_q)$ corresponding to connected $G$
	covers is at least $\frac{|N|}{2|G|} \cdot |Z(\mathbb F_q)|$.
\end{lemma}
\begin{proof}
	We now assume $N \subsetneq G$ as otherwise the result is trivial.
First, we observe that 
the $\mathbb F_q$ points of 
$\cgquohur N G {\invc \leq n} c {\mathbb F_q}$ correspond bijectively to isomorphism classes
of not-necessarily-connected $G$ extensions
such that their base change to $\overline{\mathbb F}_q$ is a disjoint union of
$|G|/|N|$ connected $N$
extension over $\overline{\mathbb F}_q$.

Therefore,
every connected $G$ extension which becomes a connected $N$ extension geometrically
corresponds to a point of
$\cgquohur N G {\invc \leq n} c {\mathbb F_q}$.
We now fix a particular geometrically connected component $Z \subset 
\cgquohur N G {\invc \leq n} c {\mathbb F_q}$
and its base change $\overline Z := Z \times_{\mathbb F_q} \overline{\mathbb
F}_q$.
We assume this parameterizes covers with invariant $n$ and $n$ is sufficiently
large.
To complete the proof,
it suffices to show the
points of
$Z(\mathbb F_q)$ 
corresponding to disconnected $G$ extensions over $\mathbb F_q$ form a
proportion at most $1 - \frac{|N|}{2|G|}$
of $Z(\mathbb F_q)$.

Now, consider the set of components $Z_1^{N'}, \ldots,
Z_t^{N'}$ of 
$\cgquohur N {N'} {\invc \leq n} c {\mathbb F_q}$
over $Z$ for each $N \subset N' \subset G$.
In total, the map 
$\cgquohur N {N'} {\invc \leq n} c {\mathbb F_q}
\to \cgquohur N G {\invc \leq n} c {\mathbb F_q}$
has degree $|G|/|N'|$, 
and irreducibility of $Z$ implies the action of $G/N$ permutes these
components 
$Z_1^{N'}, \ldots,
Z_t^{N'}$
transitively.
Note that by
\autoref{notation:connected-quotient-hur},
we can assume there is no $N'$ so that 
$\cgquohur N {N'} {\invc \leq n} c {\mathbb F_q}$
contains $|G|/|N'|$ geometrically irreducible components over $Z$. That is, $t
< |G|/|N'|$.

Therefore, for each $N' \subset G$, we may now assume 
the number of geometrically irreducible components of 
$\cgquohur N {N'} {\invc \leq n} c {\mathbb F_q}$
over $Z$
is strictly less than
$|G|/|N'|$.
We wish to bound the total number of $\mathbb F_q$ points in the image of maps
$Z^{N'}_i(\mathbb F_q) \to Z(\mathbb F_q)$, for $Z^{N'}_i$ some geometrically irreducible component.
If $Z^{N'}_i$ is not stabilized by the $G/N'$ action,
$Z^{N'}_i(\mathbb F_q) \to Z(\mathbb F_q)$ will all factor through
$Z^{N''}_j(\mathbb F_q)$, where
$Z^{N''}_j \subset
\cgquohur N {N''} {\invc \leq n} c {\mathbb F_q}$
is some component with $N' \subset N'' \subset G$.
Hence, it suffices to count $\mathbb F_q$ points from geometrically irreducible
components corresponding to $N' \subset N$ which possess a unique geometrically
irreducible component $Z^{N'}$ of 
$\cgquohur N {N'} {\invc \leq n} c {\mathbb F_q}$
over $Z$.

Due to the stacky way in which we are counting points,
the image of the points
$Z^{N'}(\mathbb F_q) \to Z(\mathbb F_q)$
account for $\frac{|N'|}{|G|} \cdot |Z_1^{N'}(\mathbb F_q)|$
many points of 
$Z(\mathbb F_q)$.
Recall that the final part of
\autoref{lemma:weak-point-bound} implies that $Z^{N'}(\mathbb F_q)$ is
arbitrarily well approximated by $q^{\dim Z}$ for $q$ sufficiently large.
Using this, a simple inclusion exclusion then shows that for any $\varepsilon > 0$, the points in the image
of
$\cup_{N'} Z^{N'}(\mathbb F_q) \to Z(\mathbb F_q)$
(where the union is taken over those $N' \subset G$ containing $N$ such that
there is a unique geometrically irreducible component $Z^{N'}$ in
$\cquohur N {N'} n c {\mathbb F_q}(\mathbb F_q)$,
over $Z$)
account for at most
$1-\frac{\phi(|G|/|N|)}{|G|/|N|}+ \varepsilon$ of the points of
$\cgquohur N {G} n c {\mathbb F_q}(\mathbb F_q)$.
Here, we may have to increase the value of the constant $C$ which is our lower
bound for $q$ in order to make this true.
Taking $\varepsilon < \frac{|G|}{2|N|}$, we can make
$0 < 1- \frac{\phi(|G|/|N|)}{|G|/|N|}+ \varepsilon <  1-\frac{|G|}{2|N|}$, as
claimed.
\end{proof}

Finally, we conclude the proof of 
\autoref{theorem:turkelli} by combining our upper and lower bounds above.

\subsubsection{Proof of \autoref{theorem:turkelli}}
\label{subsubsection:turkelli-proof}

To prove \eqref{equation:all-malle-count}, we wish to count connected $G$ covers of $\mathbb F_q(t)$ of
	invariant at most $n$ with
	inertia in $c$. Each such cover geometrically becomes a disjoint union
	of $|G|/|N|$ Galois $N$ covers over $\overline{\mathbb F}_q(t)$
	where $N \subset G$ is normal with cyclic quotient, using structure of
	the absolute Galois group of $\mathbb F_q(t)$. Thus, we can separately
	count such extensions for each such $N \subset G$. That is, it only
	remains to prove \eqref{equation:sub-malle-count}.

Fixing $N \subset G$ a normal subgroup with cyclic quotient, the upper bound on the number of such  extensions then follows from
	\autoref{lemma:component-count-malle-upper}, since every such $G$
	extensions corresponds to a point on 
	some geometrically irreducible component of $\cgquohur N G {\invc \leq
	n} {c} {\mathbb F_q}$ using
	\autoref{lemma:connected-cover-observation}.

	We finally deduce the lower bound by combining
	\autoref{lemma:component-count-malle-lower},
	\autoref{lemma:connected-cover-observation},
	and
	\autoref{lemma:n-extensions-constant}.
	Indeed, \autoref{lemma:component-count-malle-lower} shows $|\cgquohur N G {\invc \leq n} {c} {\mathbb
	F_q}(\mathbb F_q)|$
	is bounded below by our desired lower bound, up to a constant.
	However, not all points of $\cgquohur N G {\invc \leq n} {c} {\mathbb
	F_q}(\mathbb F_q)$
	correspond to connected $G$ covers. 
	To this end, \autoref{lemma:connected-cover-observation}
	shows that every connected $G$-cover of invariant at most $n$
	with inertia in $c$
	corresponds to some $\mathbb F_q$
	points of $\cgquohur N G {\invc \leq n} {c} {\mathbb
	F_q}$
	and
	\autoref{lemma:n-extensions-constant}
	shows that at least a proportion $|N|/2|G|$ of these covers
	do in fact correspond to connected $G$ covers.
	Hence, we obtain our desired lower bound.
	\qed

\section{Further questions and conjectures}
\label{section:further-questions}
In this section, we collect various additional questions and conjectures.
We raise questions relating to Malle's conjecture in
\autoref{subsection:malle-questions}, relating to Gerth's conjecture in 
\autoref{subsection:gerth}, 
relating to the Picard rank conjecture in
\autoref{subsection:integral-picard-rank},
and questions relating to homological stability in
\autoref{subsection:stability-questions}.
We conclude with an interesting conjecture related to higher genus Hurwitz
spaces in
\autoref{subsection:higher-genus-conjecture}.

\subsection{Malle's conjecture}
\label{subsection:malle-questions}

Recall that Malle's original conjecture predicts the number of $G$ field extensions of
$\mathbb Q$ which are Galois of discriminant at most $X$ to be of the form 
$C(\mathbb Q, G) X^{a(G-\id,\Delta)} (\log X)^{b_M(\mathbb Q, G, (G-\id)_\Delta)}$, for
some constant $C(\mathbb Q, G)$ depending on $G$ and 
$a(G-\id,\Delta), b_M(\mathbb Q, G, (G-\id)_\Delta)$ having analogous
definitions to those given in
\autoref{notation:malle}, where $\Delta$ is the discriminant invariant of
\autoref{example:discriminant}.
When working over function fields, as explained in
\autoref{remark:malle-periodic}, this asymptotic has no chance of holding.
Nevertheless, 
one may conjecture there exists a collection of constants which are ``periodic''
as in \autoref{theorem:malle-g-connected}, and we refer to such a collection of
constants as the {\em constant in Malle's conjecture for $G$}
when working over $\mathbb F_q(t)$.

The recent preprint
\cite{loughranS:malles-conjecture-and-brauer-groups} proposes a
conjecture for the constant in Malle's conjecture for counting
the number of $G$ extensions of $\mathbb Q$
which intersect $\mathbb Q(\mu_{\exp(G)})$ trivially, for $\exp(g)$ the least
common multiple of the orders of all $g \in G$.
This appears to be a very reasonable condition to impose from the
function field perspective as it corresponds to counting geometrically
connected $G$ covers.
We begin by posing a question to which the answer is surely "yes," but we state
it as a question to encourage someone to do it!
\begin{question}
	\label{question:}
	Can one generalize the conjecture 
	\cite[Conjecture 9.3]{loughranS:malles-conjecture-and-brauer-groups}
	to counting $G$ extensions of $\mathbb F_q(t)$? 
\end{question}
\begin{question}
	\label{question:}
	Can one use the ideas of
	appearing in \autoref{section:malle} to compute what the constant should
	be in Malle's conjecture for counting extensions of $\mathbb F_q(t)$?
	Can one compute what the constant should
	be in Malle's conjecture for counting extensions of $\mathbb F_q(t)$
	which are geometrically connected?
	Does this agree with (a suitable generalization of) the predictions of \cite{loughranS:malles-conjecture-and-brauer-groups}?
\end{question}

\begin{remark}
	\label{remark:bhargava-conj}
	One well known example where the constant in Malle's 
	conjecture for counting by discriminant over $\mathbb Q$ has been predicted
	is the case that $G$ is the symmetric group $S_d$ and $c = G - \id$.
	This is the subject of Bhargava's conjecture
	\cite[Conjecture
	1.2]{bhargava:mass-formulae-for-extensions-of-local-fields}.
	Although it should be possible to compute the constant 
	when counting by reduced discriminant over $\mathbb F_q(t)$ by making the constants in \autoref{section:malle}
	explicit, the tools developed in this paper are not yet sufficient to
	compute the constant when counting by discriminant. 
	The reason for this is that, when counting $S_d$ extensions over
	$\mathbb F_q(t)$ by discriminant, there is a
	contribution to the constant coming from extensions parameterizing 
	covers with a single branch point whose inertia is a $3$-cycle, and the
	remaining inertial elements are transpositions. In order to obtain the exact constant
	when counting by discriminant, one would therefore want to compute the stable homology
	when one has a single $3$-cycle and many transpositions. However,  
	\autoref{theorem:all-large-stable-homology} does not cover this case.
	We are currently working on computing the stable cohomology in this
	sense.
\end{remark}

\begin{remark}
	\label{remark:}
	One can ask for even more than computing the stable cohomology in specific directions,
	as described in 
	\autoref{remark:bhargava-conj}.
	Namely, one can fix a normal subgroup $H \subset G$, fix a $G/H$
	extension $K/\mathbb F_q(t)$ and ask to count $G$ extensions $L/\mathbb
	F_q(t)$ containing $K$.
	This is closely related to \cite[Conjecture
	9.6]{loughranS:malles-conjecture-and-brauer-groups}
	and also related to understanding the Poonen-Rains heuristics for
	quadratic twist families of elliptic curves.
	This is also related to questions we are currently working on.
\end{remark}

\autoref{remark:bhargava-conj} leads to the following
questions:

\begin{question}
	\label{question:}
	Can one compute the constant in Malle's conjecture for counting
	$S_d$ extensions of $\mathbb F_q(t)$ by reduced discriminant?
	Can one similarly compute the constant Malle's conjecture for other
	groups when counting by reduced discriminant?
\end{question}
The next question seems interesting and potentially quite approachable via similar techniques to
those in this paper, but we have not pursued it.
\begin{question}
	\label{question:extend-constant}
	Can one show there is a (suitably periodic) constant in Malle's
	conjecture over $\mathbb F_q(t)$ when counting by an arbitrary
	invariant, by proving a generalization of
	\autoref{theorem:malle-g-connected}?
	If so, can one compute the constant in Malle's conjecture over $\mathbb
	F_q(t)$
	when counting by an arbitrary invariant?
	Can one verify such constants exist when one counts all extensions,
	instead of just geometrically connected such extensions?
\end{question}

\subsection{Gerth's conjecture}
\label{subsection:gerth}

Recall that the original Cohen-Lenstra heuristics predict the distribution of
the odd part of the class group.
In 
\cite[Conjecture
(C14')]{gerth:densities-for-the-ranks-of-certain-p-class-groups}
and
\cite{gerth:extension-of-conjectures}
Gerth proposed
a generalization of these conjectures, which, among other things,
aims to understand the even part of the class
group.
There has been substantial progress toward these conjectures in
\cite[Theorem 1.9 and 1.12]{smith:the-distribution-of-selmer-groups-1}
and also
\cite[Theorem 1.1]{koymansP:on-the-distribution-l-cyclic-fields}.
One way of characterizing the essence of Gerth's conjecture (although this is
not literally stated by Gerth) is that even though
the average size
$\on{Cl}(\mathscr O_K)[2]$ for $\mathscr O_K$ varying over imaginary quadratic
fields is infinite, 
if $n$ is an integer,
$\on{Cl}(\mathscr O_K)[n]/\on{Cl}(\mathscr O_K)[2]$ should conjecturally be
finite.

\begin{question}
	\label{question:}
	Can one compute the moments associated to
$\on{Cl}(\mathscr O_K)[n]/\on{Cl}(\mathscr O_K)[2]$
as one varies over quadratic fields $K/\mathbb F_q(t)$ ramified over $\infty$, in the function field
setting?
Is there some way to phrase Gerth's conjecture in terms of counting
	points on Hurwitz spaces?
\end{question}

One proposal for how one could try to make sense of this is given in
\cite{liu:on-the-distribution-of-class-groups-of-abelian-extensions}.

\begin{remark}
	\label{remark:}
	It seems plausible that one could carry out a similar procedure to that
	in \autoref{section:clm} to count the total number of elements of a
	fixed even order in class groups of bounded discriminant.
	While this wouldn't exactly yield Gerth's conjectures, we still believe
	it would be quite interesting to carry out.
\end{remark}

\subsection{An integral asymptotic Picard rank conjecture}
\label{subsection:integral-picard-rank}

We now mention some questions relating to an integral version of the Picard rank
conjecture.

\begin{remark}
	\label{remark:}
	We proved that the rational Picard group of the Hurwitz space
	$\churp {G} n c$
	of $G$ covers of $\mathbb P^1$ with inertia in a conjugacy class $c$
	stabilizes to $0$ as $n$ grows in \autoref{theorem:stable-picard}.
	We show more generally that the localization of the Picard group of each
	component at $\mathbb Z[\frac{1}{2|G|}]$ is $\mathbb Z/(2n-2) \mathbb Z \otimes \mathbb
	Z[\frac 1 {2|G|}]$.
	In particular, in contrast to the situation over $\mathbb A^1$, the integral Picard groups of covers of $\mathbb P^1$ do
	not stabilize in general.
\end{remark}

Note that
$H^2(\on{Conf}_{\mathbb P^1,n}, \mathbb Z) \simeq
\mathbb Z/(2n-2)\mathbb Z$, as is shown, for example, in \cite[Theorem
1.3]{schiessl:integral-cohomology-of-configuration-spaces-of-the-sphere}.
This leads us to ask whether the failure of stabilization of the Picard
group
is fully accounted for
by the fact that 
$H^2(\on{Conf}_{\mathbb P^1,n}, \mathbb Z)$
fails to stabilize in $n$.
\begin{question}
	\label{question:integral-picard-stabilization}
	There is a map $\pic(\on{Conf}_{\mathbb P^1, n} \times BS_d) \to \pic(\churp {S_d} n
	c )$,
	for $n$ even,
	where $c \subset S_d$ is the conjugacy class of transpositions.
Does $\pic(\churp {S_d} n c )/ \pic(\on{Conf}_{\mathbb P^1,
n} \times BS_d)$ stabilize to a fixed finite group, depending only on $d$, as $n \to
\infty$, over even numbers?

More generally, for $G$ an arbitrary finite group, $c$ a conjugacy class, 
and $Z \subset \churp {G} n c$ a component,
we ask whether
$\pic(Z)/ \pic(\on{Conf}_{\mathbb P^1, n} \times BS_d)$
stabilizes as $n$ ranges over numbers dividing the order of the image of $c$ in
$G^{\ab}$, and if this stable value is independent of the choice of $Z$.

We also ask whether
$\pic(\churz {S_d} n c )/ \pic([(\on{Conf}_{\mathbb P^1, n} \times BS_d)/\pgl_2])$
stabilizes.
\end{question}

\subsection{Homological stability}
\label{subsection:stability-questions}

We now mention a few interesting directions in which one may attempt to improve
our results.

\begin{question}
	\label{question:completelyintegral}
	Can one prove a version of \autoref{theorem:all-large-stable-homology}
	with $\mathbb Z$ coefficients, instead of only $\mathbb Z[1/G^0_c]$
	coefficients?
\end{question}
\begin{remark}
	\label{remark:}
The ideas in 	
\cite{bianchi:deloopings-of-hurwitz-spaces}
may be helpful in computing the integral dominant stable homology for $c$ a
finite rack.
We believe it would be interesting to work this out in detail. We note that the results of \cite{fennRS:the-rack-space} imply that understanding the group completion of Hurwitz spaces for racks can be reduced to the case of quandles, which is the case studied by Bianchi.
\end{remark}

\begin{question}
	\label{question:optimalline}
	Can one improve the proof of
	\autoref{theorem:some-large-homology-stabilizes}
	to yield smaller constants
	$I$ and $J$ than those obtained by making the proof of
	\autoref{theorem:some-large-homology-stabilizes}
	explicit?
\end{question}
In particular, if one were able to improve the constants, one could prove more
cases of the Picard rank conjecture \autoref{conjecture:picard-rank}. It would be quite interesting if it were
possible to sufficiently improve it so as to prove all (or even almost all) of the
remaining cases of the Picard rank conjecture.

Given that higher order stability phenomena are present in many examples where homological stability is, such as in the homology of multicolored configuration spaces and the homology of the moduli of curves \cite{galatius20192}, we ask if this holds for Hurwitz spaces.

\begin{question}
	\label{question:higherorderstability}
	Do there exist higher order stability phenomena for the homology of connected Hurwitz spaces?
\end{question}

Finally, we ask for a generalization to higher genus base curves, which we are
currently thinking about.

\begin{question}
	\label{question:highergenus}
	Can one prove a version of our results over higher genus punctured
	curves in place of $\mathbb A^1_{\mathbb C}$?
\end{question}

\subsection{A conjecture on stable homology for higher genus bases}
\label{subsection:higher-genus-conjecture}

In \autoref{theorem:some-large-homology-stabilizes}
and \autoref{theorem:all-large-stable-homology},
we show that the homology of Hurwitz spaces associated to unramified covers of punctured genus $0$ curves
stabilize, in a suitable sense, and we compute the dominant stable homology.
Letting $\Sigma_{g,n}$ denote a genus $g$ surface with $n$ punctures.
We can view this as a homological stability result over
$\Sigma_{0,n}$ as $n$ grows. It is natural to ask if a similar homological
stability phenomenon can hold as $g$ grows. It is also natural
to try to allow both $n$ and $g$ to grow, but for simplicity, we will now fix $n
= 0$ and consider the case that only $g$ grows.

To make a precise statement with $g$ growing, we first define the relevant
Hurwitz stack.
Let $g \geq 2$ be an integer and $G$ be a finite group.
Let $\on{CHur}_{g}^{G}$ denote the algebraic stack over $\mathbb C$
whose $S$ points are given by
smooth proper curves $C \to S$ of genus $g$ with geometrically connected fibers
together with a finite geometrically connected \'etale Galois $G$ cover $X \to C$, up to
isomorphism of covers over $S$. This stack is a gerbe over a finite \'etale
cover of $\mathscr M_g$, the moduli stack of curves of genus $g$.
For example, this can be constructed as an open substack of the stack of twisted $G$-covers from
\cite[\S2.2]{abramovichCV:twisted-bundles}, which parameterizes covers of smooth
curves.
In the case of $\Sigma_{0,n}$ our result in
\autoref{theorem:all-large-stable-homology} roughly says that when we increase
$n$, the homology of the Hurwitz space agrees with the stable homology of some
version of $\mathscr M_{0,n}$. This leads us to the following conjecture:

\begin{conjecture}
	\label{conjecture:genus-hurwitz}
	Fix a finite group $G$.
	There are constants $A$ and $B>0$, depending on $G$,
	so that for 
	any connected component component $Z \subset \on{CHur}_{g}^{G}$
	and
	$g > A + Bi$, the map $H_i(Z, \mathbb Z[\frac 1 {|G|}]) \to
	H_i(\mathscr M_g, \mathbb Z[\frac 1 {|G|}]) $
	is an isomorphism.
\end{conjecture}

In other words, we conjecture that the homology of each component of Hurwitz space agrees
with that of $\mathscr M_g$, and this stability occurs in a linear range.

\begin{remark}
	\label{remark:}
	A version of
	\autoref{theorem:all-large-stable-homology} has been proven in the large
	$g$ limit by Putman in
	\cite[Theorem C]{putman:partial-torelli}.
	Namely, if one considers a variant of 
	$\on{CHur}_{g}^{G}$ which parameterizes $G$ covers of genus $g$ curves
	with a boundary component (instead of with no boundary component)
	then 
	\cite[Theorem C]{putman:partial-torelli}
	shows the integral homology of these spaces stabilize in $g$.
	However, Putman does not determine the stable value of these homology
	groups.
\end{remark}

\begin{remark}
	\label{remark:}
	It is known that the $0$th homology of 
	$\on{CHur}_{g}^{G}$ stabilizes
	to $H_2(G, \mathbb Z)/\on{Out}(G)$
	by \cite[Theorem 6.20]{dunfield-thurston}.
	Therefore, the conjecture above, in conjunction with the fact that the
	homology of $\mathscr M_g$ stabilizes would imply that the homology
	of $\on{CHur}_g^G$ stabilizes.
\end{remark}

\begin{remark}
	\label{remark:}
	Very few cases of \autoref{conjecture:genus-hurwitz} are known.
	The case $i = 0$ holds tautologically.
	The case that $G \simeq \mathbb Z/\ell \mathbb Z$ is almost established in
	\cite[Theorem A]{putman:the-stable-cohomology-with-level-structure},
	which builds on several other of Putman's recent papers, spanning
	several hundred pages.
	Even this work does not quite establish
	\autoref{conjecture:genus-hurwitz} for two reasons. First, the isomorphism is only
	established with rational coefficients. Second, the stability range is
	only shown to be quadratic in the homology index $i$ instead of our conjectured linear
	stability range.

	We are not aware of any other known cases. We will mention that the case
	of $i = 1$ is closely related to Ivanov's conjecture,
	see
	\cite[\S7]{Ivanov:problems} and \cite[Problem 2.11.A]{Kirby:problems}.
	The $i=1$ case is also related to 
	the Putman-Wieland conjecture \cite[Conjecture
	1.2]{putmanW:abelian-quotients}, although the Putman-Wieland conjecture
	only asserts that a
	certain subspace of the first rational cohomology stabilizes to $0 = H^1(\mathscr
	M_g, \mathbb Q)$, not that 
	$H^1(\on{CHur}_{g}^{G}, \mathbb Q)$
	stabilizes to $0$.
	It was established that this above mentioned subspace of $H^1$ does indeed stabilize to $0$ in
	\cite[Theorem 1.4.1]{landesmanL:canonical-representations}.
	It is also possible to deduce this from
	\cite[Theorem 1.1(i)]{looijenga:arithmetic-representations}, using an
	argument similar to that in \cite[Corollary 6.17]{dunfield-thurston}.
\end{remark}

\begin{remark}
	\label{remark:}
	Above, we made a conjecture for the homology of spaces of maps from
	curves of growing genus to $BG$, where $G$ is a finite group. It would be
	interesting if one could extend this to the case that $G$ is an
	algebraic group. In this case, such a conjecture would be closely
	related to important results in the burgeoning field of mapping class
	group actions on character varieties.
\end{remark}

\bibliographystyle{alpha}
\bibliography{../bibliography}

\end{document}